\documentclass[12pt]{article}

\usepackage{graphicx}
\usepackage{amsmath,amsthm,amssymb,enumerate}%, esint}
\usepackage{euscript,mathrsfs}
\usepackage{color}
\usepackage{dsfont}
\usepackage[citecolor=blue,colorlinks=true]{hyperref}
\usepackage[left=2cm,right=2cm,top=3.5cm,bottom=3.5cm]{geometry}
\usepackage{color}
\usepackage[framemethod=tikz]{mdframed}
\usepackage{multirow}
\usepackage{subcaption}

\usepackage{soul}

\catcode`\@=11 \@addtoreset{equation}{section}

\catcode`\@=12

\allowdisplaybreaks

\newtheorem{Theorem}{Theorem}[section]
\newtheorem{Proposition}[Theorem]{Proposition}
\newtheorem{Lemma}[Theorem]{Lemma}
\newtheorem{Corollary}[Theorem]{Corollary}

\theoremstyle{Definition}
\newtheorem{Definition}[Theorem]{Definition}

\newtheorem{Assumption}[Theorem]{Assumption}
\newtheorem{Example}[Theorem]{Example}
\newtheorem{Remark}[Theorem]{Remark}

\newcommand{\bTheorem}[1]{
\begin{Theorem} \label{T#1} }
\newcommand{\eT}{\end{Theorem}}

\newcommand{\bProposition}[1]{
\begin{Proposition} \label{P#1}}
\newcommand{\eP}{\end{Proposition}}

\newcommand{\bLemma}[1]{
\begin{Lemma} \label{L#1} }
\newcommand{\eL}{\end{Lemma}}

\newcommand{\bCorollary}[1]{
\begin{Corollary} \label{C#1} }
\newcommand{\eC}{\end{Corollary}}

\newcommand{\bRemark}[1]{
\begin{Remark} \label{R#1} }
\newcommand{\eR}{\end{Remark}}

\newcommand{\bDefinition}[1]{
\begin{Definition} \label{D#1} }
\newcommand{\eD}{\end{Definition}}

\newcommand{\bFormula}[1]{
	\begin{equation} \label{#1}}
\newcommand{\eF}{\end{equation}}

\newcommand{\vc}[1]{{\bf #1}}

\newcommand{\Div}{{\rm div}_{ x}}
\newcommand{\Grad}{\nabla_{ x}}
\newcommand{\vr}{\varrho}
\newcommand{\vt}{\vartheta}

\newcommand{\defjump}[1]{\left[\left[ #1 \right]\right]}
\def\inerface{\Sigma_{\tt int}}
\newcommand{\projection}[1]{\Pi_h[ #1 ]}

\newcommand{\norm}[1]{ \lVert #1 \rVert }
\newcommand{\abs}[1]{\left| #1\right|}

\newcommand{\tvS}{{\widetilde{\eta}}}

\newcommand{\tvp}{{\widetilde{p}}}
\newcommand{\tve}{{\widetilde{e}}}

\newcommand{\tvU}{\widetilde{\vU}}

\newcommand{\dt}{\,{\rm d} t }
\newcommand{\dx}{\,{\rm d} { x}}
\newcommand{\dxdt}{\dx  \dt}

\newcommand{\pd}{\partial}

\newcommand{\E}{\mathbb{E}}

\newcommand{\jump}[1]{\left[ \left[ #1 \right] \right]}

\newcommand{\vrh}{\vr_h}

\newcommand{\vmh}{\vm_h}

\newcommand{\tvm}{\widetilde{\vc{m}}}
\newcommand{\tS}{\widetilde{S}}
\newcommand{\bfphi}{\boldsymbol{\phi}}

\newcommand{\dsx}{\mathrm{d}S_x}

\newcommand{\vuh}{\vu_h}
\newcommand{\mh}{\vc{m}_h}

\newcommand{\Ov}[1]{\overline{#1}}

\newcommand{\vF}{\vc{F}}
\newcommand{\aleq}{\stackrel{<}{\sim}}

\newcommand{\tvr}{\widetilde \vr}
\newcommand{\tvu}{{\widetilde \vu}}
\newcommand{\tvt}{\widetilde \vt}

\newcommand{\vu}{\vc{u}}
\newcommand{\vm}{\vc{m}}

\newcommand{\vq}{\vc{q}}
\newcommand{\vn}{\vc{n}}

\newcommand{\vU}{\vc{U}}

\newcommand{\intO}[1]{\int_{\Omega} #1 \ \dx}

\newcommand{\intOB}[1]{\int_{\Omega} \left( #1 \right) \ \dx}

\newcommand{\I}{\mathbb{I}}

\newcommand{\DD}{{\rm d}}

\def\softd{{\leavevmode\setbox1=\hbox{d}%
          \hbox to 1.05\wd1{d\kern-0.4ex{\char039}\hss}}}
\definecolor{Cgrey}{rgb}{0.85,0.85,0.85}
\definecolor{Cblue}{rgb}{0.50,0.85,0.85}
\definecolor{Cred}{rgb}{1,0,0}
\definecolor{fancy}{rgb}{0.10,0.85,0.10}

\begin{document}

%%%%%%%%%%%%%%%%%%%%%%%%%%%%%%%%

\title{Error estimates of the Godunov method for the multidimensional compressible Euler system}

\author{M\' aria Luk\' a\v cov\' a -- Medvi\softd ov\' a%\thanks{M.L. has been funded by the Deutsche Forschungsgemeinschaft (DFG, German Research Foundation) - Project number 233630050 - TRR 146 as well as by  TRR 165 Waves to Weather. She is grateful to the Gutenberg Research College for supporting her research. %\newline \hspace*{1em} 
%The research of B.S. leading to these results has received funding from the Czech Sciences Foundation (GA\v CR), Grant Agreement 21-02411S. The Institute of Mathematics of the Academy of Sciences of the Czech Republic is supported by RVO:67985840.\newline \hspace*{1em} 
%The research of Y. Y. was funded by Sino-German (CSC-DAAD) Postdoc Scholarship Program in 2020 - Project number 57531629.}
\and Bangwei She%\thanks{%$^\spadesuit$	The research of B.S. leading to these results has received funding from the Czech Sciences Foundation (GA\v CR), Grant Agreement 21-02411S. The Institute of Mathematics of the Academy of Sciences of the Czech Republic is supported by RVO:67985840.}
\and Yuhuan Yuan%\thanks{%$^\clubsuit$	The research of Y. Y. was funded by Sino-German (CSC-DAAD) Postdoc Scholarship Program in 2020 - Project number 57531629.}
}

\maketitle

\bigskip

\bigskip
\centerline{$^{*,\ddag}$Institute of Mathematics, Johannes Gutenberg-University Mainz}
\centerline{Staudingerweg 9, 55 128 Mainz, Germany}
\centerline{lukacova@uni-mainz.de, yuhuyuan@uni-mainz.de}

\bigskip

\centerline{$^{\dag}$Academy for Multidisciplinary studies, Capital Normal University}
\centerline{ West 3rd Ring North Road 105, 100048 Beijing, P. R. China}
\centerline{and}
\centerline{%$^{\dag}$
Institute of Mathematics of the Czech Academy of Sciences}
\centerline{\v Zitn\' a 25, CZ-115 67 Praha 1, Czech Republic}
\centerline{she@math.cas.cz}

\date

\begin{abstract}
We derive a priori error of the Godunov method for the multidimensional Euler system of gas dynamics.  
To this end we apply the relative energy principle and estimate the distance between the numerical solution and the strong solution. 
This yields also the estimates of the $L^2$-norm of  errors in density, momentum  and entropy. 
Under the assumption that the numerical density and energy are bounded, we obtain a convergence rate of $1/2$ for the relative energy in the $L^1$-norm. 
Further, under the assumption -- the total variation of numerical solution is bounded, we obtain the first order convergence rate for the relative energy in the $L^1$-norm. 
Consequently, numerical solutions (density, momentum and entropy) converge in the $L^2$-norm with the convergence rate of $1/2$. %, which is valid for rarefaction waves. 
The numerical results presented for Riemann problems are consistent with our theoretical analysis.

\end{abstract}


\tableofcontents

%%%%%%%%%%%%%%%%%%%%%%%%%%%%%%%%%%%%%%%%%%%%%%%%%%%%%%%%%%%%%%%%%%%%%

%\section{Weak and strong solutions to the Euler system}
%\label{WSES}
\section{Introduction}\label{Introduction}
We consider the Euler system governing the motion of a compressible gas
\begin{equation}\label{pde}
 \partial_t \vU + \Div \vF(\vU) = 0, \quad (t,x) \in (0,T) \times \Omega. 
\end{equation}
Here $\Omega \subset \mathbb{R}^d \, (d=1,2,3)$ is a bounded computational domain, $\vU=(\vr, \vm, E)^T$ represents the fluid density, momentum and total energy, while $\vF$ is the flux function given by 
\[\vF= (\vm, \vu \otimes \vm+ p \I , \vu (E+p))^T. 
\]
Here for positive $\vr$, $\vu =  \frac{\vm}{\vr}$ is the velocity of the fluid and  $p$ is the pressure satisfying the  state equation of perfect gas  
\begin{equation}\label{EOS}
p = (\gamma-1)\vr e, \quad \gamma \in (1,2]
\end{equation}
with the specific internal energy $e= \frac{E}{\vr}  -\frac{1}{2}|\vu|^2 $. 

We close the system with initial data 
\begin{subequations}\label{ini}
\begin{equation}\label{ini-1}
\vU(0,x) =\vU_0 = (\vr_0, \vm_0:=\vr_0 \vu_0, E_0)
\end{equation}
satisfying 
\begin{equation}\label{ini-2}
 \vr_0 > 0 \quad \mbox{and} \quad E_0 \in L^1(\Omega)
\end{equation}
\end{subequations}
 and impermeability boundary condition
\begin{equation}\label{I5}
\vu \cdot \vc{n}|_{\partial \Omega} = 0,
\end{equation}
where $\vn$ is the outer normal vector on the boundary $\partial \Omega$. 
%Alternatively, one may consider the periodic boundary conditions. 
Taking the Second law of Thermodynamics into account we further require that the entropy inequality holds, i.e.
\begin{equation}\label{I6}
\partial_t  \, \eta(\vU ) + \Div  \,\vq(\vU) \geq 0.
\end{equation}
Here $(\eta, \vq)$ is the physical entropy pair given by
\begin{equation}\label{I7}
\eta = C_v \vr S , \quad \vq = \eta \vu \quad \mbox{with } C_v =  \frac{1}{\gamma -1} \mbox{ and } S= \ln\left(\frac{p}{\vr^\gamma}\right) . 
\end{equation}

During the past few decades numerical simulation of the Euler system has been a hot topic in the field of computational mechanics and physics, cf.~Toro~\cite{Toro}, Feistauer et al.~\cite{Feistauer-Felcman-Straskraba:2003} , Li et al.~\cite{Li-Zhang-Yang:1998}, LeVeque~\cite{Leveque:1992}.
Despite the success in practical simulations, a rigorous convergence analysis of the numerical methods still remains open in general.  
Most literature results were focused on scalar conservation laws. % see~\cite{Kuznetsov:1976,Tang-Teng:1995,Teng-Zhang:1997,Tang-Teng:1997,Tadmor-Tang:1999}. 
%Harten, Hyman and Lax pointed out that the monotone difference schemes are of at most first-order accuracy.  
Kuznetsov \cite{Kuznetsov:1976} showed that the (upper) $L^1$ error bound is $\mathcal{O}(h^{1/2})$ for multi-dimensional scalar conservation laws %as $h$ goes to zero 
under the assumptions on the boundedness of the total variation and continuity in time of numerical solutions, where $h$ is the mesh parameter. 
Further, Cockburn et al. \cite{Cockburn-Coquel-LeFloch:1994} and Vila \cite{Vila:1994}  extended the result of Kuznetsov  and obtained the $L^1$-error bounds of ${\cal O}(h^{1/4})$ without the assumptions of bounded total variation and continuity in time. 
The convergence rate of some specific waves was also studied in one dimension. 
Concerning the linear advection equation, Tang and Teng \cite{Tang-Teng:1995} showed  the sharpness of the $\mathcal{O}(\sqrt{\Delta x})$ $L^1$-error for  monotone difference schemes with BV initial data. 
For the nonlinear scalar equation %by means of Jennings traveling discrete shock waves, 
Teng and Zhang \cite{Teng-Zhang:1997} showed the optimal  convergence rate of $1$ in the $L^1$-norm for the viscosity method and monotone schemes if a solution is piecewise constant with finitely many shocks.  
Moreover, for the piecewise smooth entropy solution with finitely many rarefaction waves, Tang and Teng \cite{Tang-Teng:1997} showed that the error of viscosity solution to the inviscid solution is bounded by $\mathcal{O}(\varepsilon |\log \varepsilon| + \varepsilon)$ in the $L^1$-norm, where $\varepsilon$ denotes the viscosity coefficient.
% which is an improvement of the $\mathcal{O}(\sqrt{\varepsilon})$ upper bound. If neither central rarefaction waves nor spontaneous shocks occur, the error bound is improved to $\mathcal{O}(\varepsilon)$. 
%If the solution is a non-increasing piecewise smooth solution with finitely many shocks, the convergence rate of $L^1$ error is also $1$. 
Furthermore, Tadmor and Tang \cite{Tadmor-Tang:1999} studied the pointwise error estimates and showed that the thicknesses of the shock and rarefaction layers are of order $\mathcal{O}(\varepsilon)$ and $\mathcal{O}(\varepsilon \log^2 \varepsilon)$, respectively. 
We point out that the error estimates for scalar hyperbolic conservation laws are typically given in terms of the $L^1$-norm in space.

%{\cblue Check for scalar equation:  
%\begin{itemize}
%\item what is the convergence rate of $L^2$-error if the entropy solution only contains the rarefaction waves?
%\item do the above results also hold for multi-D?  
%	
%	References \cite{Kuznetsov:1976,Tang-Teng:1995,Teng-Zhang:1997,Tang-Teng:1997} work for 1D.
%\end{itemize}
%}

When considering the multidimensional nonlinear system of hyperbolic conservation laws, to our best knowledge, the only result was done by Jovanovi\'{c} and Rohde \cite{JoRo}, where the convergence rate of $1/2$  was presented  in terms of the $L^2$-errors between the numerical solutions and the classical solution $(\vU \in C^1)$  under the assumption of uniform boundedness of numerical solutions and their $H^1$ semi-norm.   
In this paper we estimate the error between the numerical solutions and the strong solution $(\vU \in W^{1,\infty})$  %via the relative energy which yields the $L^2$-error of density, momentum  and entropy, 
assuming that the total variation of the numerical solution is  bounded.  Comparing with \cite{JoRo}, we obtain the same convergence rate under a weaker assumption. % and a completely different proof. % but our analysis is also valid to the rarefaction waves. 
Moreover, without the assumption of bounded total variation, we still have the convergence rate of  $1/4$. 

The main tool used in the paper is the so-called relative energy functional originally introduced by Dafermos~\cite{Dafer}. 
This technique has been largely used in the analysis of the weak--strong uniqueness and singular limit of the compressible fluid flows, see the monograph of Feireisl and Novotn\'{y} \cite{FeNo}, B\v{r}ezina and Feireisl \cite{Brezina-Feireisl:2018a}, and Feireisl et al. \cite{FLMS,Feireisl-Lukacova-Necasova-Novotny-She:2018}. 
Recently, this technique has also been successfully applied to the convergence analysis of numerical solutions of compressible viscous fluids, see Feireisl et al. \cite{FHMN} and Mizerov{\'a} and She \cite{HS_MAC}. Here we adapt the technique to the Euler system and estimate the corresponding relative energy, which yields the control of the $L^2$-error of density, momentum and entropy, too.% while the \cite{JoRo} works with the relative entropy.  

The rest of the paper is organized as follows. 
In Section \ref{sec_pre} we introduce some preliminaries. 
More precisely, we recall the Godunov method and its consistency formulation proved in  Luk\'{a}\v{c}ov\'{a} and Yuan \cite{LMY}. 
We define the strong solution of the Euler system and the relative energy. 
Further, we prove the relative energy inequality in Section \ref{sec_ee} and estimate its error in the  $L^1$-norm. 
Finally, in Section \ref{sec_exp} we present some numerical experiments to validate theoretical results. %, also some interesting experiment to investigate in the future. 

%%%%%%%%%%%%%%%%%%%%%%
\section{Preliminaries}\label{sec_pre}
%In this section we introduce the preliminaries, including the entropy stability, strong solution and consistency error of the Godunov approximation. 
In this section we introduce the preliminaries, including the formulation of the Godunov method, its consistency formulation, and the definitions of the strong solution and relative energy. % and corresponding derivative products for the error estimates in Section~\ref{sec_ee}.

To begin, we define the following notations for the later use
\begin{align*}
&\bullet\quad a \lesssim b \quad \mbox{ if } \quad a \leq c b \mbox{ with a positive constant c},
\\
&\bullet\quad a \approx b \quad \mbox{ if } \quad a \lesssim b \mbox{ and  } b \lesssim a.
\end{align*}

\subsection{Godunov method}
The computational domain $\Omega$ consists of rectangular meshes $\overline{\Omega} := \bigcup_{K} \overline{K}$.
We denote the set of all mesh cells as $\mathcal{T}_h$ and the set of all interior faces of $\mathcal{T}_h$ as $\Sigma_{\rm int}$. 
We consider the space of piecewise constant functions
\begin{equation}
\mathcal{Q}_h(\Omega) = \{ v :  v|_{K^o} = \mbox{constant}, ~ \mbox{for all}~ K \in \mathcal{T}_h \}
\end{equation}
and define the projection operator
\begin{equation}
\Pi_h :  L^1(\Omega) \rightarrow \mathcal{Q}_h(\Omega), \quad \projection{\phi}_K = \frac{1}{|K|} \int_K \phi (x) ~dx,
\end{equation}
where $|K|$ is the Lebesgue measure of $K$.

Let $\vU_h \in  \mathcal{Q}_h(\Omega;\mathbb{R}^{d+2})$.
Then the semi-discrete form of the finite volume method with the Godunov flux, i.e. the Godunov method can be described as 
\begin{subequations}\label{eq:semi-discrete-RPflux-phi}
\begin{align} 
& \int_{\Omega} \phi  \frac{d }{d t} \vU_h ~ \dx -  \sum_{\sigma \in \inerface} \int_{\sigma} \vF(\vU^{\it RP}_{\sigma}) \cdot \vn \defjump{\phi} ~dS_{x} = 0, \\
& \vU_{h0} = \projection{\vU_0} .
\end{align}
\end{subequations}
Here $\phi \in \mathcal{Q}_h(\Omega)$ is the test function, $\vU^{\it RP}_{\sigma}$ is the exact solution of a local Riemann problem along the interface $\sigma$, and the notation $\defjump{ \cdot}$ denotes the jump along the interface.

\subsection{Consistency formulation}
We recall  the consistency formulation of the Godunov method derived by Luk\'{a}\v{c}ov\'{a} and Yuan \cite{LMY}.  
We start with the following assumption.

\begin{Assumption}\label{H1}
We assume that the solution to \eqref{eq:semi-discrete-RPflux-phi} satisfies
\begin{equation}\label{eqH1}
%\mbox{(H1)} \qquad \qquad 
0 < \underline{\vr} \leq \vrh , \quad
0 < E_h \leq \overline{E}  \quad \mbox{uniformly for } h \rightarrow 0
\end{equation}
for all $t\in[0,T]$, where $\underline{\vr}, \overline{E}$ are some positive constants. 
\end{Assumption}

\begin{Lemma}\label{Lemma1}{\footnote{The proof of Lemma~\ref{Lemma1} could be found in \cite{Feireisl-Lukacova-Mizerova:2020a}.}}
Under Assumption~\ref{H1}
there hold
\begin{align}
&0 < \underline{\vr} \leq \vr_h \leq \overline{\vr}, \quad
|\vu_h| \leq \overline{u},\quad
0 < \underline{p} \leq p_h \leq \overline{p}, ~ \\
& |\mh| \leq \overline{m},\quad
0 < \underline{E} \leq E_h \leq \overline{E}, \quad
0 < \underline{\vartheta} \leq \vartheta_h \leq \overline{\vartheta}
\end{align}
uniformly for $h \rightarrow 0, t\in[0,T]$ with positive constants $ \overline{\vr},\, \overline{u},\, \underline{p},\, \overline{p},\, \overline{m},\, \underline{E},\,\underline{\vartheta},\,\overline{\vartheta}$ depending on $\underline{\vr}, \overline{E}$, where $\vartheta := \frac{p}{\vr}$ is the absolute temperature.
\end{Lemma}
%The proof of Lemma~\ref{Lemma1} could be found in \cite{Feireisl-Lukacova-Mizerova:2020a}.

\begin{Theorem}(Consistency formulation) \label{Th1}
% Let  $\vU_{0,h} = \Pi_h \vU_0$ be the piecewise constant projection of the initial data $\vU_0$.  
 Let $(\vrh, \mh, \eta_h)$ be the numerical solutions obtained by the Godunov method \eqref{eq:semi-discrete-RPflux-phi} on the time interval $[0,T]$ satisfying Assumption \ref{H1}. 
Then for any $\tau \in (0,T)$ the following hold:
\begin{itemize}
		\item for all $\phi \in W^{1,\infty}((0,T)\times \Omega)${\footnote{Throughout this paper, we refer $f \in W^{1,\infty}$ to $f \in W^{1,\infty}\bigcap C^0$ .  }}
		\begin{equation}\label{CE1}
		\left[ \int_{\Omega} \vrh \phi ~\dx  \right]_{t = 0}^{t = \tau} = \int_{0}^{\tau} \int_{\Omega} \bigg( \vrh \partial_t \phi +   \mh \cdot \Grad  \phi  \bigg) \dxdt + \int_{0}^{\tau} e_{\vr,h}(t,\phi) \dt;
		\end{equation}
		
		\item for all $\bfphi \in W^{1,\infty}((0,T)\times \Omega; \mathbb{R}^d)$
	\begin{equation}\label{CE2}
		\begin{aligned}
		\left[ \int_{\Omega} \mh \cdot  \bfphi ~\dx  \right]_{t = 0}^{t = \tau} =
		&\int_{0}^{\tau} \int_{\Omega} \bigg( \mh \cdot \partial_t \bfphi +   \frac{\mh \otimes\mh }{\vrh} : \Grad  \bfphi \\
		& + p_h  \Div  \bfphi   \bigg)\dxdt + \int_{0}^{\tau} e_{\vm,h}(t,\bfphi) \dt;
		\end{aligned}
	\end{equation}
		
		\item for all $\phi \in W^{1,\infty}((0,T)\times \Omega),\, \phi \geq 0$
		\begin{equation}\label{CE3}
		\left[ \int_{\Omega} \eta_h \phi ~\dx  \right]_{t = 0}^{t = \tau} \geq \int_{0}^{\tau} \int_{\Omega} \bigg( \eta_h \partial_t \phi + \vq_h \cdot \Grad  \phi \bigg)\dxdt + \int_{0}^{\tau} e_{\eta,h}(t,\phi) \dt;
		\end{equation}
		
		\item
		\begin{equation}\label{CE4}
		\int_{\Omega} E_h(\tau) ~\dx = \int_{\Omega} E_{0,h} ~\dx
		\end{equation}
\end{itemize}
	
with	bounded errors $e_{j,h}, (j = \vr, \vm, \eta)$ satisfying
	\begin{equation}\label{CE5}
	\begin{aligned}
	\|e_{j,h}\|_{L^1(0,T)} & \lesssim 
	h	\norm{ \phi }_{W^{1,\infty}((0,T)\times \Omega)}  \int_{0}^{\tau} \sum_{\sigma \in \inerface} \int_{\sigma}  \abs{\jump{\vU_h}} ~ \dsx  \dt 
\\&
	\lesssim h^{1/2} \| \phi \|_{W^{1,\infty}((0,T)\times \Omega)} \left(\int_{0}^{\tau} \sum_{\sigma \in \inerface} \int_{\sigma}  \abs{\jump{\vU_h}}^2 ~ \dsx  \dt\right)^{1/2}.
	 \end{aligned}
	\end{equation}
\end{Theorem}

\subsection{Strong solution}
Our aim is to analyze the convergence rate of the Godunov method when approximating the strong solution of the Euler system \eqref{pde}--\eqref{I5}. 
%Here, we introduce the target solution  of the numerical approximation. In our case it is the strong solution given below. 
\begin{Definition} [Strong solution] \label{DefS} 
Let $\Omega \subset \mathbb{R}^d$ be a bounded domain with a boundary $\partial \Omega$ of class~$C^1$. 
We say that a trio $[\tvr, \tvu, \tvS]$ is the \emph{strong solution} of the Euler system \eqref{pde}--\eqref{I5}  if 
\begin{align*}
& \tvr \in %C([0,T] \times \Ov{\Omega}) \cap 
W^{1,\infty}((0,T) \times \Omega),\\
& \tvu \in %C([0,T] \times \Ov{\Omega}; R^d)  \cap 
W^{1,\infty}((0,T) \times \Omega; \mathbb{R}^d),\\
& \tvS \in %C([0,T] \times \Ov{\Omega}) \cap 
W^{1,\infty}((0,T) \times \Omega), \\
&  \tvr>0 \mbox{ and }  \vt(\tvr, \tvS) >0 \ \mbox{for any}\ (t,x) \in [0,T] \times \Ov{\Omega}
%&  \vr>0 \mbox{ and }  \vt(\vr, \eta) := \frac{p(\vr, \eta)}{\vr}>0 \ \mbox{for any}\ (t,x) \in [0,T] \times \Ov{\Omega},
%\\& \vu(t,x) \cdot \vc{n}(x) = 0 \ \mbox{for any}\ t \in [0,T],\ x \in \partial \Omega;
\end{align*}
and the equations \eqref{pde}--\eqref{I5} are satisfied for almost everywhere.
\end{Definition}

\noindent Let us point out that  we consider  $\vr$ and $\eta$ as the independent thermodynamical variables throughout the paper, meaning that all other thermodynamical variables are functions of $(\vr, \eta)$. Accordingly, we write   $\widetilde{v} = v(\tvr,  \tvS)$, $v\in\{p, e, \vt, S\}$, for the strong solution. Moreover, we denote $\tvm = \tvr \tvu$ and $\tvU =\vU(\tvr, \tvu, \tvS)$.

Since the domain is bounded and $(\tvr, \tvt)$ is continuous and positive, we have
%positive, hence we have 
\begin{equation}\label{H2}
0 < \underline{\vr} \leq \tvr,\quad 0 < \underline{\vt} \leq \tvt.
\end{equation}

%\begin{Remark}
%If  $(\vr, \vu, \eta)$ is a strong solution of the Euler system \eqref{pde}--\eqref{I5}, then it also satisfies the entropy equality, i.e. $\partial_t  \, \eta(\vU ) + \Div  \,\vq(\vU) = 0$. 
%\end{Remark} 

\begin{Remark}
According to the definition of strong solution, we know that an entropy solution only containing finitely many rarefaction waves is also a strong solution.
\end{Remark}

We recall \emph{Gibbs'} relation 
\begin{equation} \label{Gibbs}
\tvt \DD \tS = \DD \widetilde{e} + \widetilde{p} ~\DD \left( 1/\tvr \right) .
\end{equation}
%where $\vt=p/\vr$ represents the absolute temperature. 
Consequently,  for any strong solution $(\tvr, \tvu, \tvS)$ we obtain the following identities which will be used in Section~\ref{sec_ee}
\begin{equation}\label{EQS}
	\begin{aligned}
	& \partial_t \tvr + \tvu \cdot \Grad \tvr + \tvr \,\Div \tvu = 0, \\
	& \partial_t \tvu + \tvu \cdot \Grad \tvu + \frac1{\tvr}\, \Grad \widetilde{p} = 0, \\
	& \partial_t \tvS + \tvu \cdot \Grad \tvS +  \tvS\, \Div  \tvu = 0,\\
	& \partial_t \widetilde{p} + \tvu \cdot \Grad \widetilde{p} + \gamma  \widetilde{p} \, \Div \tvu = 0, \\
	& \partial_t \tS + \tvu \cdot \Grad \tS = 0, \\
	& \partial_t \tvt + \tvu \cdot \Grad \tvt + \left(\partial_{\tvS} \widetilde{p} \right)_{\tvr} \Div \tvu = 0,	
	\end{aligned}
\end{equation}
see \cite{FLMS,Toro} for more details.

%Under assumption -- lower bound on density and upper bound on energy, we show the equivalence between the relative energy and the $L^2$-error of numerical solution, see \eqref{REL2-1}. 
%Under some reasonable conjecture -- semi-norm of numerical solution is bounded, we obtain first-order convergence rate of relative energy, i.e.  half-order of the $L^2$-error.

\subsection{Relative energy}
In this part we introduce the relative energy % originally introduced by Dafermos~\cite{Dafer}, 
and  show the relationship between the relative energy and the $L^2$-error of $(\vr, \vm, \eta)$ for numerical solutions. 

Let $(\vr, \vm, \eta)$ and $(\tvr, \tvu, \tvS)$ be two vectors consisting of density, momentum and velocity, respectively, and total entropy. 
In the context of the compressible Euler system, the relative energy reads
\begin{equation} \label{RE}
{\E} \left( \vr, \vm, \eta \Big| \tvr, \tvu, \tvS \right) 
= ~ \frac{1}{2} \vr \left| \frac{\vm}{\vr} - \tvu \right|^2  + \vr e
- \frac{\pd  (\vr e)}{\pd \vr}\Big|_{(\tvr, \tvS)} (\vr - \tvr) -  \frac{\pd  (\vr e)}{\pd \eta}\Big|_{(\tvr, \tvS)}  (\eta - \tvS ) -  \tvr \tve
\end{equation}
for $\vr>0$.% with $\tve = e(\tvr, \tvS)$.  

\begin{Lemma}\label{LMEQ}
let $(\tvr,\tvu,\tvS)$ be the strong solution of the Euler system in the sense of Definition~\ref{DefS} and let $(\vrh,\mh,\eta_h)$ be a numerical solution of the Euler system obtained by \eqref{eq:semi-discrete-RPflux-phi} satisfying Assumption~\ref{H1}. 
Then we have the following equivalence 
\begin{equation} \label{REL2-1}
\E  \left( \vrh,  \mh, \eta_h \Big| \tvr, \tvu, \tvS \right) 
 \approx  |\vm_h - \tvm|^2 +  |\eta_h - \tvS|^2 + |\vrh - \tvr|^2.
\end{equation}
\end{Lemma} 
\begin{proof}
First, taking the derivatives of $ \vr e  $ with respect to $ ( \vr , \eta)$
%\footnote{Let us point out that we consider $\vr$ and $\eta$ as the independent thermodynamical variables throughout the whole paper, meaning all other thermodynamical variables are functions of $(\vr, \eta)$.}
 we obtain
\begin{equation}\label{DE1}
 \partial_{\vr} \left( \vr e \right) =\left(1+C_v \right)  \vt- \frac{\eta \vt}{\vr} , \quad
\partial_{\eta} \left( \vr e \right) = \vt. 
\end{equation}
Further, by the product rule and Gibbs' relation \eqref{Gibbs} we derive
\begin{equation}\label{DE3}
\DD \vt =\frac{\vt}{C_v \vr} \left(1 - \frac{\eta}{\vr}\right)  \DD \vr  + \frac{\vt}{C_v \vr} \DD \eta,
\end{equation}
and 
\begin{equation}\label{DE4}
\DD \left( (1+C_v)\vt - \frac{\vt \eta}{\vr} \right) = \left( (1+C_v) - \frac{\eta}{\vr} \right) \DD \vt - \frac{\vt}{\vr} \DD \eta + \frac{\vt \eta}{\vr^2} \DD \vr,
\end{equation}
which leads to
\begin{equation}\label{DE2}
\nabla_{(\vr,\eta)}^2 (\vr e) = \frac{\vt}{C_v \vr} 
\begin{pmatrix}
1 & 1-\frac{\eta}{\vr} \\
1-\frac{\eta}{\vr} & C_v + \left( 1-\frac{\eta}{\vr} \right)^2
\end{pmatrix}.
\end{equation}
As $(\vr, \vu, \eta) $ is the strong solution we know that $\nabla_{(\vr,\eta)}^2 (\vr e) |_{(\tvr, \tvS)}$ is symmetric positive definite and bounded from below and above, which  implies
\begin{align} \label{REL2}
\E  \left( \vrh,  \mh, \eta_h \Big| \tvr, \tvu, \tvS \right) 
& \approx  |\vu_h - \tvu|^2 +  |\eta_h - \tvS|^2 + |\vrh - \tvr|^2.
\end{align}
Next, we recall Assumption \ref{H1} and  the uniform upper bound of $\vuh$ due to Lemma~\ref{Lemma1} to conclude that 
\begin{align*}
&|\vm_h - \tvm|^2  \leq |\vrh(\vuh - \tvu)|^2 + |(\vr_h - \tvr) \tvu)|^2   \lesssim  |\vuh - \tvu|^2 +  |\vrh - \tvr|^2,\\
& |\vuh - \tvu|^2 \lesssim |\tvr(\vuh - \tvu)|^2  \lesssim
 |\mh - \tvm|^2 + |\vu_h(\tvr - \vrh)|^2 \lesssim
 |\vm_h - \tvm|^2 + |\vrh - \tvr|^2  .
 \end{align*}
Substituting the above two inequalities into \eqref{REL2} we finish the proof. 
\end{proof}
Lemma~\ref{LMEQ} means that the $L^1$-norm of $\E  \left( \vrh,  \mh, \eta_h \Big| \tvr, \tvu, \tvS \right) $ is equivalent to the $L^2$-norm of the errors in   $(\vrh-\tvr, \mh-\tvm, \eta_h-\tvS)$ as long as the entropy stable numerical solution $( \vrh,  \mh, \eta_h)$ satisfies Assumption \ref{H1} and  $(\tvr, \tvu, \tvS )$ is the strong solution of the Euler system in the sense of Definition~\ref{DefS}.

\section{Error estimates}\label{sec_ee}
Equipped with consistency formulation of the Godunov method we are now ready to estimate the relative energy in the $L^1$-norm and error between the numerical solution $( \vrh, \mh, \eta_h )$  and the strong solution $( \tvr, \tvu, \tvS )$ in the $L^2$-norm. 

\begin{Theorem}[Error estimates]\label{Th2}
Let $\Omega \subset \mathbb{R}^d$, $d=1,2,3,$ be a bounded domain with a boundary $\partial \Omega \in C^1$. 
Let $( \tvr, \tvu, \tvS )$ be the strong solution of the complete Euler system \eqref{pde}  in the sense of Definition \ref{DefS} with initial data \eqref{ini} satisfying 
\begin{equation*}
\| \vU_{h0} -\vU_0\|_{L^2(\Omega)} \lesssim h^{1/2} 
\end{equation*}
 and the impermeability boundary condition \eqref{I5}. 
%$\tvr(0, \cdot) = \vr_0 ,\ \tvr(0, \cdot) \tvu(0, \cdot) = \vm_0,\ \tvS(0, \cdot) = \eta_0$ {\cred satisfying $(\vr_0 ,\vm_0, \eta_0) \in BV(\mathbb{R}^{d+2})$.}
%\begin{equation*}
%\tvr(0, \cdot) = \vr_0 ,\ \tvr(0, \cdot) \tvu(0, \cdot) = \vm_0,\ \tvS(0, \cdot) = \eta_0.
%\end{equation*}
%Moreover, assume that
%\begin{equation}\label{H2}
%0 < \underline{\vr} \leq \tvr,\quad 0 < \underline{\vt} \leq \tvt.
%%0 < \underline{\vr} \leq \tvr \leq \Ov{\vr},\quad
%% |\tvS| \leq \Ov{\eta}. 
%\end{equation}

Suppose that $( \vrh, \vmh, \eta_h )$ is the numerical solution obtained by the Godunov method \eqref{eq:semi-discrete-RPflux-phi}.
%with initial data $(\vr_{h0}, \vm_{h0}, \eta_{h0}) := (\projection{\vr_0}, \projection{\vm_0}, \projection{\eta_0})${\footnote{ The notation $\Pi_h$ is the piecewise constant projection operator.}}. 
Let Assumption \ref{H1} hold.
Then the following estimate of the relative energy holds for any $\tau\in(0,T]$ 
\begin{equation}\label{ECR}
\intO{{\E} \left( \vrh,  \mh, \eta_h \Big| \tvr, \tvu, \tvS \right)(\tau, \cdot) }  \\
\lesssim \exp \left( \tau \; c\left( \Omega, \| \tvU \|_{W^{1, \infty}((0,T) \times \Omega; R^d)} \right) \right)  h^{1/2}. 
\end{equation}
\end{Theorem}

\begin{proof}
We prove \eqref{ECR} in two steps: 
\begin{itemize}
\item Viewing $( \tvr, \tvu, \tvS )$ as the test function in the consistency formulation, we derive the relative energy inequality between $( \vrh, \mh, \eta_h )$ and $( \tvr, \tvu, \tvS )$;
\item Approximating the above inequality such that all  terms on the right hand side can be bounded by the discretization parameter $h$ or by the relative energy, we finally estimate the relative energy by Gronwall's lemma. 
\end{itemize}
%In what follows we show the details. 
\paragraph{Step 1.}
%In this step, we derive the relative energy inequality.  
Rewriting the relative energy \eqref{RE} into a more convenient form we obtain
\begin{equation}\label{RE1}
\begin{aligned}
{E} \left( \vr_h, \vm_h, \eta_h \Big| \tvr, \tvu, \tvS \right) 
&= ~ \frac{1}{2} \vr_h \left| \frac{\vm_h}{\vr_h} - \tvu \right|^2  + \vr_h e_h
- \left( (1+C_v) \tvt - \frac{\tvt \tvS}{\tvr} \right) (\vr_h - \tvr) - \tvt (\eta_h - \tvS ) -  \tvr \tve 
\\& = 
\left[ \frac{1}{2} \frac{|\vm_h|^2}{\vr_h} + \vr_h e_h \right] + \vr_h  \left[ \frac{1}{2} |\tvu|^2 -  (1+C_v)\tvt + \frac{\tvt \tvS}{\tvr} \right]  - \vm_h \cdot \tvu- \eta_h  \tvt + \tvp. 
\end{aligned}
\end{equation}
First, we take $\frac{1}{2} |\tvu|^2 -  (1+ C_v) \tvt + \frac{\tvt \tvS}{\tvr}$ as the test function in consistency formulation of the density equation \eqref{CE1} to derive 
\begin{align*}
&\left[ \int_{\Omega} \vrh \left(\frac{1}{2} |\tvu|^2 -  (1+ C_v) \tvt + \frac{\tvt \tvS}{\tvr}\right) ~\dx  \right]_{t = 0}^{t = \tau} 
=   \int_{0}^{\tau} \int_{\Omega} \Bigg( \vrh \partial_t \left(\frac{1}{2} |\tvu|^2 -  (1+ C_v) \tvt + \frac{\tvt \tvS}{\tvr}\right) 
\\ &  \quad +   \mh \cdot \Grad  \left(\frac{1}{2} |\tvu|^2 -  (1+ C_v) \tvt + \frac{\tvt \tvS}{\tvr}\right) \Bigg)  \dxdt + \int_{0}^{\tau} e_{\vr,h}(t,\tvU) \dt.
\end{align*}
Analogously, we set $\tvu$ and $\tvt$ respectively as the test functions in consistency formulations of the momentum equation \eqref{CE2} and entropy inequality \eqref{CE3} to get 
\begin{align*}
\left[ \int_{\Omega} \mh \cdot \tvu ~\dx  \right]_{t = 0}^{t = \tau} 
=~\int_{0}^{\tau} \intOB{ \mh \cdot \partial_t \tvu +   \frac{\mh \otimes\mh }{\vrh} : \Grad  \tvu  + p_h  \Div  \tvu  } \dt + \int_{0}^{\tau} e_{\vm,h}(t,\tvU) \dt ,
\end{align*}
and
\begin{align*}
\left[ \int_{\Omega} \eta_h \tvt ~\dx  \right]_{t = 0}^{t = \tau} 
\geq ~\int_{0}^{\tau} \intOB{ \eta_h \partial_t \tvt +  \eta_h \frac{\mh}{\vrh} \cdot \Grad  \tvt }\dt + \int_{0}^{\tau} e_{\eta,h}(t,\tvU) \dt.
\end{align*}
Then, we  combine the above three formulae together with the energy equality \eqref{CE4} and find
\begin{equation}\label{RE1}
\begin{aligned}
& \left[ \intO{\E  \left( \vrh,  \mh, \eta_h \Big| \tvr, \tvu, \tvS \right)(t, \cdot) } \right]_{t = 0}^{t = \tau}  
\leq  - \int_0^\tau \intO{  \frac{(\vrh \tvu - \mh) \otimes (\vrh \tvu - \mh) }{\vrh} : \Grad \tvu } \dt \\ 
	&+ \int_0^\tau \intO{\big( ( \tvp - p_h ) \Div \tvu  +  (\partial_t \tvp + \tvu \cdot \Grad \tvp) \big)} \dt    \\
	&+ \int_0^\tau \intO{ (\vrh \tvu - \mh ) \cdot \left[ \partial_t \tvu + \tvu \cdot \Grad \tvu +  \frac{1}{\tvr} \Grad \tvp  \right]} \dt  \\
	& - \int_0^\tau \intOB{ \vrh \partial_t \left( (1+ C_v)\tvt - \frac{\tvt \tvS}{\tvr} \right) + \mh \cdot \Grad \left( (1+ C_v)\tvt - \frac{\tvt \tvS}{\tvr} \right) } \dt  \\
	& - \int_0^\tau \intOB{ \eta_h \partial_t \tvt + \eta_h \frac{\mh}{\vrh} \cdot \Grad \tvt  + (\vrh \tvu - \mh) \frac{1}{\tvr} \Grad \tvp } \dt  \\
	& + \int_{0}^{\tau} \left( e_{\vr,h}(t,\tvU) - e_{\vm,h}(t,\tvU) - e_{\eta,h}(t,\tvU)\right) \dt, 
\end{aligned}
\end{equation}
where we have used the following identities
\begin{align*}
 \tvu \otimes \tvu : \Grad \tvu  =  \tvu \cdot (\tvu \cdot \Grad \tvu), \qquad 
\intO{ \tvu \cdot\Grad \tvp }  = - \intO{ \tvp \,\Div \tvu } , 
\\
\frac{(\vrh \tvu - \mh) \otimes (\vrh \tvu - \mh) }{\vrh} : \Grad \tvu = ~ \vrh \tvu \otimes \tvu : \Grad \tvu + \frac{\mh \otimes \mh}{\vrh} : \Grad \tvu   
\\  - \mh \cdot (\tvu \cdot \Grad) \tvu  -  \tvu \cdot \left( \mh\cdot \Grad \right) \tvu  .
\end{align*}
Further, employing the relations \eqref{DE3} and \eqref{DE4}  we can reformulate \eqref{RE1}  as  
\begin{equation}\label{RE1-1}
\setlength\abovedisplayskip{2pt}
\setlength\belowdisplayskip{0pt}
\begin{aligned}
& \left[ \intO{{\E} \left( \vrh,  \mh, \eta_h \Big| \tvr, \tvu, \tvS \right)(t, \cdot) } \right]_{t = 0}^{t = \tau}   \\
\leq &~ - \int_0^\tau \intO{  \frac{(\vrh \tvu - \mh) \otimes (\vrh \tvu - \mh) }{\vrh} : \Grad \tvu } \dt  \\ 
&- \int_0^\tau \intO{ \Big[ p_h - \tvp - \partial_{\tvr} \tvp   (\vrh - \tvr)- \partial_{\tvS} \tvp   (\eta_h - \tvS) \Big] \Div \tvu } \dt    \\
&+ \int_0^\tau \intO{ (\vrh \tvu - \mh ) \cdot \left[ \partial_t \tvu +
	\tvu \cdot \Grad \tvu +  \frac{1}{\tvr} \Grad \tvp  \right]} \dt  \\
& + \int_0^\tau \intO{ \left[ \left(1- \frac{\vrh}{\tvr} \right)\partial_{\tvr} \tvp  -  \left(\eta_h- \frac{\vrh}{\tvr} \tvS \right)\partial_{\tvr} \tvt\right]   \left[ \partial_t \tvr + \tvu \cdot \Grad \tvr +  \tvr \, \Div  \tvu \right] } \dt  \\
& + \int_0^\tau \intO{ \left[ \left(1- \frac{\vrh}{\tvr} \right)\partial_{\tvS} \tvp  -  \left(\eta_h- \frac{\vrh}{\tvr} \tvS \right)\partial_{\tvS} \tvt\right]   \left[ \partial_t \tvS + \tvu \cdot \Grad \tvS +  \tvS \, \Div  \tvu \right] } \dt  \\
& - \int_0^\tau \intO{  \left( \eta_h - \frac{\vrh}{\tvr} \tvS \right)  \left(\frac{\mh}{\vrh} - \tvu \right) \cdot \Grad \tvt } \dt  \\
& + \int_{0}^{\tau} \left( e_{\vr,h}(t,\tvU) - e_{\vm,h}(t,\tvU) - e_{\eta,h}(t,\tvU)  \right) \dt,
\end{aligned} 
\end{equation}
where we have denoted $\partial_{\tvr} \tvp:= \frac{\pd p(\tvr,\tvS)}{\pd \vr}$ and the definitions of $\partial_{\tvS} \tvp$,  
$\partial_{\tvS} \tvt$ and $\partial_{\tvS} \tvt$ are analogous. 

Then applying the equalities stated in \eqref{EQS} to \eqref{RE1-1} we have
\begin{equation}\label{RE2}
\begin{aligned}
& \left[ \intO{\E  \left( \vrh,  \mh, \eta_h \Big| \tvr, \tvu, \tvS \right)(t, \cdot) } \right]_{t = 0}^{t = \tau}   \\
& \leq  - \int_0^\tau \intO{  \frac{(\vrh \tvu - \mh) \otimes (\vrh \tvu - \mh) }{\vrh} : \Grad \tvu } \dt  \\ 
&\quad - \int_0^\tau \intO{ \Big[ p_h - \tvp - \partial_{\tvr} \tvp  (\vrh - \tvr)- \partial_{\tvS} \tvp  (\eta_h - \tvS) \Big] \Div \tvu } \dt    \\
&\quad - \int_0^\tau \intO{  \left( \eta_h - \frac{\vrh}{\tvr} \tvS \right)  \left(\frac{\mh}{\vrh} - \tvu \right)\cdot \Grad \tvt } \dt  \\
&\quad + \int_{0}^{\tau} \left( e_{\vr,h}(t,\tvU) - e_{\vm,h}(t,\tvU) - e_{\eta,h}(t,\tvU) \right) \dt.
\end{aligned}
\end{equation}

\paragraph{Step 2.} In this step we shall estimate the right hand side of the inequality \eqref{RE2} and complete the proof by Gronwall's lemma. 
We begin with the following observation owing to the uniform bounds on $\tvr$, $\tvt$ and $\tvS$, as well as \eqref{REL2}
\begin{align*}
& \left|\left( \eta_h - \frac{\vrh}{\tvr} \tvS \right) \cdot \left(\frac{\mh}{\vrh} - \tvu \right)\right| 
 \lesssim   \left|\ \eta_h - \frac{\vrh}{\tvr} \tvS \right|^2 + \left| \frac{\mh}{\vrh} - \tvu \right|^2 
 \\&
\leq \left|\ \eta_h -\tvS \right|^2+ \left|\tvS- \frac{\vrh}{\tvr} \tvS \right|^2 + \left| \frac{\mh}{\vrh} - \tvu \right|^2
 = \left|\ \eta_h -\tvS \right|^2+ \left|\tvr- \vrh \right|^2 \left( \frac{\tvS}{\tvr}  \right)^2 + \left| \frac{\mh}{\vrh} - \tvu \right|^2
 \\& 
 \lesssim  ~  |\eta_h - \tvS|^2 + |\vrh - \tvr|^2+  \left|\frac{\mh}{\vrh} - \tvu\right|^2 
%\leq \left|\frac{\mh}{\vrh} - \tvu\right|^2 + \vrh e_h  - \partial_{\tvr} (\tvr \tve) (\vrh - \tvr) -\partial_{\tvS} (\tvr \tve) (\eta_h - \tvS ) -  \tvr \tve
%\\ & 
\lesssim \E  \left( \vrh,  \mh, \eta_h \Big| \tvr, \tvu, \tvS \right). 
\end{align*}
Hence, we may estimate \eqref{RE2} in the following way
\begin{multline}\label{RE3}
 \left[ \intO{\E  \left( \vrh,  \mh, \eta_h \Big| \tvr, \tvu, \tvS \right)(t, \cdot) } \right]_{t = 0}^{t = \tau} 
  \leq   c(\norm{\tvU}_{W^{1,\infty}((0,T) \times \Omega;R^{d+2})}) h^{1/2}   \\
+  c\left( \Omega, \| \tvU \|_{W^{1, \infty}((0,T) \times \Omega; R^{d+2})} \right) \int_0^\tau \intO{\E  \left( \vrh,  \mh, \eta_h \Big| \tvr, \tvu, \tvS \right)(t, \cdot) }  \dt,  
\end{multline}
where we have recalled the consistency error stated in Theorem~\ref{Th1}. 

Further, applying Gronwall's lemma and recalling the projection error for piecewise constant functions 
\begin{equation*}
\norm{\tvU- \Pi_h \tvU}_{L^\infty} \aleq c(\norm{\tvU}_{W^{1,\infty}((0,T)\times \Omega; R^{d+2}) }) h
\end{equation*}
we conclude the proof, i.e. 
\begin{align*}
& \intO{\E  \left( \vrh,  \mh, \eta_h \Big| \tvr, \tvu, \tvS \right)(\tau, \cdot) }  
\leq  c(\norm{\tvU}_{W^{1,\infty}((0,T) \times \Omega;R^{d+2})}) h^{1/2}  
\\ &\quad + \exp \left( \tau \; c\left( \Omega, \| \tvU \|_{W^{1, \infty}((0,T) \times \Omega; R^d)} \right) \right) \intO{\E  \left( \vr_{h0},  \vm_{h0}, \eta_{h0} \Big| \vr_0, \vm_0/\vr_0, \eta_0 \right) }  
\\& \lesssim   h^{1/2}  \exp \left( \tau \; c\left( \Omega, \| \tvU \|_{W^{1, \infty}((0,T) \times \Omega; R^d)} \right) \right)  . 
\end{align*}
\end{proof}

We directly obtained the following a priori error estimates in the $L^2$-norm.
\begin{Proposition}\label{PR1}
Under the same condition as Theorem~\ref{Th2} it holds for any $\tau \in (0,T)$ 
\begin{equation}
\|\vr_h - \tvr\|_{L^2(\Omega)} \lesssim h^{1/4}, \quad
\|\vm_h - \tvm\|_{L^2(\Omega)} \lesssim h^{1/4}, \quad
\|\eta_h - \tvS\|_{L^2(\Omega)} \lesssim h^{1/4}.
\end{equation}
\end{Proposition}
%From Theorem \ref{Th2} and \eqref{REL2} we have
%\begin{equation*}
%(\vrh,\ \mh,\ \eta_h ) ~ \longrightarrow ~ (\tvr,\  \tvr \tvu, \tvS) \quad \mbox{\rm in} \quad  L^2(\Omega) 
%\end{equation*}
%for a.a $\tau \in (0,T)$. 
%Moreover, the convergence rate of $L^2$-error of $(\vrh,\ \mh,\ \eta_h )$ is {\bf at least $1/4$}.

In what follows we prove the first order convergence rate in terms of the relative energy under an additional assumption of bounded total variation for numerical solutions.
\begin{Theorem}\label{Th3}
In addition to the assumptions of Theorem~\ref{Th2}, we assume that
\begin{equation} \label{TVB1}
\sum_{\sigma \in \inerface} \int_{\sigma} |\jump{\vU_h}|_{\sigma}  \dsx   \lesssim 1. 
\end{equation}
Then there hold 
\begin{equation*}
\intO{\E  \left( \vrh,  \mh, \eta_h \Big| \tvr, \tvu, \tvS \right)(\tau, \cdot) }  \\
\lesssim h \exp \left( \tau \; c\left( \Omega, \| \tvU \|_{W^{1, \infty}((0,T) \times \Omega; R^d)} \right) \right) .
\end{equation*}
\end{Theorem}
\begin{proof}
With \eqref{TVB1} the consistency error can be estimated and improved by
\begin{equation*}
\|e_{j,h}(\phi)\|_{L^1(0,T)} \lesssim h \| \phi \|_{W^{1,\infty}((0,T)\times {\Omega})} , 
\end{equation*}
which concludes the proof.
\end{proof}

\begin{Remark}
Here we point out that the assumption \eqref{TVB1} is slightly weaker than the assumption used in \cite{JoRo}
\begin{equation} \label{TVB2}
\sum_{\sigma \in \inerface} \int_{\sigma} \frac{| \jump{\vU_h}| ^2}{h}  \dsx   \lesssim 1. 
\end{equation}

Moreover, for the case of $d=1$ the assumption \eqref{TVB1} is exactly the TVB condition, which is a known property for the Godunov method. 
\end{Remark}

\begin{Remark}
Let us consider piecewise constant initial data which generate finitely many rarefaction waves. It is obvious that such kind of initial data fulfills the condition $\| \vU_{h0} - \vU_0  \|_{L^2(\Omega)} \lesssim h^{1/2} $ assumed in Theorem~\ref{Th2}. 
Moreover, we can expect \eqref{TVB1} or \eqref{TVB2} to hold, which consequently implies Theorem~\ref{Th3}. 
Thus, $\| \vU_h - \tvU \|_{L^2(\Omega)} \lesssim h^{1/2} $ holds for any $\tau\in(0,T)$.
%the  $L^2$-error of $(\vrh,\, \mh,\, \eta_h )$ is  $\mathcal{O}(h^{1/2})$. 

%In addition, we would like to point out that,  the results for scalar equation in \cite{Tang-Teng:1997} showed that the convergence rate of $L^2$-error can be a little less than $1/2$. 
\end{Remark}

%\begin{Remark}
%Note that our analysis does not hold for the solution  containing discontinuities (contact wave or shock waves). 
%This is because the test function of the relative energy is no longer continuous so that we can not use the estimates of consistency error \eqref{CE5}. 
%\end{Remark}

%%%%%%%%%%%%%%%%%%%%%%%%%%%%%%%%%%%%%%%%%%%%%%%%%%%%%%%%%%%%%%%%%%%%%

\section{Numerical experiments}\label{sec_exp}
In this section we simulate several one- and two-dimensional Riemann problems.  
The examples only containing rarefaction waves are used to validate our theoretical results.  
In addition, we also test examples containing contact waves or shock waves or both and compute experimentally convergence rates.
We point out that in our simulations there is no projection error of initial data due to these simple Riemann problems and good uniform meshes.

In the following we calculate the relative energy in the $L^1$-norm and 
 the errors  of $(\vr, \vm,\eta)$ in the $L^2$-norm. 
In addition to the Godunov method, we also test the convergence rates of the viscosity finite volume (VFV) method  originally introduced and studied by Feireisl et al.~\cite{FLM}. 
In our numerical tests, we  take $\gamma=1.4$ and $\rm CFL=0.9$ for the Godunov method while $\rm CFL=0.3$ is used for the VFV method. 
%Moreover, for the sake of simplicity we call the Godunov method as RP method hereafter. 
Unless otherwise specified,  the errors of $(\vr,\vm,\eta), \E$ mean the $L^2$-error of $(\vr,\vm,\eta)$ and the $L^1$-norm of the relative energy  $\E$; the convergence rates of $(\vr,\vm,\eta), \E$ mean the convergence rate of the $L^2$-error of $(\vr,\vm,\eta)$ and the $L^1$-norm of $\E$.

\subsection{One dimensional experiments}\label{sec_exp:1D}
We start with one dimensional Riemann problems in the computational domain $\Omega=[0,1]$. 
Here, the solution $\tvU$ in the relative energy is taken as the reference (exact) solution computed on the uniform mesh with $20480$ cells.

\begin{Example}[{\tt 1D single wave}] \label{example:1D-singlewave} \rm
This example is used to measure the convergence rate of three different types of waves --  a single contact ({\tt C}) wave, a single rarefaction ({\tt R})  wave and a single shock ({\tt S}) wave. 

\begin{table}[htpb]
		\centering
		\caption{ Initial data of 1D single wave. } \label{table:1D-singlewave}
		\begin{tabular}{|c|c|ccc|c|c|ccc|c|c|ccc|}
			\hline
			\multicolumn{2}{|c|}{}						& $\vr$ & $u$   & $p$ & \multicolumn{2}{c|}{}						& $\vr$ & $u$   & $p$ & \multicolumn{2}{c|}{}						& $\vr$ & $u$   & $p$\\
			\hline
			\hline
			\multirow{2}{*}{\tt C} 		& left	 	& 0.5 & 0.5 & 5
			& \multirow{2}{*}{\tt R} 		& left	 	& 0.5197 & -0.7259 & 0.4 
			& \multirow{2}{*}{\tt S} 		& left	 	& 1 & 0.7276 & 1\\
			\cline{2-5}
			\cline{7-10}
			& right	& 1 & 0.5 & 5 &						
			& right	& 1 & 0 & 1 &
			& right	& 0.5313 & 0 & 0.4 \\
			\hline
		\end{tabular}
	\end{table}
	
Given the initial data in Table \ref{table:1D-singlewave}, we compute the contact, rarefaction and shock wave till $T = 0.2, 0.2$ and $0.25$, respectively. 
Figure \ref{figure:1D-singlewave-rho} (resp. Figure \ref{figure:1D-singlewave-entropy}) shows the density $\vr$ (resp. the entropy $\eta$) obtained on different meshes with $n (=1/h) =32, 64, \dots,1024$ cells. 
Moreover, we present in Figure \ref{figure:1D-singlewave}  the errors of $(\vr,  \vm, \eta)$ in $L^2$-norm and $\E$ in $L^1$-norm, see the details in Table \ref{table:1D-singlewave-error} and \ref{table:1D-singlewave-error-1}. 
%We show the plots of $\vr$ in Figure \ref{figure:1D-singlewave-rho} and $\eta$ in Figure \ref{figure:1D-singlewave-entropy}. 
%Moreover, we present in Figure \ref{figure:1D-singlewave}  the errors of $\vr,  m, \eta$ in $L^2$-norm and $\E$ in $L^1$-norm, see also part of the results in Table \ref{table:1D-singlewave-error} and \ref{table:1D-singlewave-error-1}. 

The numerical results show that 
\begin{itemize}
\item the Godunov method and the VFV method have  similar convergence rates;
 
\item for single rarefaction wave the convergence rate of $(\vr,\vm,\eta)$ (resp. $\E$)  is slightly greater than $1/2$ (resp. $1$), which is consistent to our theoretical results;

\item for single contact wave the convergence rate of $(\vr,\vm,\eta)$ (resp. $\E$)  is around $1/4$ (resp. $1/2$); 
%which is consistent to the sharpness $L^1$-error estimate in \cite{Tang-Teng:1995};

\item for single shock wave the convergence rate of $(\vr,\vm,\eta)$ (resp. $\E$)  is around $1/2$ (resp. $1$).
\end{itemize}

%From Table \ref{table:1D-singlewave-error} we observe that the convergence rate of $\vr,\eta$ (resp. $\E$) obtained by the RP method is around $1/4$ (resp. $1/2$) for the single contact wave, $1/2$ (resp. $1$) for the shock wave, and slightly higher than $1/2$ (resp. $1$)  for the rarefaction wave. From Table  \ref{table:1D-singlewave-error-1}  we observe similar convergence rate for the VFV method. 
\end{Example}	
\begin{Remark}\rm 
Here we compare the above observation with the result of Tadmor and Tang \cite{Tadmor-Tang:1999} for the rarefaction wave and the shock wave. 
\begin{itemize}
\item Directly applying the pointwise error estimate for scalar equation in \cite{Tadmor-Tang:1999}, i.e.  
\begin{equation*}%\label{eq:point-error-rarefaction}
|(u^{\varepsilon} - u)(x,t)|  \approx \mbox{dist}(x,R(t))^{-1}\varepsilon \log^2 \varepsilon
\end{equation*}
with rarefaction set $R(t)$, 
we obtain that the $L^2$-error is bounded by $\varepsilon^{1/2} \log^2 \varepsilon$. 
Setting the vanishing viscosity coefficient $\varepsilon \approx h$ means that our analysis gives a better convergence rate.
\item
Applying the pointwise error estimate for scalar equation in \cite{Tadmor-Tang:1999}, i.e.  
\begin{equation*}%\label{eq:point-error-shock}
|(u^{\varepsilon} - u)(x,t)|  \approx  \mbox{dist}(x,S(t))^{-1} \varepsilon,
\end{equation*}
where $S(t)$ is the streamline of shock discontinuities, 
we obtain that $L^2$-convergence rate is $1/2$, which is consistent with our observations.
\end{itemize}

\end{Remark}

\begin{figure}[htbp]
	\setlength{\abovecaptionskip}{0.cm}
	\setlength{\belowcaptionskip}{-0.cm}
	\centering
	\begin{subfigure}{0.32\textwidth}
		\includegraphics[width=\textwidth]{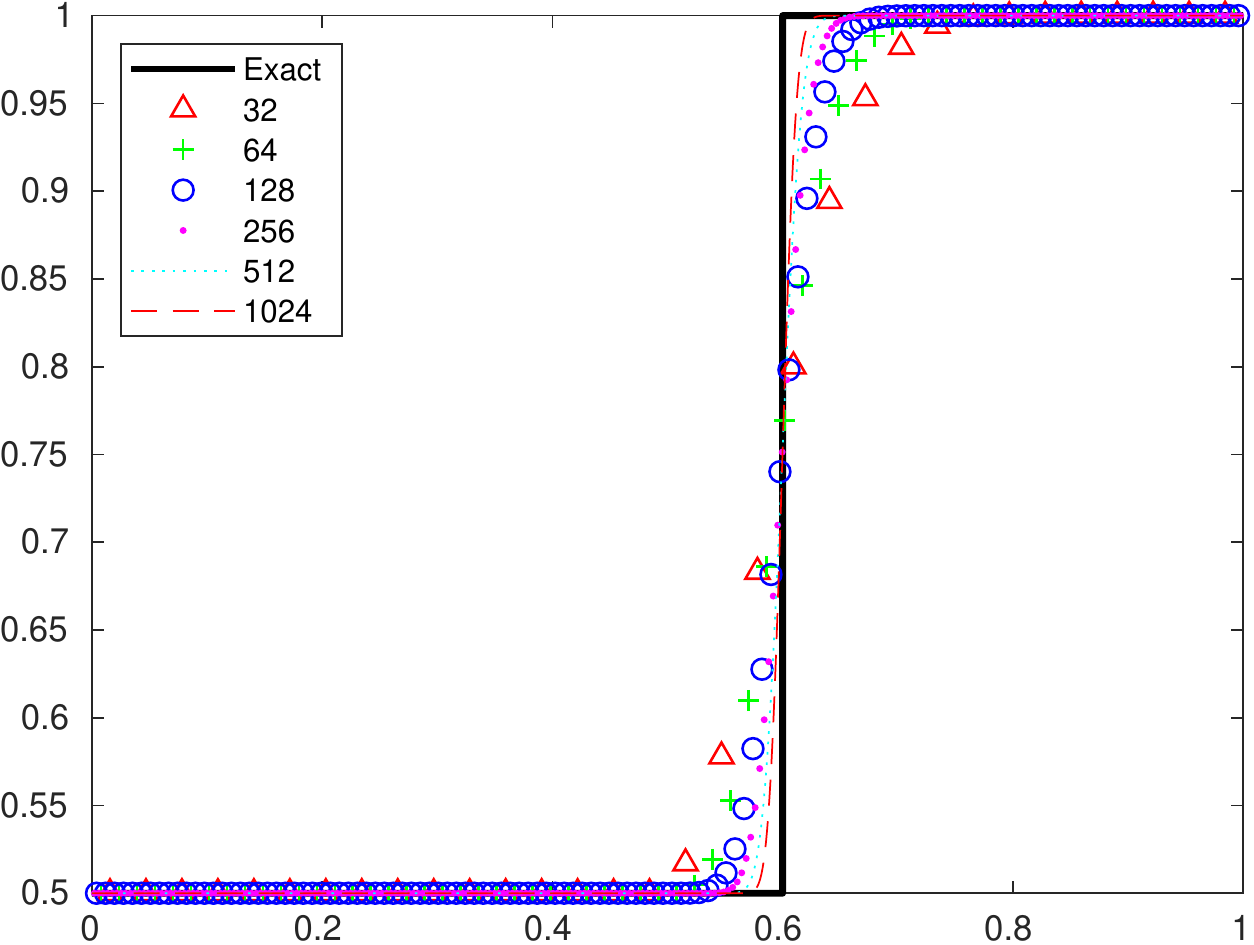}
		\caption{ Godunov - Contact }
	\end{subfigure}	
	\begin{subfigure}{0.32\textwidth}
		\includegraphics[width=\textwidth]{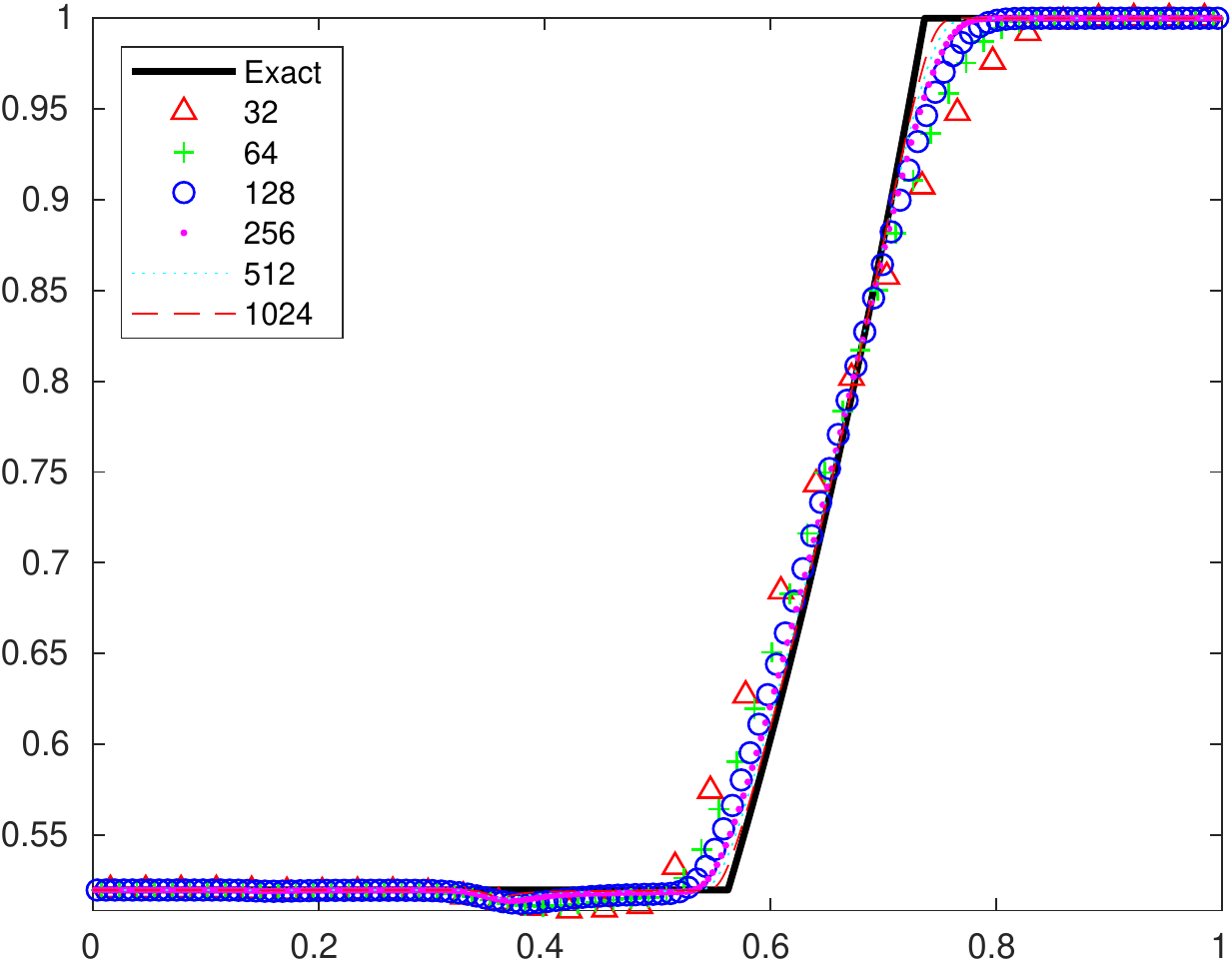}
		\caption{ Godunov - Rarefaction}
	\end{subfigure}	
	\begin{subfigure}{0.32\textwidth}
		\includegraphics[width=\textwidth]{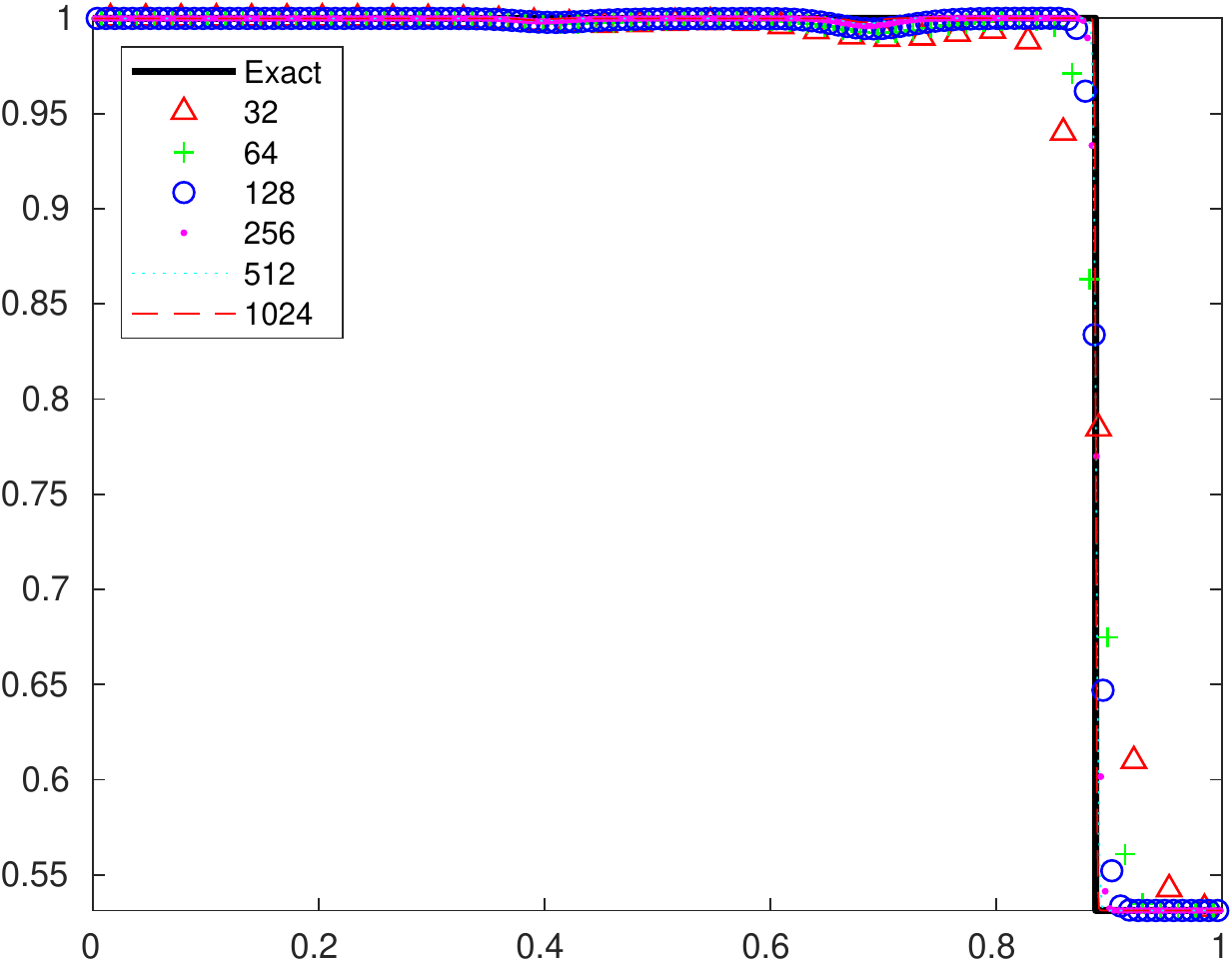}
		\caption{Godunov - Shock}
	\end{subfigure}	\\
	\begin{subfigure}{0.32\textwidth}
		\includegraphics[width=\textwidth]{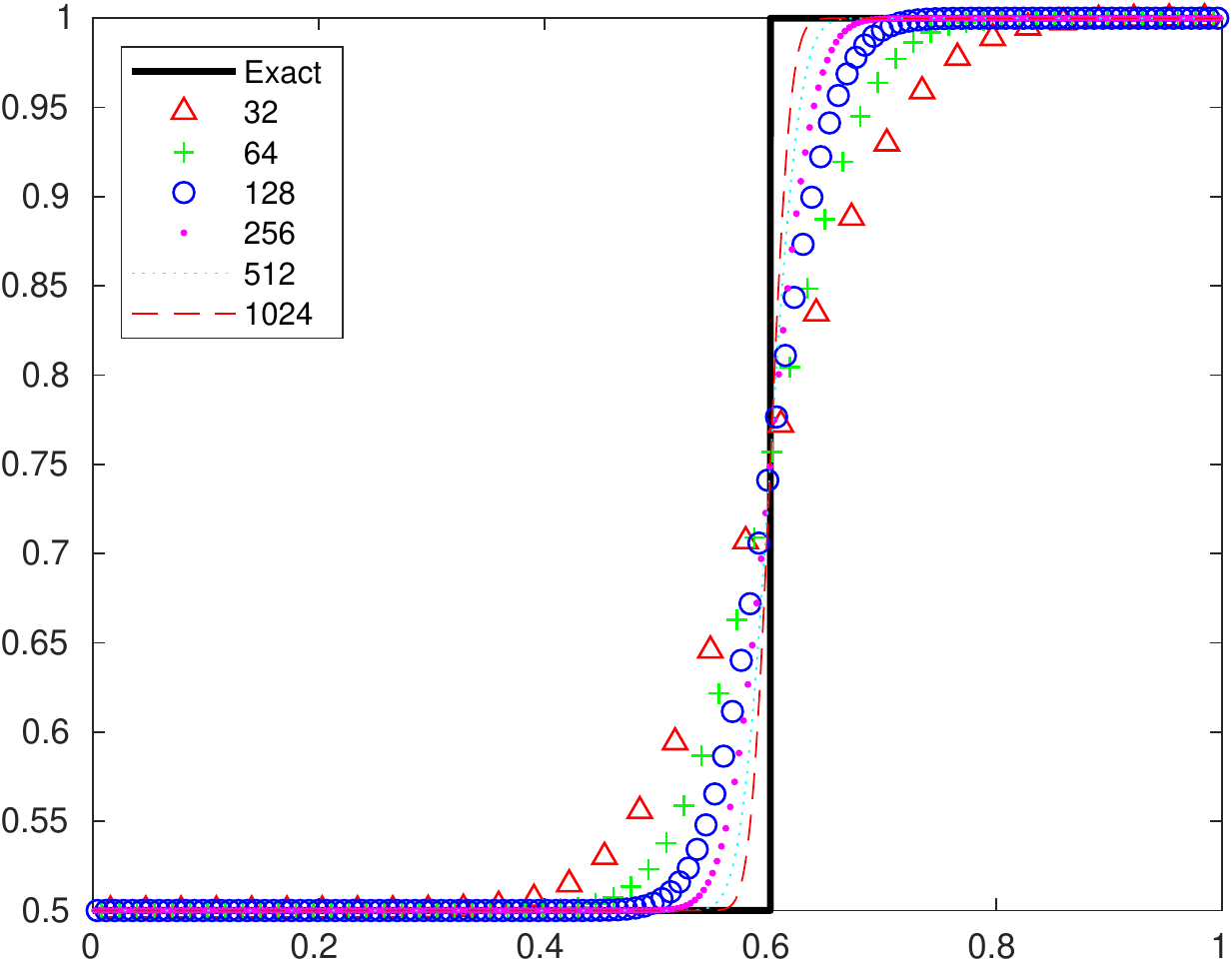}
		\caption{ VFV - Contact }
	\end{subfigure}	
	\begin{subfigure}{0.32\textwidth}
		\includegraphics[width=\textwidth]{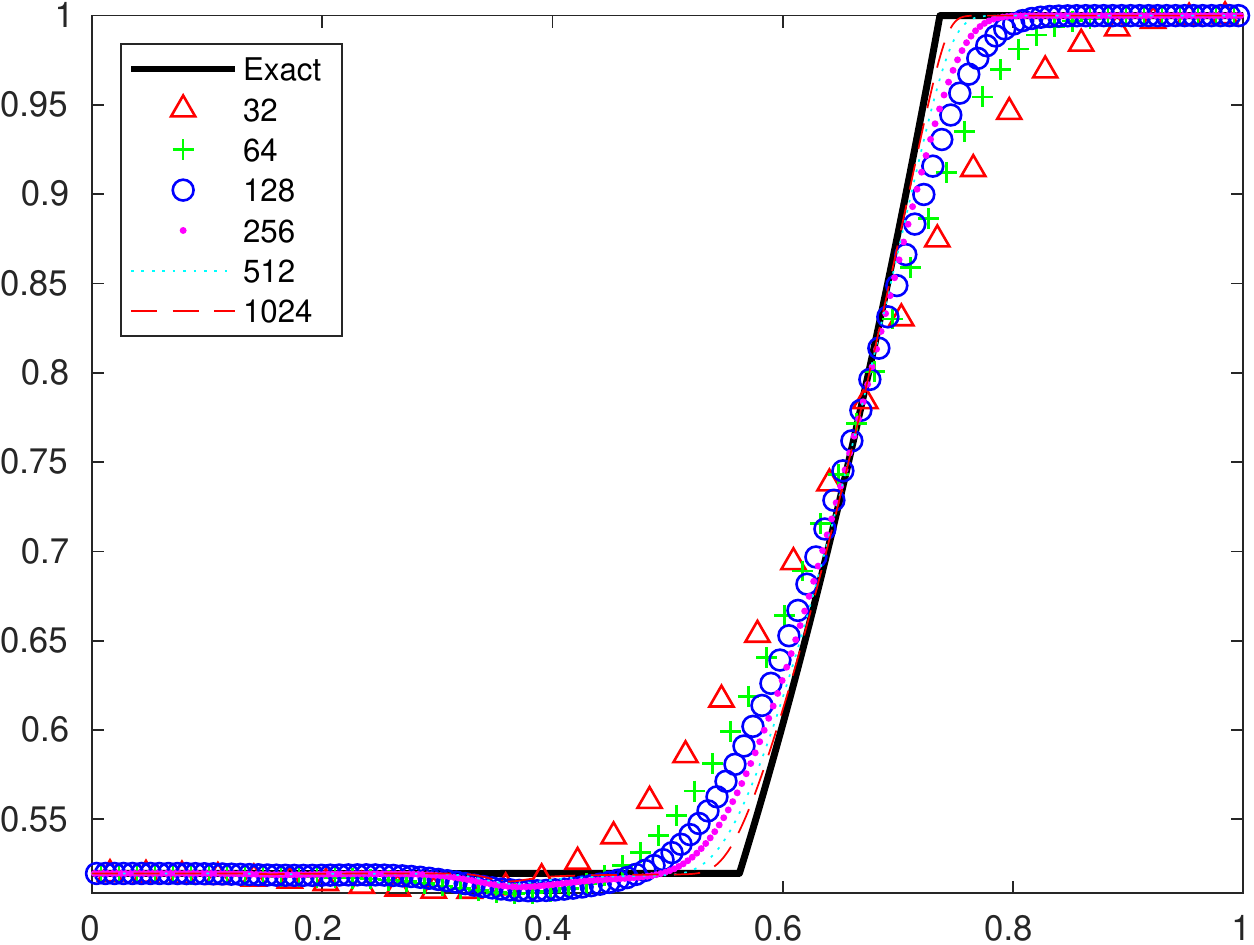}
		\caption{ VFV - Rarefaction}
	\end{subfigure}	
	\begin{subfigure}{0.32\textwidth}
		\includegraphics[width=\textwidth]{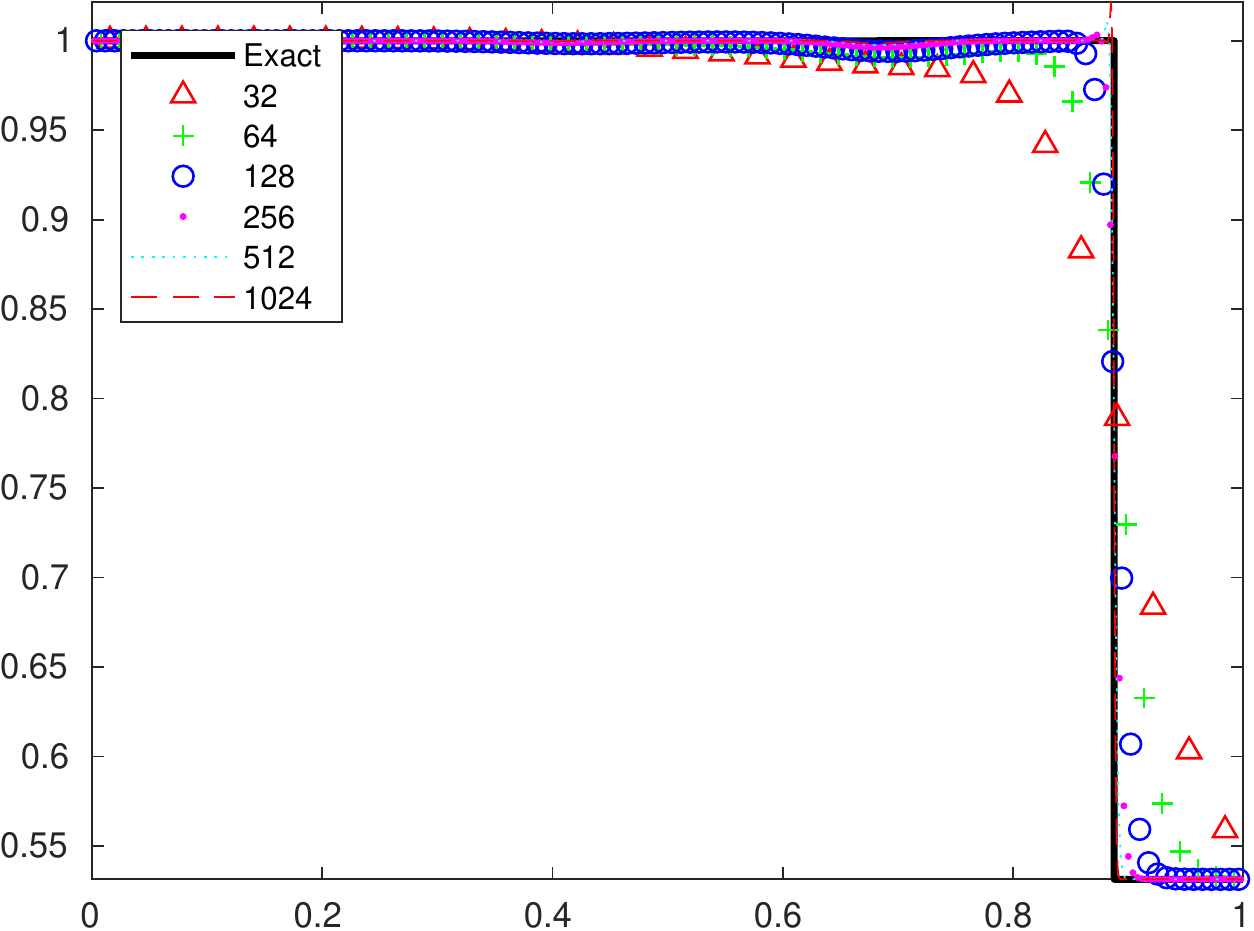}
		\caption{VFV - Shock}
	\end{subfigure}	
	\caption{\small{Example \ref{example:1D-singlewave}: density $\vr$ obtained by the Godunov method (top) and the VFV method (bottom).}}\label{figure:1D-singlewave-rho}
\end{figure}

\begin{figure}[htbp]
	\setlength{\abovecaptionskip}{0.cm}
	\setlength{\belowcaptionskip}{-0.cm}
	\centering
	\begin{subfigure}{0.32\textwidth}
		\includegraphics[width=\textwidth]{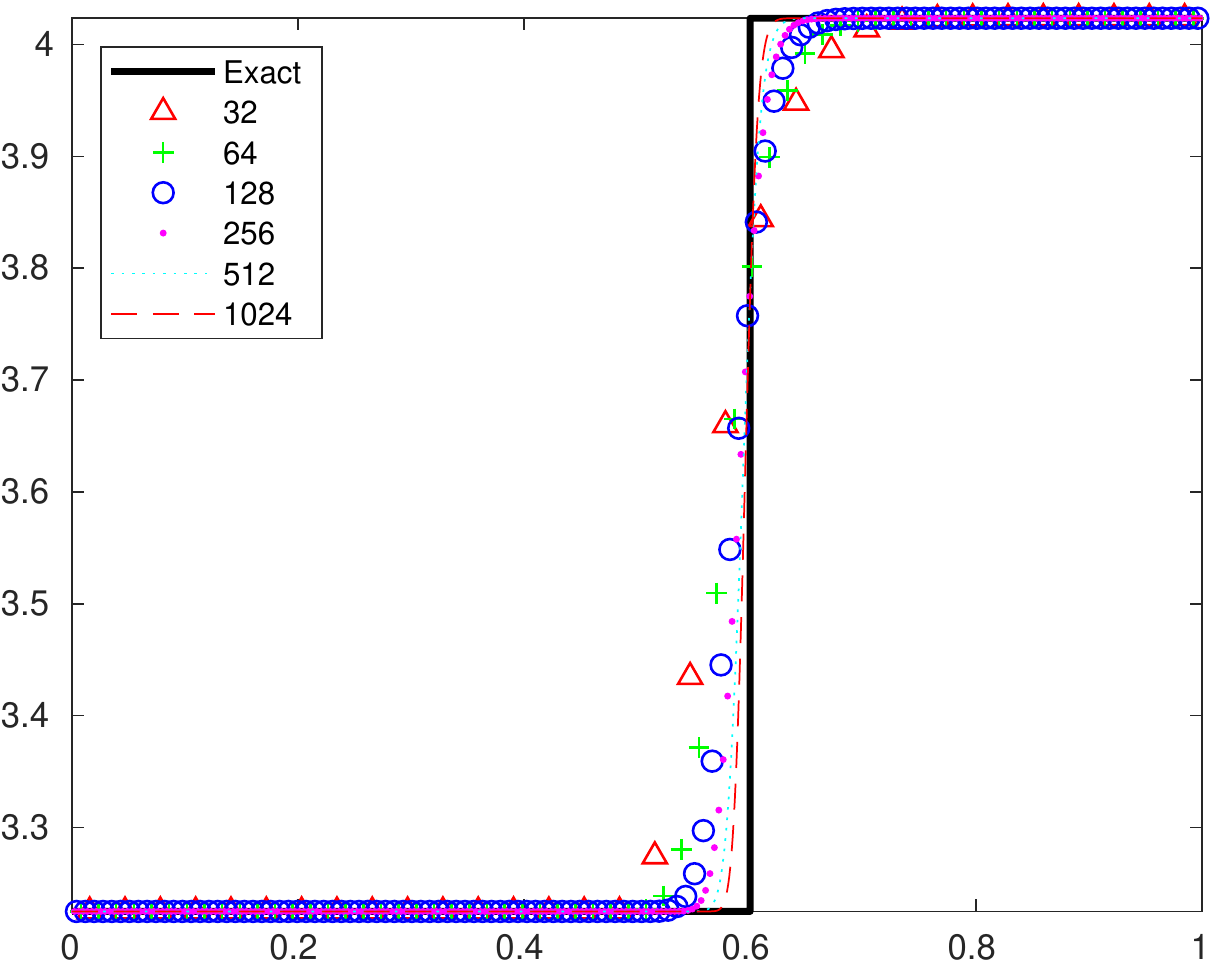}
		\caption{ Godunov - Contact }
	\end{subfigure}	
	\begin{subfigure}{0.32\textwidth}
		\includegraphics[width=\textwidth]{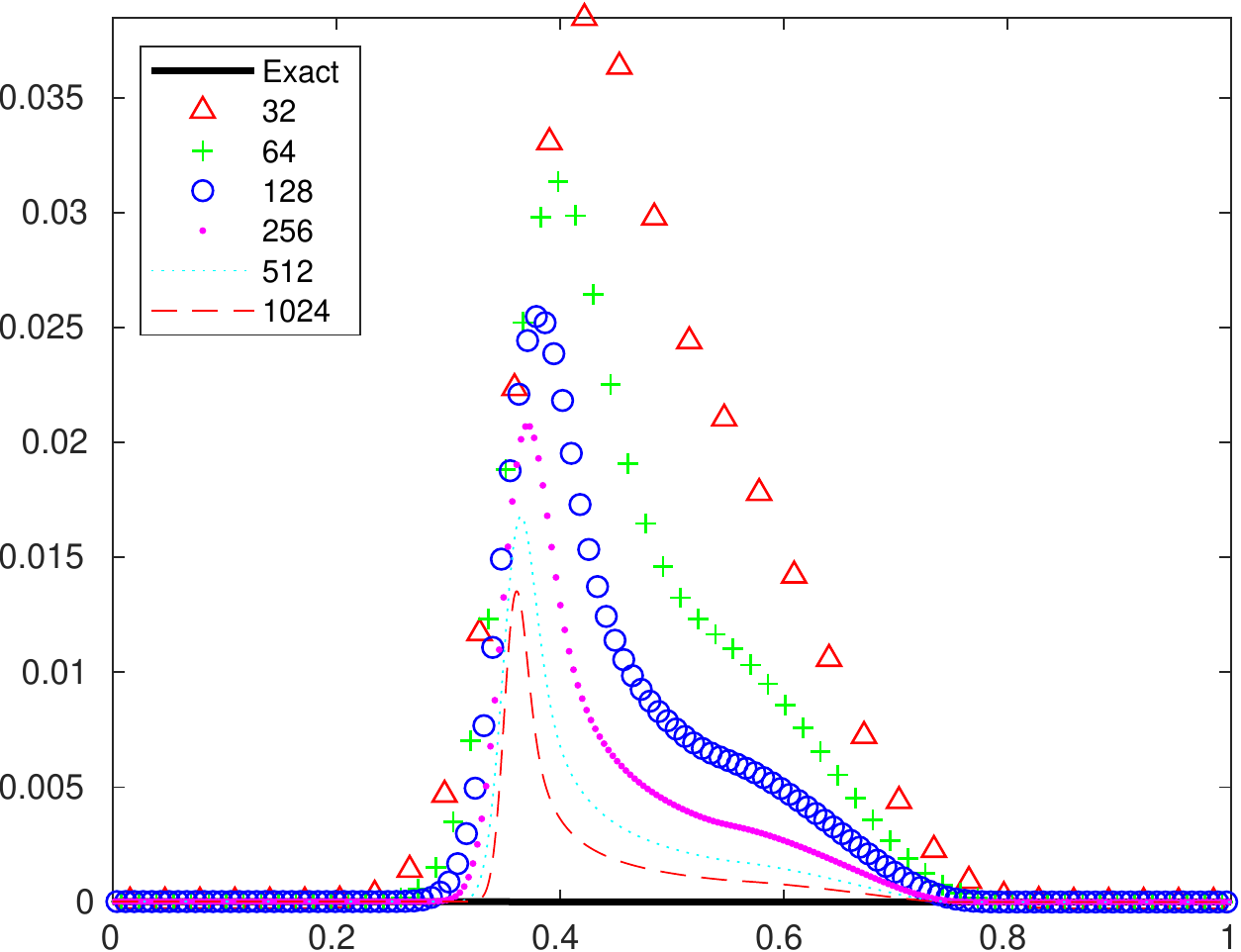}
		\caption{ Godunov - Rarefaction}
	\end{subfigure}	
	\begin{subfigure}{0.32\textwidth}
		\includegraphics[width=\textwidth]{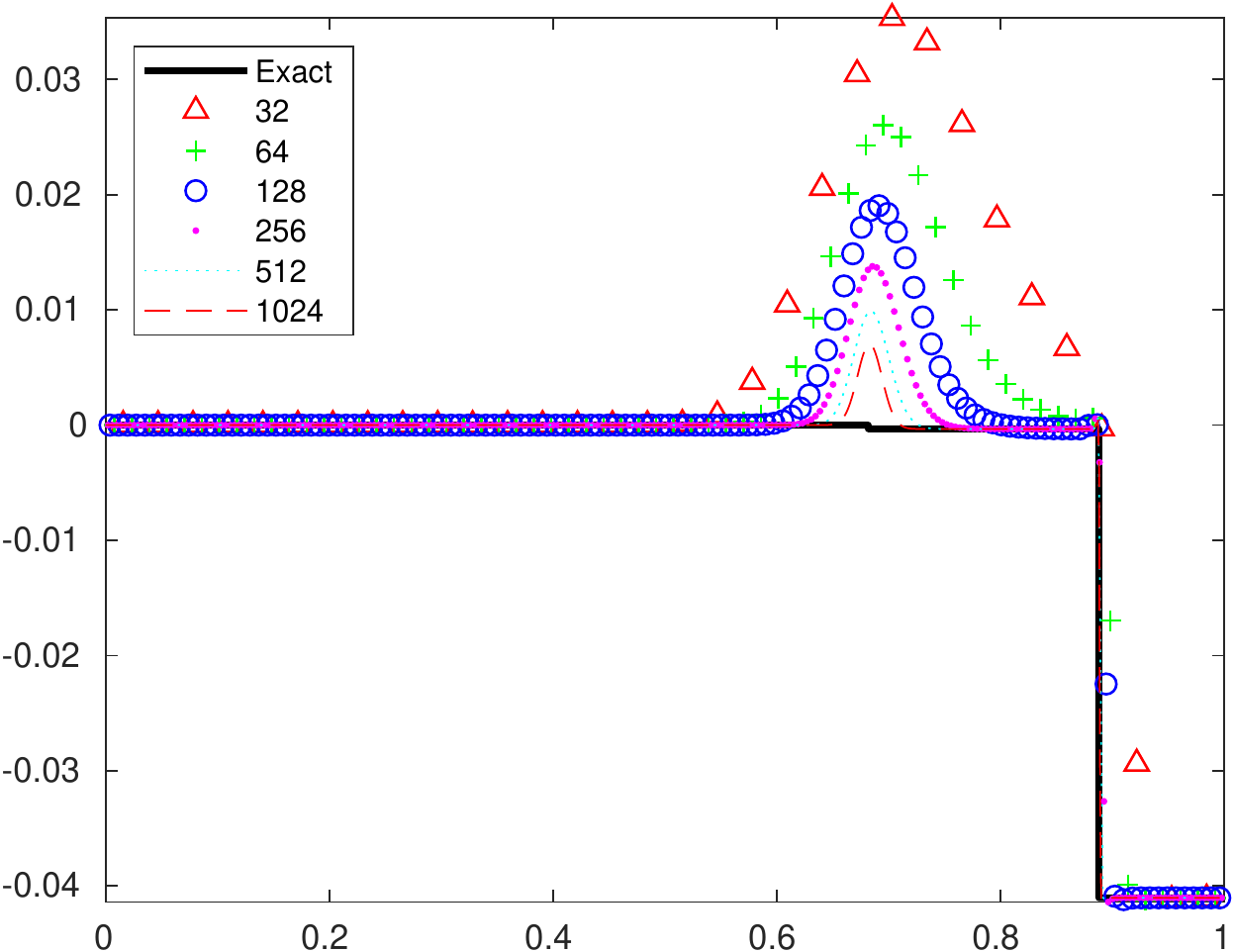}
		\caption{Godunov - Shock}
	\end{subfigure}	\\
	\begin{subfigure}{0.32\textwidth}
		\includegraphics[width=\textwidth]{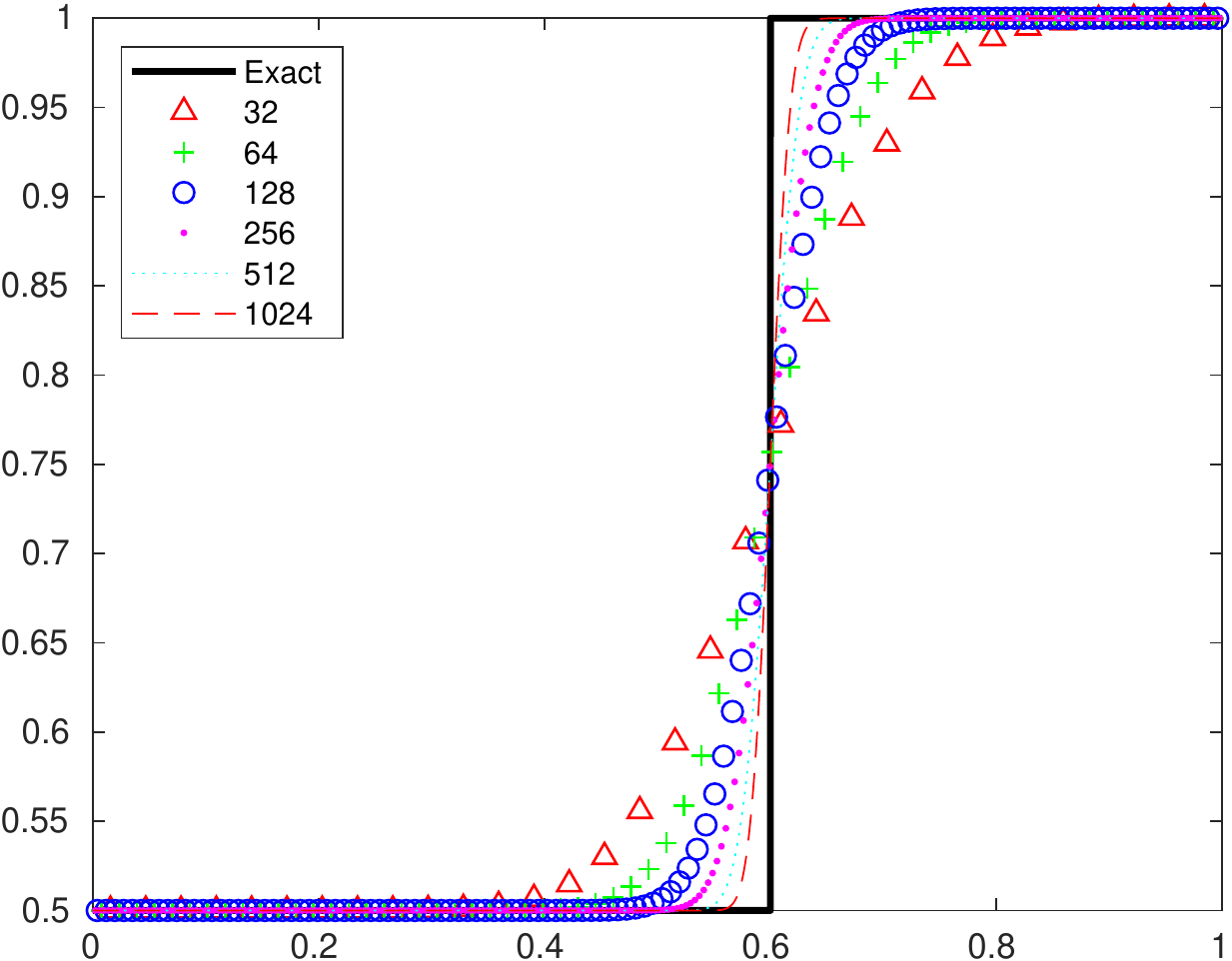}
		\caption{ VFV - Contact }
	\end{subfigure}	
	\begin{subfigure}{0.32\textwidth}
		\includegraphics[width=\textwidth]{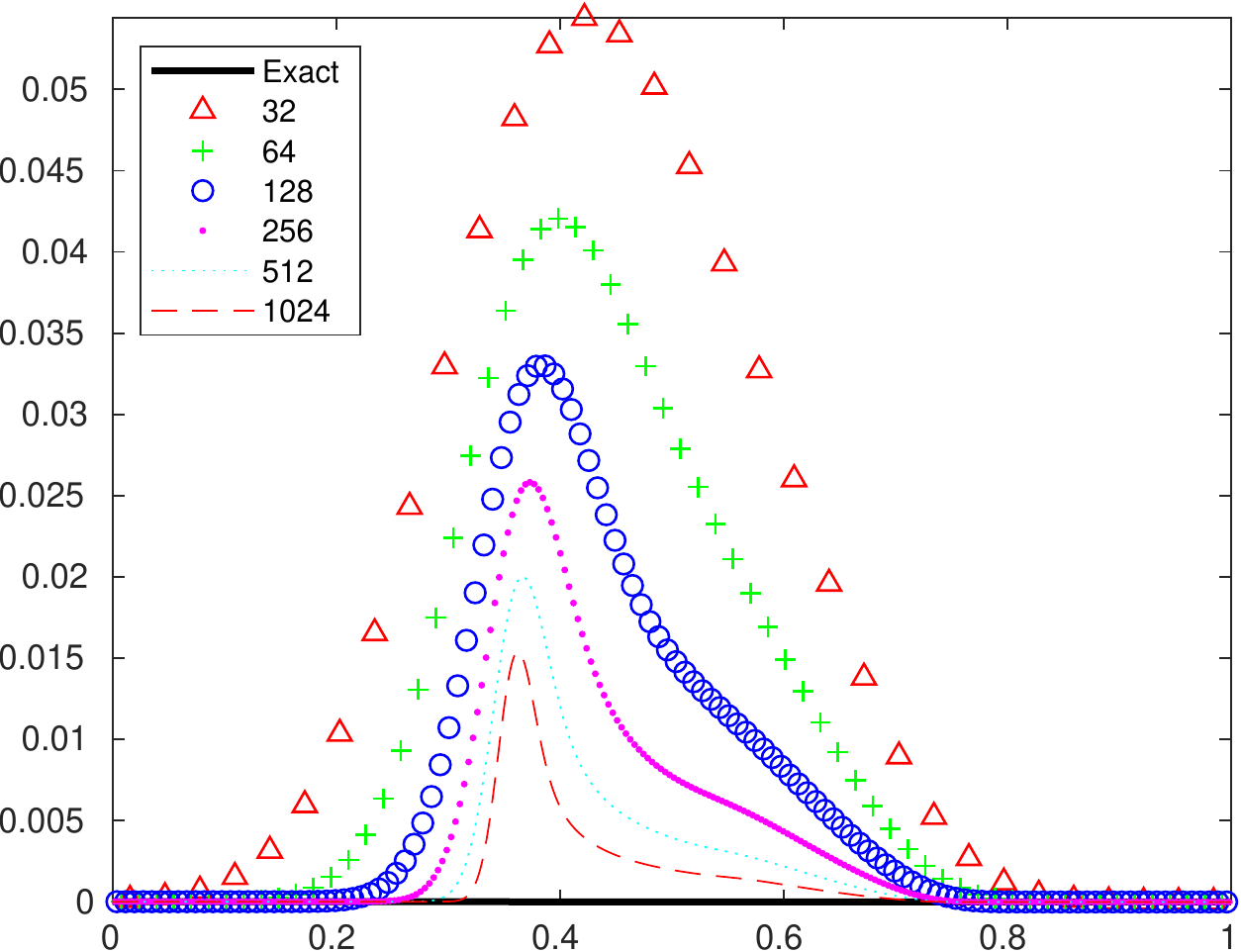}
		\caption{ VFV - Rarefaction}
	\end{subfigure}	
	\begin{subfigure}{0.32\textwidth}
		\includegraphics[width=\textwidth]{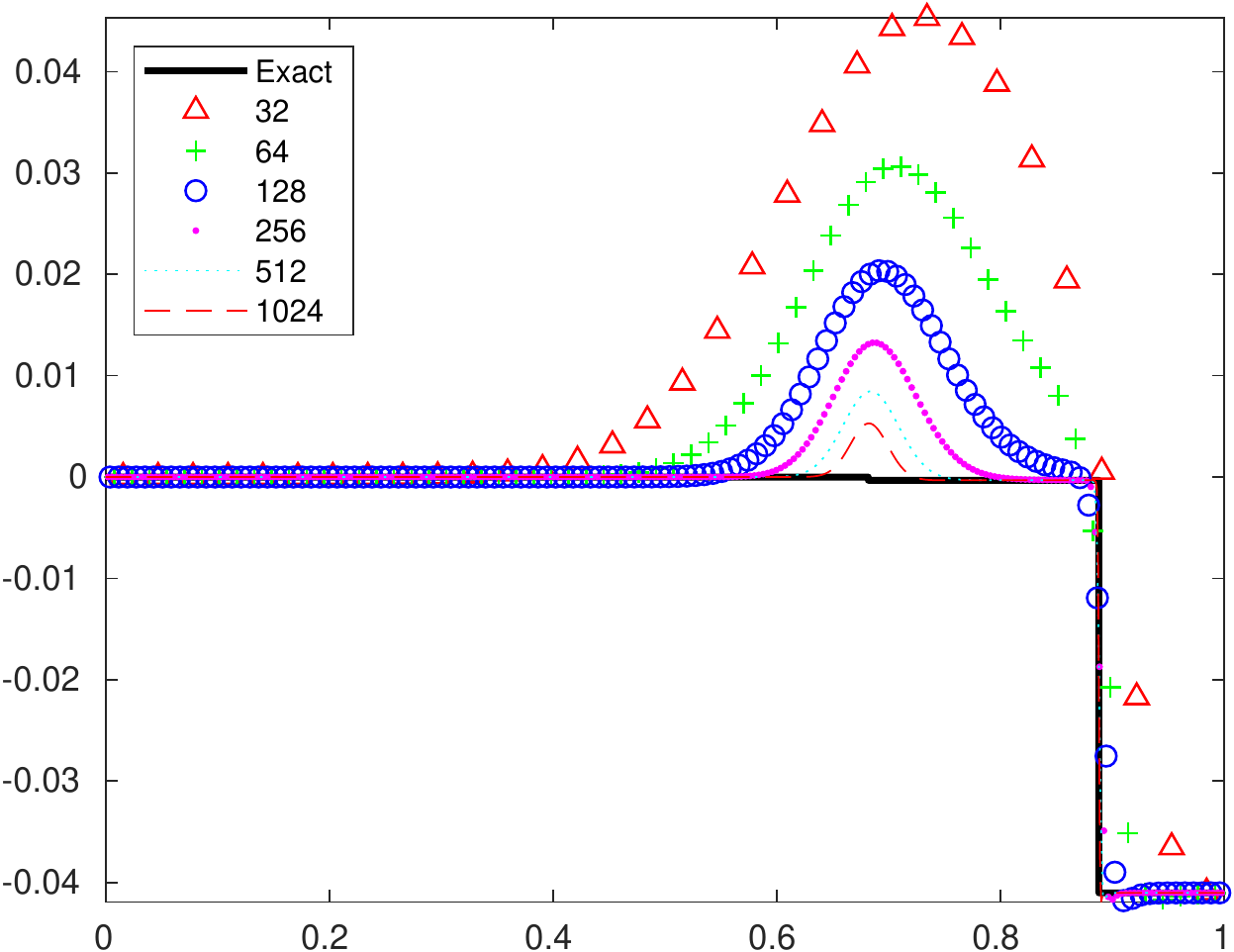}
		\caption{VFV - Shock}
	\end{subfigure}	
	\caption{\small{Example \ref{example:1D-singlewave}: entropy $\eta$ obtained by the Godunov method (top) and the VFV method (bottom).}}\label{figure:1D-singlewave-entropy}
\end{figure}

\begin{figure}[htbp]
	\setlength{\abovecaptionskip}{0.cm}
	\setlength{\belowcaptionskip}{-0.cm}
	\centering
	\begin{subfigure}{0.32\textwidth}
		\includegraphics[width=\textwidth]{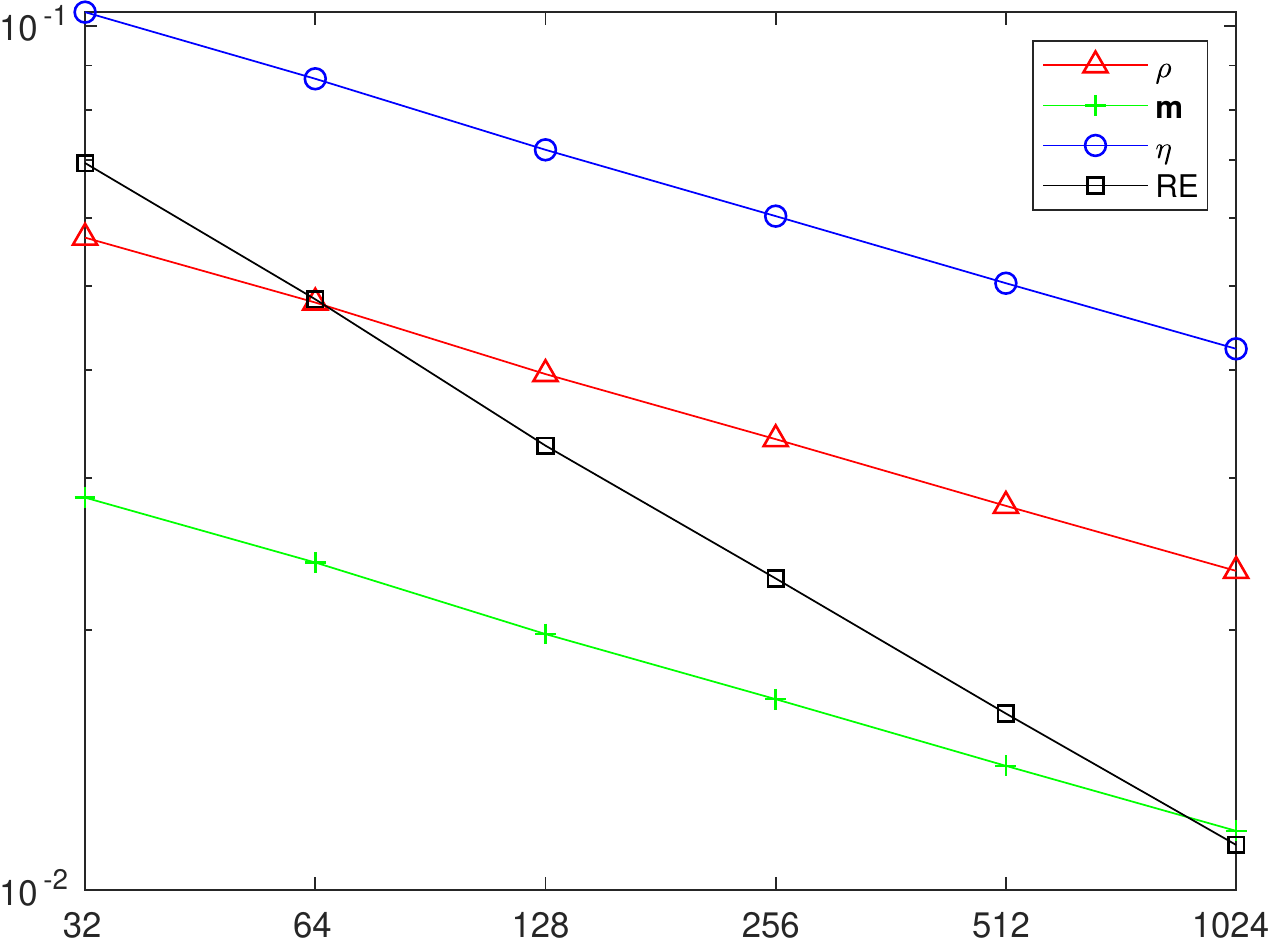}
		\caption{ Godunov - Contact }
	\end{subfigure}	
	\begin{subfigure}{0.32\textwidth}
		\includegraphics[width=\textwidth]{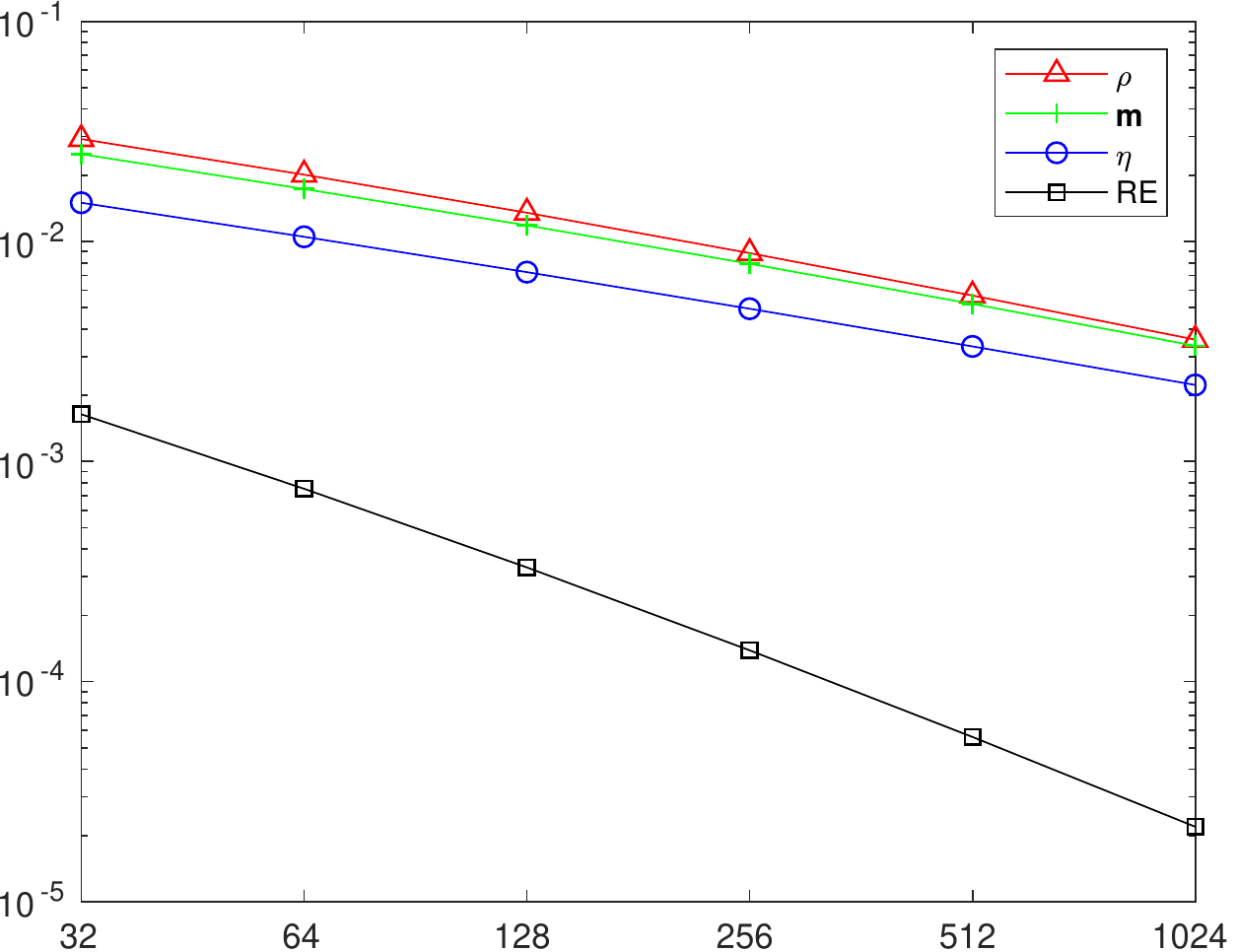}
		\caption{ Godunov - Rarefaction}
	\end{subfigure}	
	\begin{subfigure}{0.32\textwidth}
		\includegraphics[width=\textwidth]{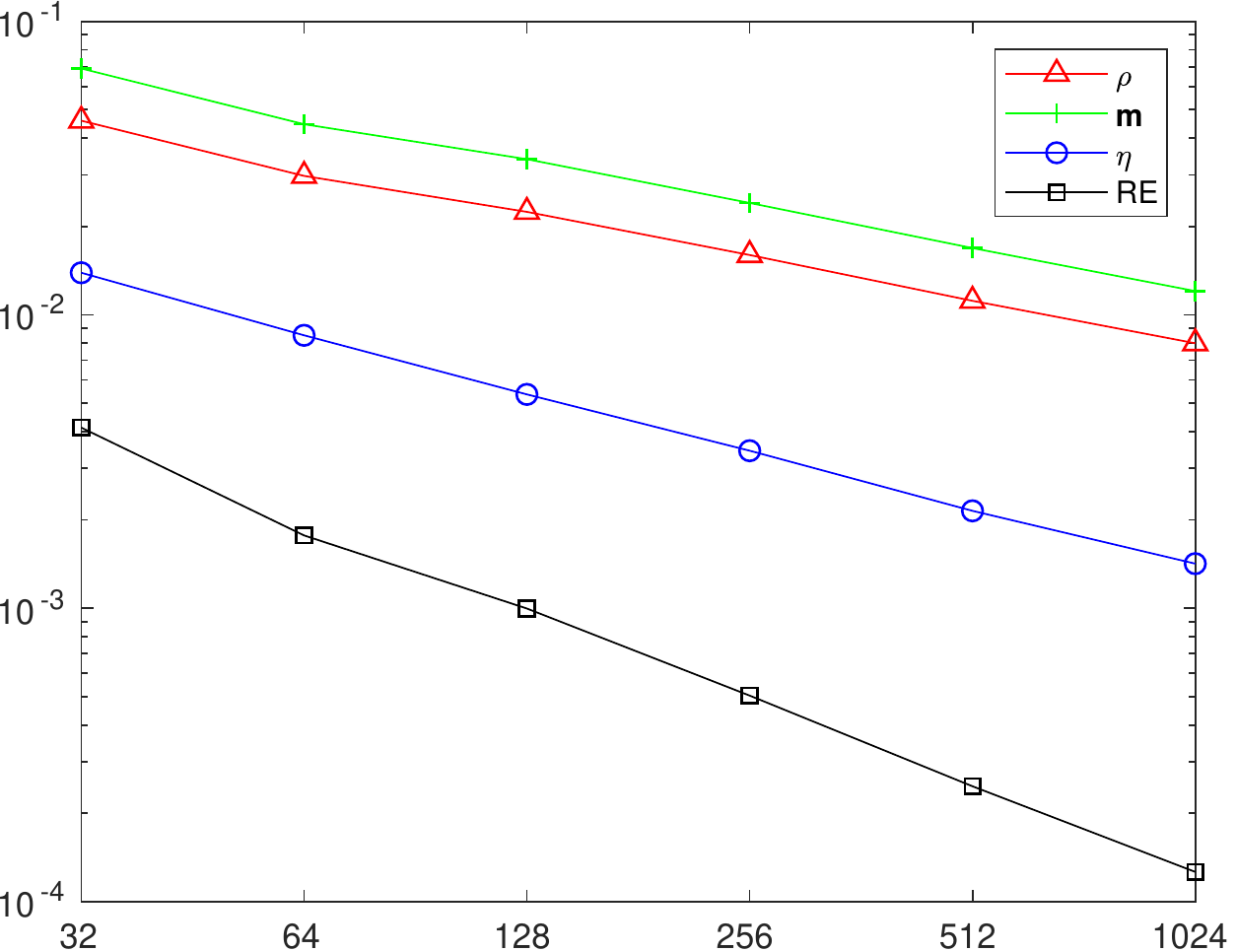}
		\caption{Godunov - Shock}
	\end{subfigure}	
	\begin{subfigure}{0.32\textwidth}
		\includegraphics[width=\textwidth]{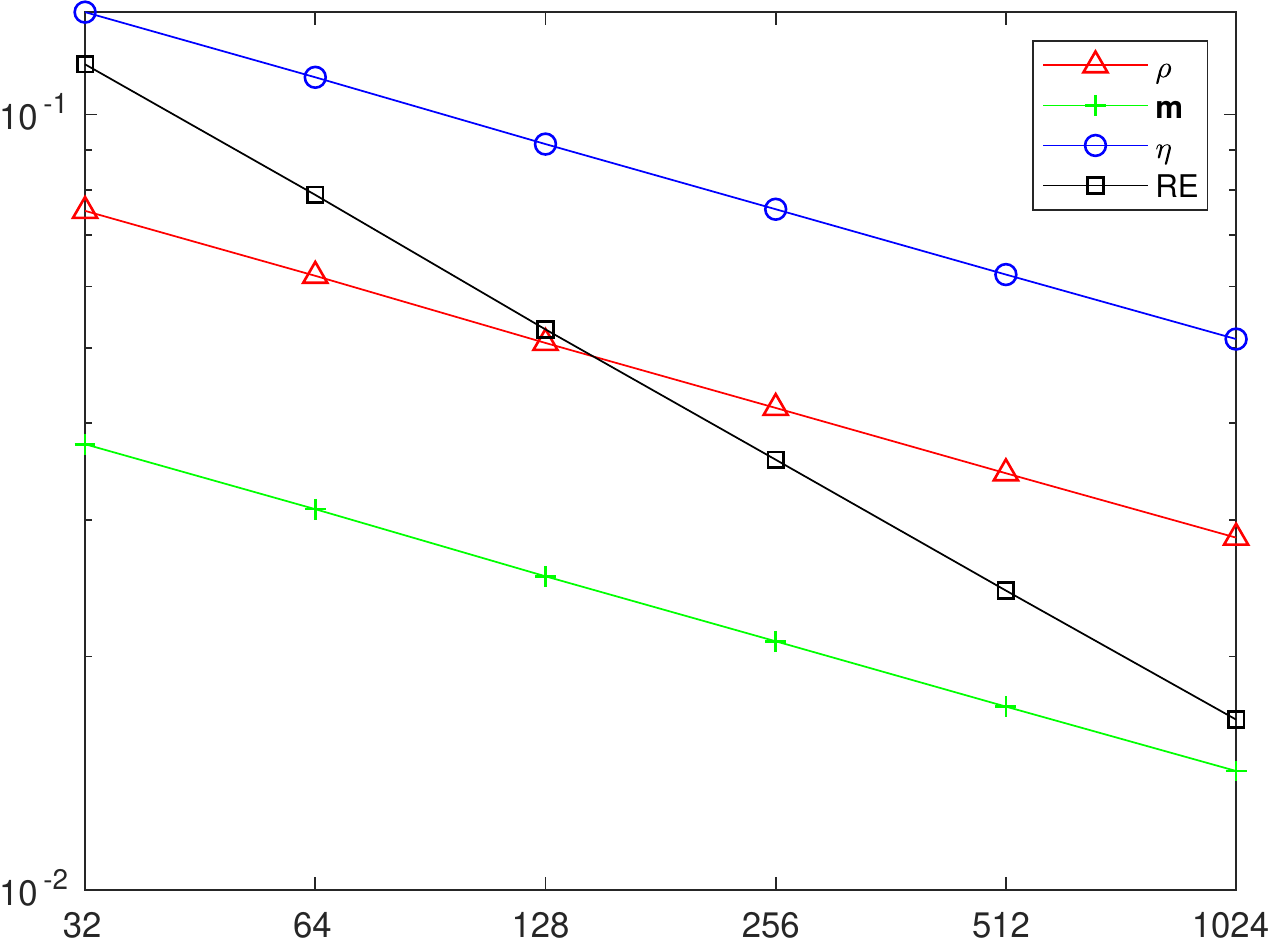}
		\caption{ VFV - Contact }
	\end{subfigure}	
	\begin{subfigure}{0.32\textwidth}
		\includegraphics[width=\textwidth]{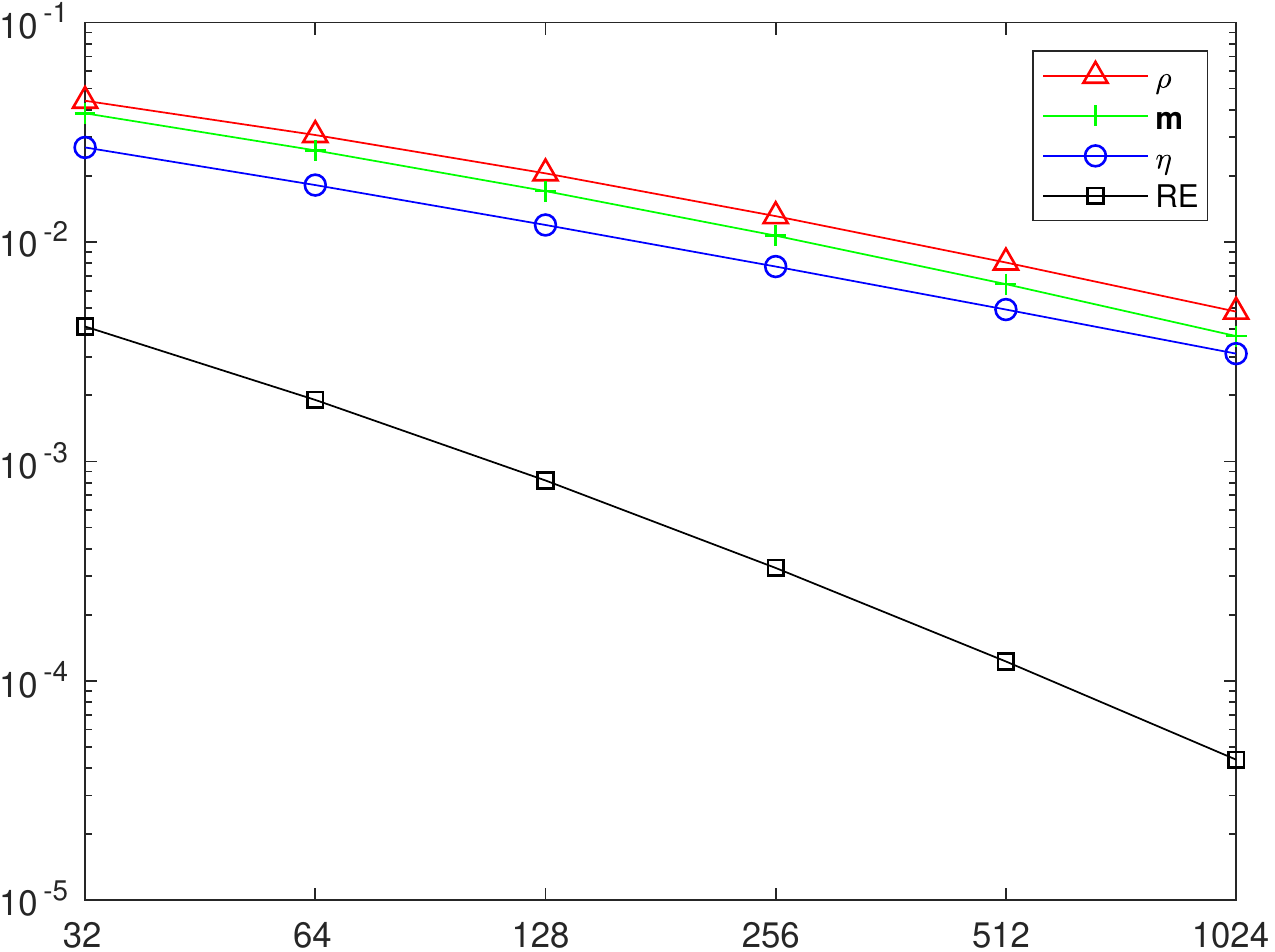}
		\caption{ VFV - Rarefaction}
	\end{subfigure}	
	\begin{subfigure}{0.32\textwidth}
		\includegraphics[width=\textwidth]{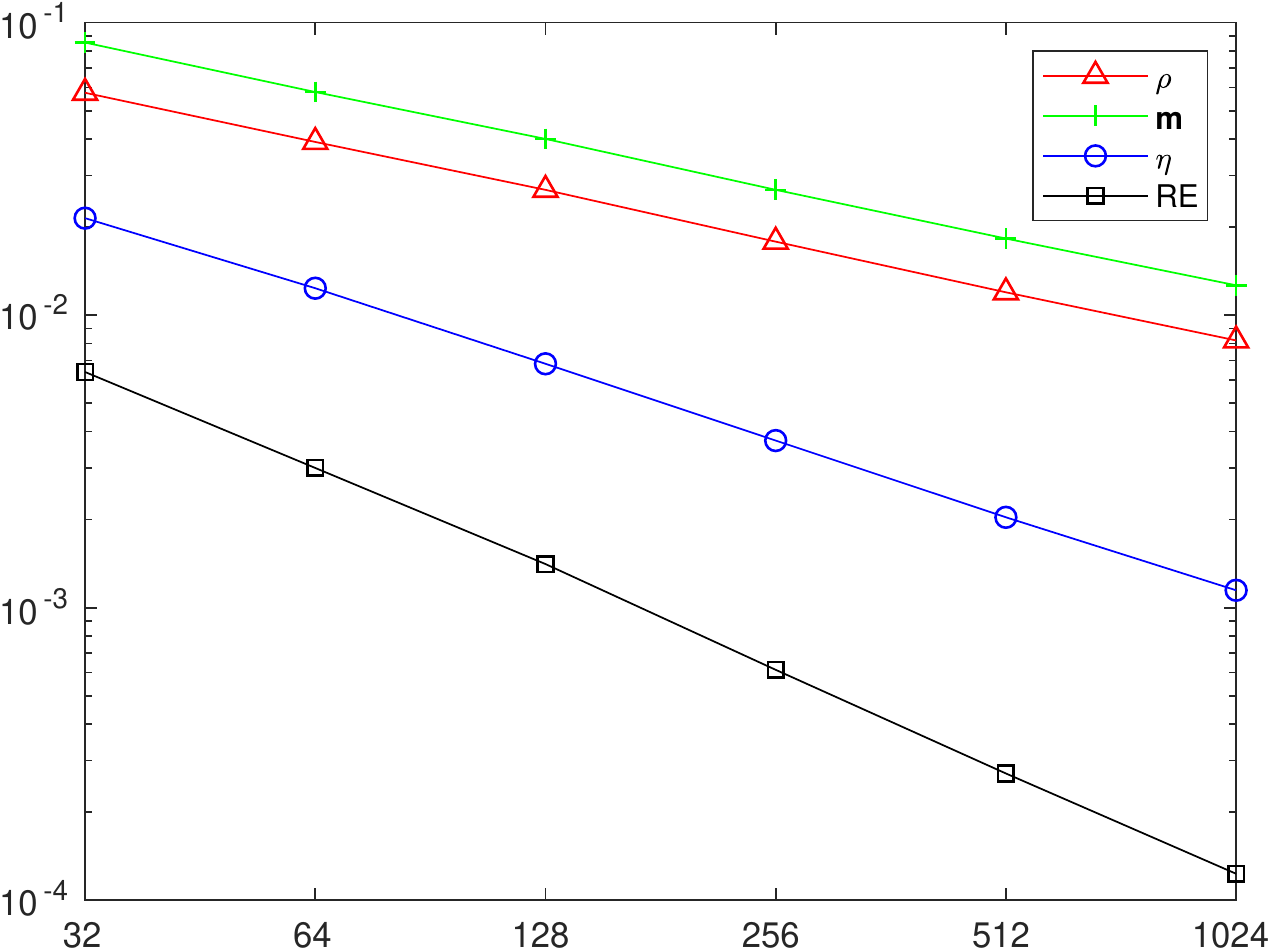}
		\caption{VFV - Shock}
	\end{subfigure}	
	\caption{\small{Example \ref{example:1D-singlewave}: the errors obtained on different meshes.}}\label{figure:1D-singlewave}
\end{figure}

\begin{table}[htbp]
	\centering
	%\caption{Example \ref{example:1D-singlewave}: $L^2$-error of $\vr,\eta$ and $L^1$-norm of $\E$. Godunov method.} \label{table:1D-singlewave-error}
	\caption{Example \ref{example:1D-singlewave}: errors and convergence rates of $\vr,\eta, \E$ of the Godunov method.} \label{table:1D-singlewave-error}
	\begin{tabular}{|c|cc|cc|cc|}
		\hline
		\multirow{2}{*}{$n=1/h$} & \multicolumn{2}{c|}{ \bf Contact } & \multicolumn{2}{c|}{ \bf Rarefaction } & \multicolumn{2}{c|}{ \bf Shock }  \\
		\cline{2-7}
		& error & order &  error & order  &  error & order  \\
		\hline
		\hline
		\multicolumn{7}{|c|}{ \bf{Density}} \\
		\hline
		\hline
32 & 0.0569 & - & 0.0292 & - & 0.0459 & -\\
64 & 0.0479 & 0.2497 & 0.0201 & 0.5380 & 0.0297 & 0.6252\\
128 & 0.0395 & 0.2753 & 0.0135 & 0.5743 & 0.0224 & 0.4072\\
256 & 0.0332 & 0.2504 & 0.0089 & 0.6098 & 0.0160 & 0.4886\\
512 & 0.0278 & 0.2565 & 0.0057 & 0.6398 & 0.0112 & 0.5183\\
1024 & 0.0234 & 0.2501 & 0.0036 & 0.6656 & 0.0080 & 0.4805\\
		\hline
		\hline
		\multicolumn{7}{|c|}{ \bf{Entropy}} \\
		\hline
		\hline
32 & 0.1038 & - & 0.0150 & - & 0.0139 & - \\
64 & 0.0869 & 0.2563 & 0.0105 & 0.5138 & 0.0085 & 0.7083\\
128 & 0.0719 & 0.2732 & 0.0073 & 0.5335 & 0.0054 & 0.6689\\
256 & 0.0603 & 0.2551 & 0.0049 & 0.5524 & 0.0034 & 0.6375\\
512 & 0.0504 & 0.2580 & 0.0033 & 0.5687 & 0.0021 & 0.6807\\
1024 & 0.0423 & 0.2522 & 0.0022 & 0.5818 & 0.0014 & 0.5978\\
		\hline
		\hline
		\multicolumn{7}{|c|}{ \bf{Relative energy}} \\
		\hline
		\hline
32 & 0.069415 & - & 0.001640 & - & 0.004126 & - \\
64 & 0.048272 & 0.5241 & 0.000752 & 1.1246 & 0.001771 & 1.2207\\
128 & 0.032676 & 0.5630 & 0.000330 & 1.1871 & 0.000998 & 0.8274\\
256 & 0.022931 & 0.5109 & 0.000139 & 1.2519 & 0.000504 & 0.9858\\
512 & 0.015997 & 0.5195 & 0.000056 & 1.3075 & 0.000247 & 1.0266\\
1024 & 0.011269 & 0.5054 & 0.000022 & 1.3554 & 0.000126 & 0.9697\\
		\hline
	\end{tabular}
\end{table}

\begin{table}[htbp]
	\centering
	\caption{Example \ref{example:1D-singlewave}: errors and convergence rates of $\vr,\eta, \E$ of the VFV method.} \label{table:1D-singlewave-error-1}
	\begin{tabular}{|c|cc|cc|cc|}
		\hline
		\multirow{2}{*}{$n=1/h$} & \multicolumn{2}{c|}{ \bf Contact } & \multicolumn{2}{c|}{ \bf Rarefaction } & \multicolumn{2}{c|}{ \bf Shock }  \\
		\cline{2-7}
		& error & order &  error & order  &  error & order  \\
		\hline
		\hline
		\multicolumn{7}{|c|}{ \bf{Density}} \\
		\hline
		\hline
32 & 0.0751 & - &  0.0440 & - & 0.0575 & -\\
64 & 0.0619 & 0.2784 & 0.0307 & 0.5185 & 0.0391 & 0.5584\\
128 & 0.0507 & 0.2877 & 0.0205 & 0.5805 & 0.0268 & 0.5440\\
256 & 0.0418 & 0.2784 & 0.0131 & 0.6479 & 0.0178 & 0.5882\\
512 & 0.0345 & 0.2799 & 0.0081 & 0.7021 & 0.0120 & 0.5752\\
1024 & 0.0285 & 0.2759 & 0.0048 & 0.7436 & 0.0082 & 0.5465\\
		\hline
		\hline
		\multicolumn{7}{|c|}{ \bf{Entropy}} \\
		\hline
		\hline
32 & 0.1356 & - & 0.0270 & - & 0.0214 & - \\
64 & 0.1117 & 0.2791 & 0.0182 & 0.5696 & 0.0124 & 0.7956 \\
128 & 0.0916 & 0.2862 & 0.0120 & 0.6043 & 0.0068 & 0.8582 \\
256 & 0.0755 & 0.2792 & 0.0077 & 0.6295 & 0.0037 & 0.8719 \\
512 & 0.0622 & 0.2799 & 0.0049 & 0.6510 & 0.0020 & 0.8713 \\
1024 & 0.0513 & 0.2764 & 0.0031 & 0.6675 & 0.0011 & 0.8287 \\
		\hline
		\hline
		\multicolumn{7}{|c|}{ \bf{Relative energy}} \\
		\hline
		\hline
32 &  0.116107 & -  & 0.004119 & -  & 0.006390 & -  \\
64 &  0.078778 & 0.5596 & 0.001906 & 1.1118 & 0.003002 & 1.0899 \\
128 &  0.052828 & 0.5765 & 0.000818 & 1.2196 & 0.001409 & 1.0911 \\
256 &  0.035875 & 0.5583 & 0.000327 & 1.3248 & 0.000613 & 1.2011 \\
512 &  0.024321 & 0.5608 & 0.000122 & 1.4157 & 0.000271 & 1.1768 \\
1024 &  0.016577 & 0.5530 & 0.000044 & 1.4865 & 0.000123 & 1.1391\\
		\hline
	\end{tabular}
\end{table}

%\newpage
\begin{Example} \label{example:1D-RP-2} \rm 
This experiment is used to further test our theoretical analysis. It describes left-going and right-going rarefaction waves, whose initial data are given by
\begin{equation*}
(\vr ,  u ,  p)(x,0)
\; = \; \begin{cases}
(1 ,\, -2 ,\, 0.4 ) , & x < 0.5, \\
(1 ,\, 2 ,\, 0.4 ) , & x > 0.5.
\end{cases}
\end{equation*}
Figure \ref{figure:1D-RP-2}(a) and (c) show the density $\vr$ obtained at $T = 0.15$ by the Godunov method and the VFV method, respectively. 
Moreover, the corresponding $L^2$-error of $(\vr, \vm, \eta)$ as well as the $L^1$-norm of $\E$ are shown in Figure~ \ref{figure:1D-RP-2}(b) and (d), see also Table \ref{table:1D-RP-2}. 

Our numerical results show that  the converge rate is approximately $1/2$ (resp. $1$) for $(\vr, \vm, \eta)$ (resp. $\E$), which is consistent with our theoretical analysis.
\end{Example}

\begin{figure}[htbp]
	\setlength{\abovecaptionskip}{0.cm}
	\setlength{\belowcaptionskip}{-0.cm}
	\centering
	\begin{subfigure}{0.45\textwidth}
		\includegraphics[width=\textwidth]{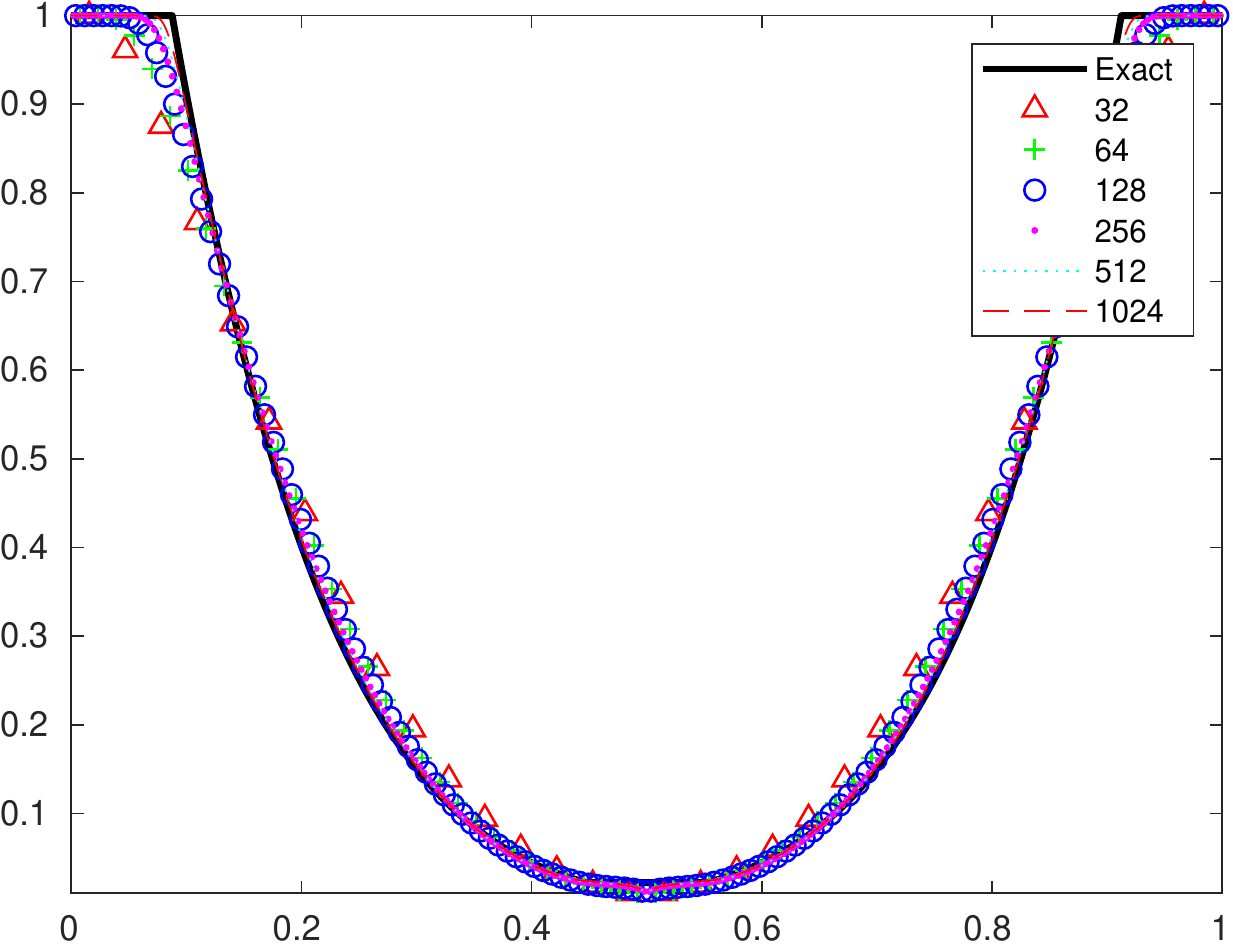}
		\caption{ Godunov - $\vr$}
	\end{subfigure}	
	\begin{subfigure}{0.46\textwidth}
		\includegraphics[width=\textwidth]{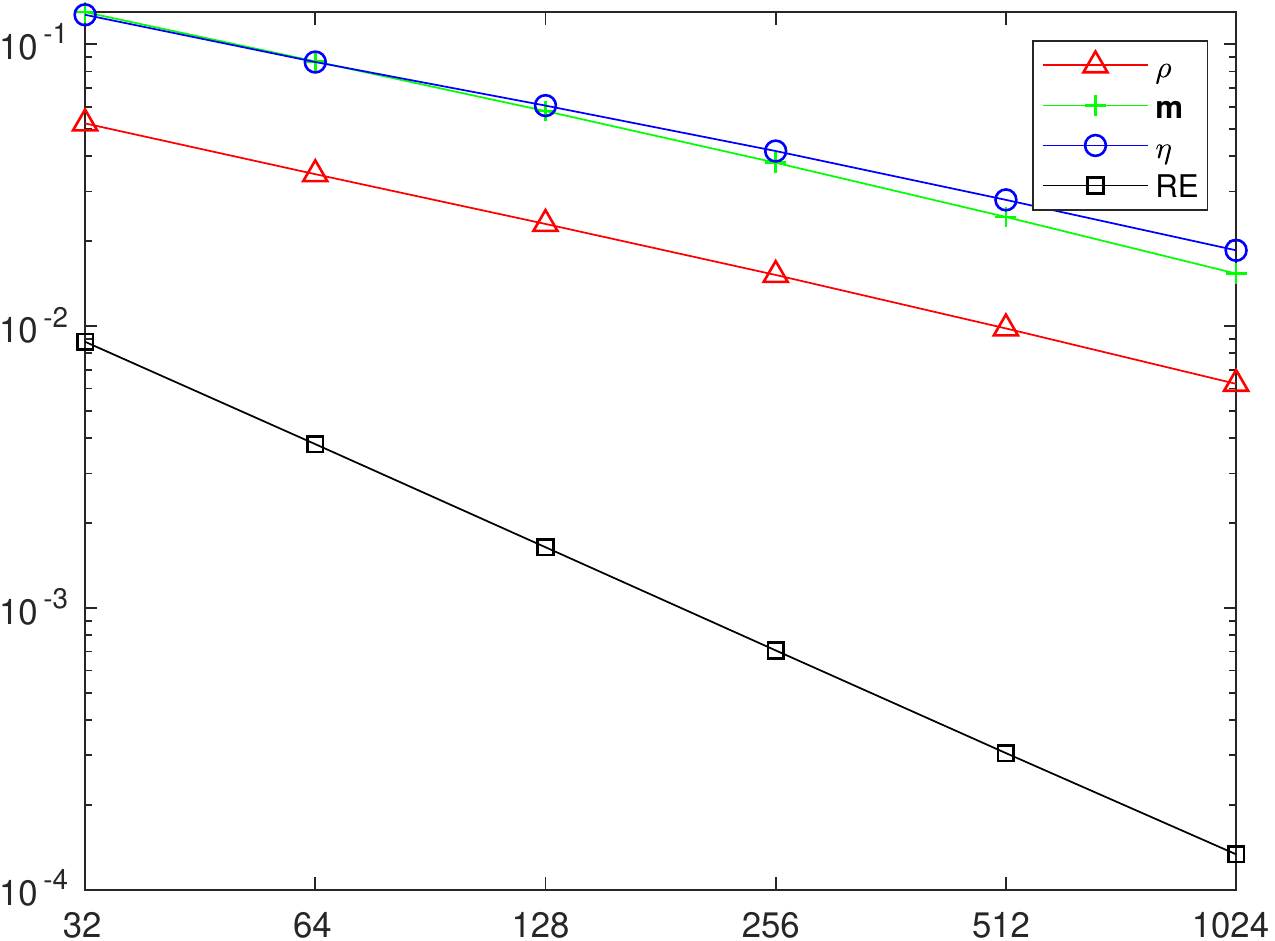}
		\caption{ Godunov - error}
	\end{subfigure}	\\
	\begin{subfigure}{0.45\textwidth}
		\includegraphics[width=\textwidth]{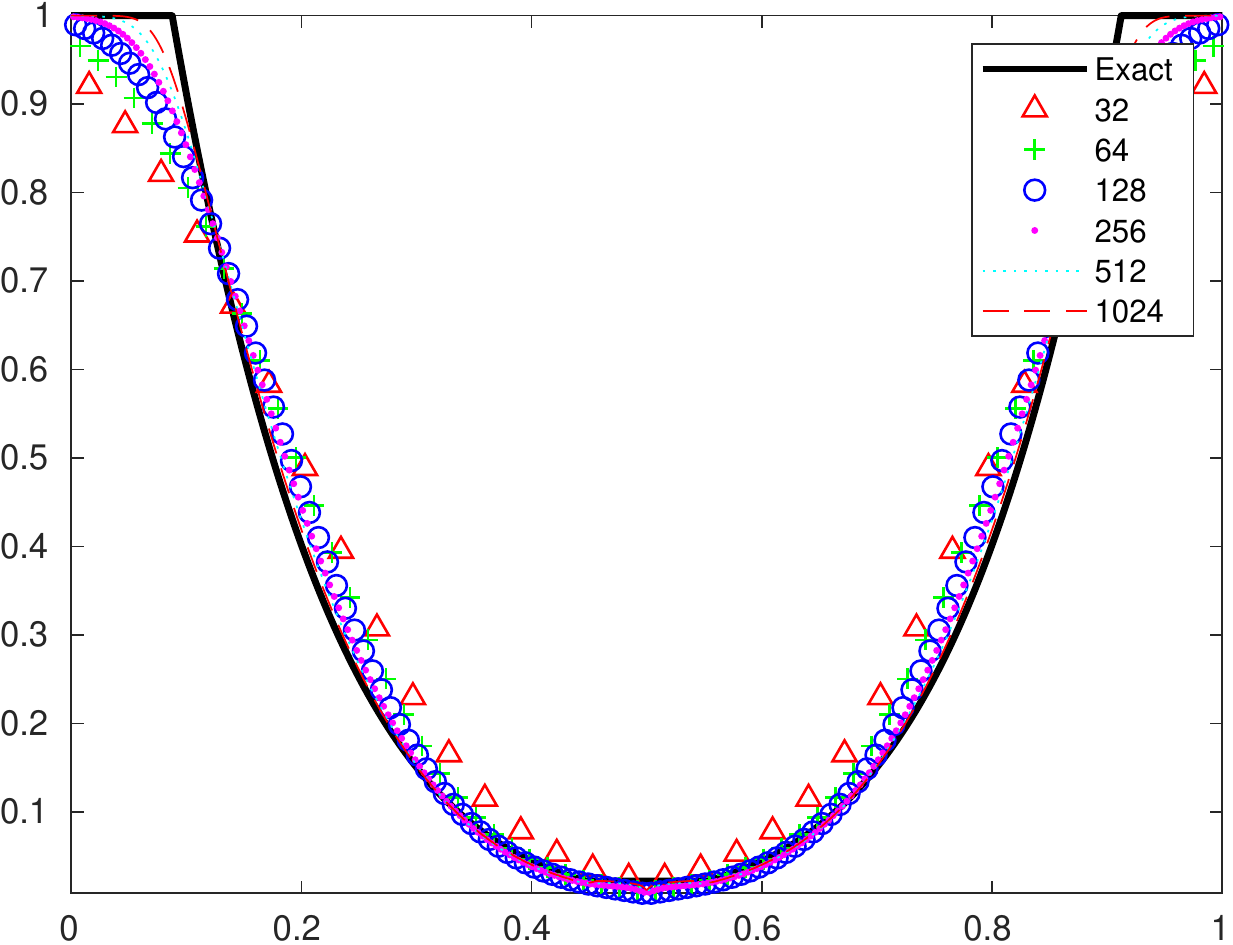}
		\caption{ VFV - $\vr$}
	\end{subfigure}	
	\begin{subfigure}{0.46\textwidth}
		\includegraphics[width=\textwidth]{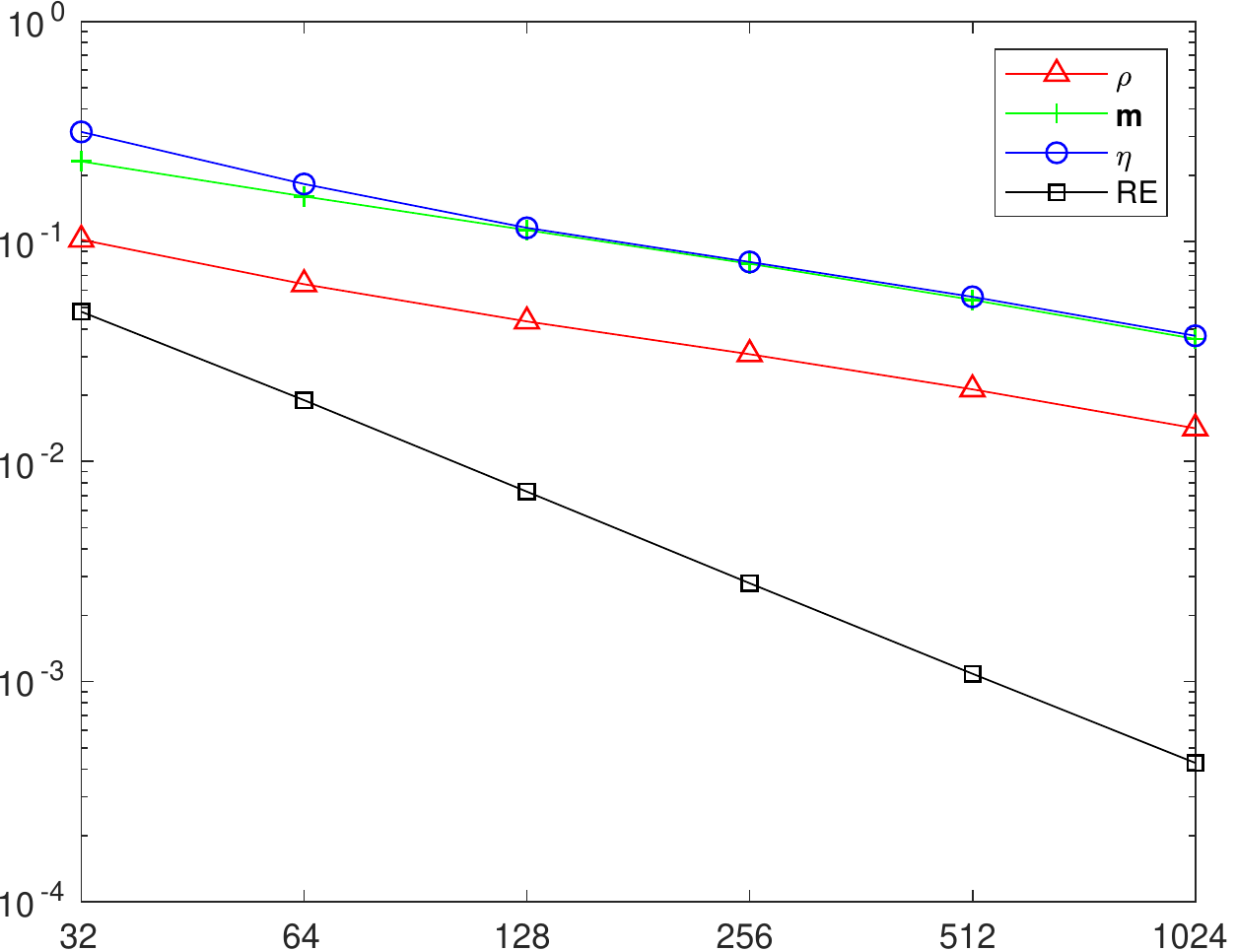}
		\caption{VFV - error}
	\end{subfigure}	
	\caption{\small{Example \ref{example:1D-RP-2}: density $\vr$ and errors at $T=0.15$ of the Godunov method and the VFV method.}}\label{figure:1D-RP-2}
\end{figure}

\begin{table}[htbp]
	\centering
	\caption{Example \ref{example:1D-RP-2}: errors and convergence rates of $\vr, \vm,\eta,\E$ of the Godunov and VFV methods.  } \label{table:1D-RP-2}
	\begin{tabular}{|c|cc|cc|cc|cc|}
		\hline
		\multirow{2}{*}{$n$} & \multicolumn{2}{c|}{ density } & \multicolumn{2}{c|}{ momentum } & \multicolumn{2}{c|}{ entropy } & \multicolumn{2}{c|}{ relative energy}  \\
		\cline{2-9}
		& error & order &  error & order & error & order &  error & order  \\
		\hline
		\hline
		\multicolumn{9}{|c|}{ \bf{ Godunov}} \\
		\hline
		\hline
32 & 0.0523 & - & 0.1299 & - & 0.1271 & - & 0.008792 & -   \\
64 & 0.0346 & 0.5987 & 0.0869 & 0.5803 & 0.0864 & 0.5566 & 0.003810 & 1.2064 \\
128 & 0.0230 & 0.5865 & 0.0579 & 0.5853 & 0.0605 & 0.5135 & 0.001641 & 1.2148 \\
256 & 0.0152 & 0.6012 & 0.0380 & 0.6090 & 0.0418 & 0.5354 & 0.000706 & 1.2162 \\
512 & 0.0098 & 0.6303 & 0.0244 & 0.6392 & 0.0280 & 0.5755 & 0.000305 & 1.2101 \\
1024 & 0.0062 & 0.6531 & 0.0153 & 0.6671 & 0.0186 & 0.5944 & 0.000134 & 1.1928 \\
		\hline
		\hline
		\multicolumn{9}{|c|}{ \bf{ VFV}} \\
		\hline
		\hline
32 & 0.1019 & - & 0.2310 & - & 0.3146 & - & 0.047945 & -  \\
64 & 0.0639 & 0.6723 & 0.1602 & 0.5279 & 0.1824 & 0.7866 & 0.019004 & 1.3350 \\
128 & 0.0433 & 0.5616 & 0.1126 & 0.5091 & 0.1153 & 0.6617 & 0.007298 & 1.3808 \\
256 & 0.0307 & 0.4950 & 0.0792 & 0.5072 & 0.0807 & 0.5151 & 0.002798 & 1.3830 \\
512 & 0.0213 & 0.5301 & 0.0543 & 0.5453 & 0.0560 & 0.5261 & 0.001086 & 1.3660 \\
1024 & 0.0142 & 0.5878 & 0.0361 & 0.5904 & 0.0373 & 0.5884 & 0.000427 & 1.3453 \\
		\hline
	\end{tabular}
\end{table}

%\newpage	
\begin{Example}\label{example:1D-Sod} \rm 
This experiment is devoted to the 1D Sod  problem, in order to test the convergence rate for the solution consisting of the left rarefaction, contact and right shock waves. 
Although the exact solution is not smooth we can still test corresponding convergence rates. 
In this example the final time is set to $T=0.15$ and the initial data are given by% that consists of a left-going rarefaction wave,  a right-going contact wave and a right-going shock wave. The initial data  read
\begin{equation*}
(\vr ,  u ,  p)(x,0)
\; = \; \begin{cases}
(1 ,\, 0 ,\, 1 ) , & x < 0.5, \\
(0.125 ,\, 0 ,\, 0.1 ) , & x > 0.5.
\end{cases}
\end{equation*}
Figure~\ref{figure:1D-Sod}(a) and (c) show the density obtained with the Godunov and VFV methods on different meshes. 
Moreover, errors of $(\vr,  \vm, \eta)$ and $\E$  are shown in Figure~\ref{figure:1D-Sod}(b) and (d), respectively, see also Table \ref{table:1D-Sod} for more details. 

These numerical results indicate that the convergence rates of $(\vr, \vm, \eta)$ (resp. $\E$) seem to be between $1/4$ and $1/2$ (resp. between $1/2$ and $1$). 	
\end{Example}

\begin{figure}[htbp]
	\setlength{\abovecaptionskip}{0.cm}
	\setlength{\belowcaptionskip}{-0.cm}
	\centering
	\begin{subfigure}{0.45\textwidth}
		\includegraphics[width=\textwidth]{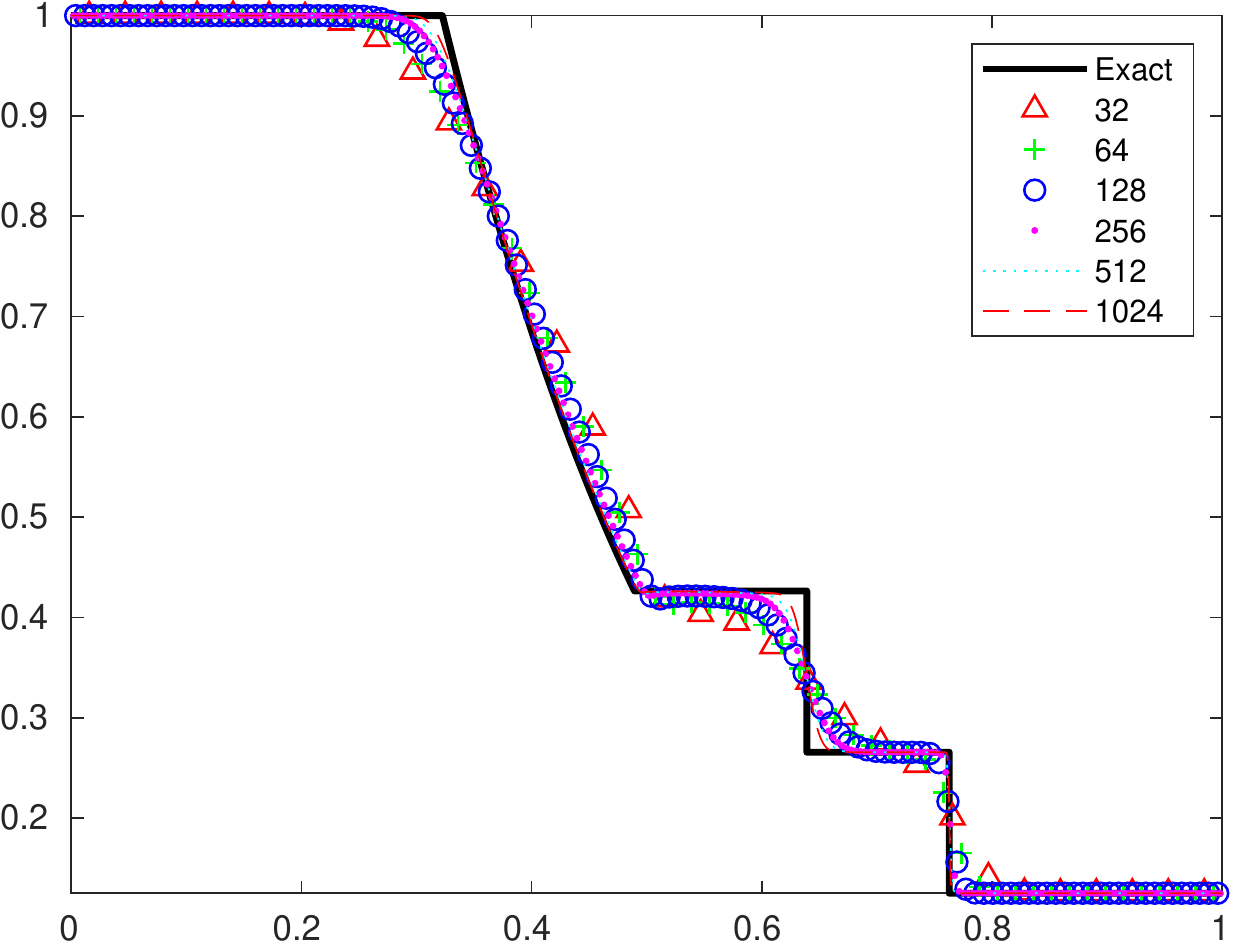}
		\caption{ Godunov - $\vr$}
	\end{subfigure}	
	\begin{subfigure}{0.46\textwidth}
		\includegraphics[width=\textwidth]{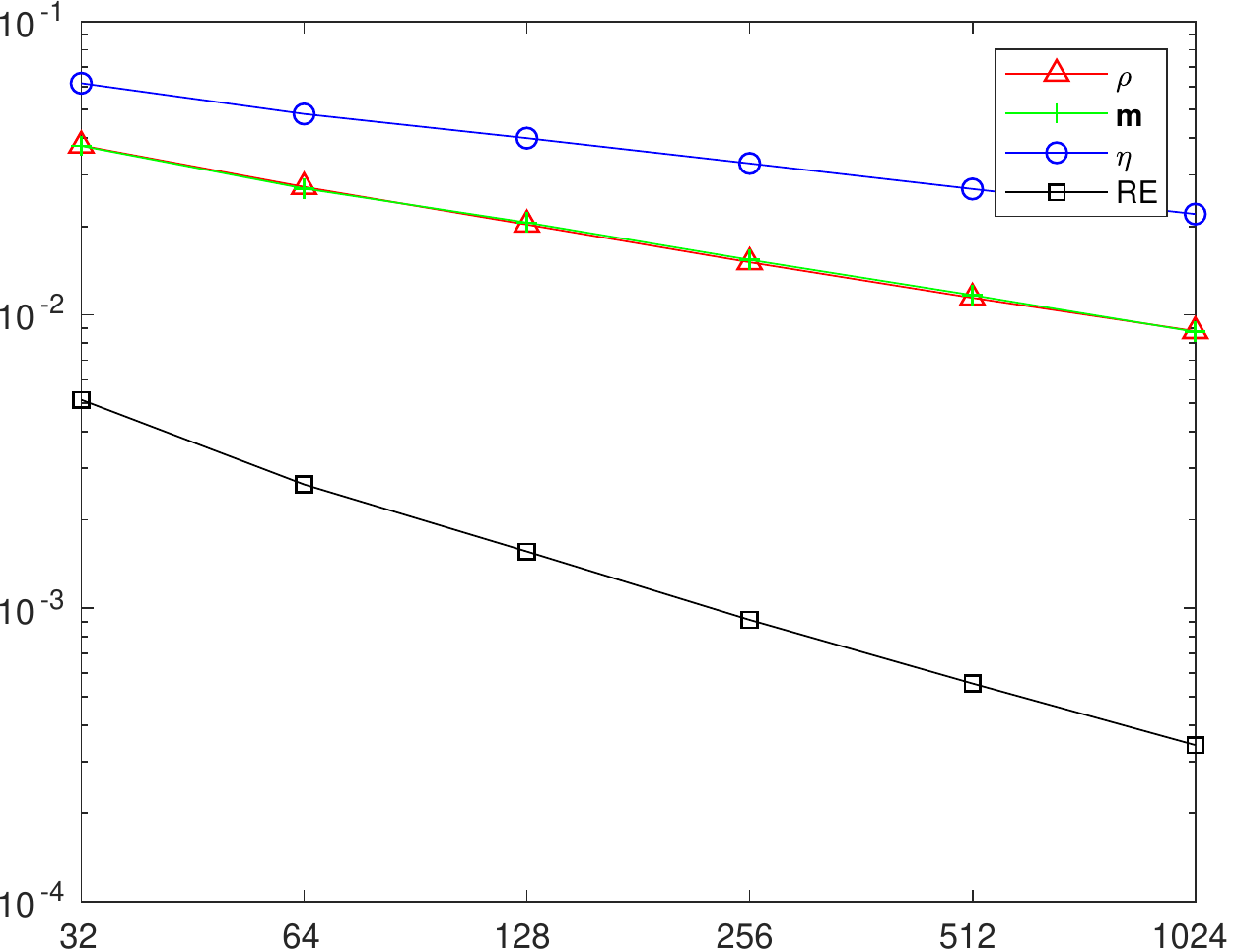}
		\caption{ Godunov - error}
	\end{subfigure}	\\
	\begin{subfigure}{0.45\textwidth}
		\includegraphics[width=\textwidth]{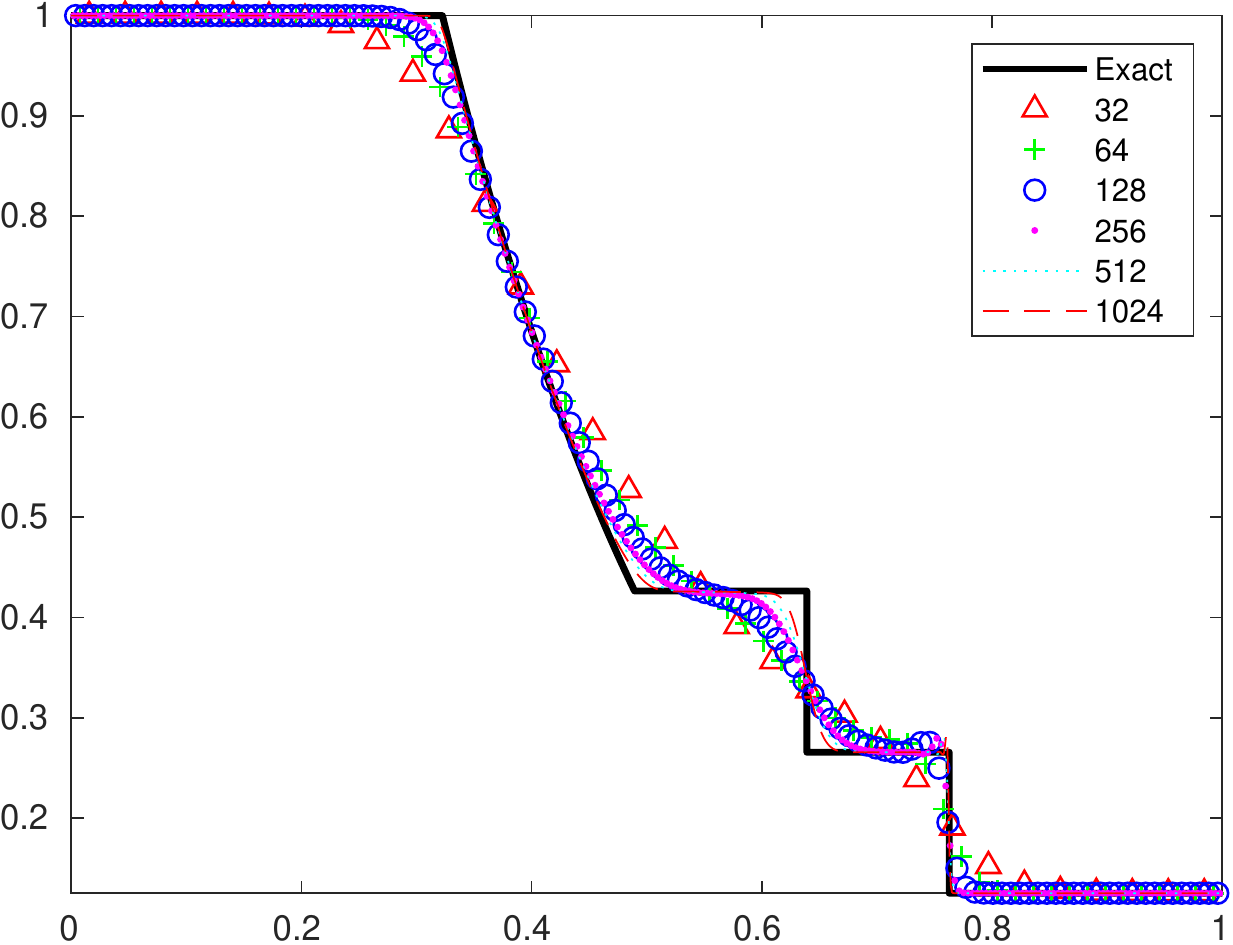}
		\caption{ VFV - $\vr$}
	\end{subfigure}	
	\begin{subfigure}{0.46\textwidth}
		\includegraphics[width=\textwidth]{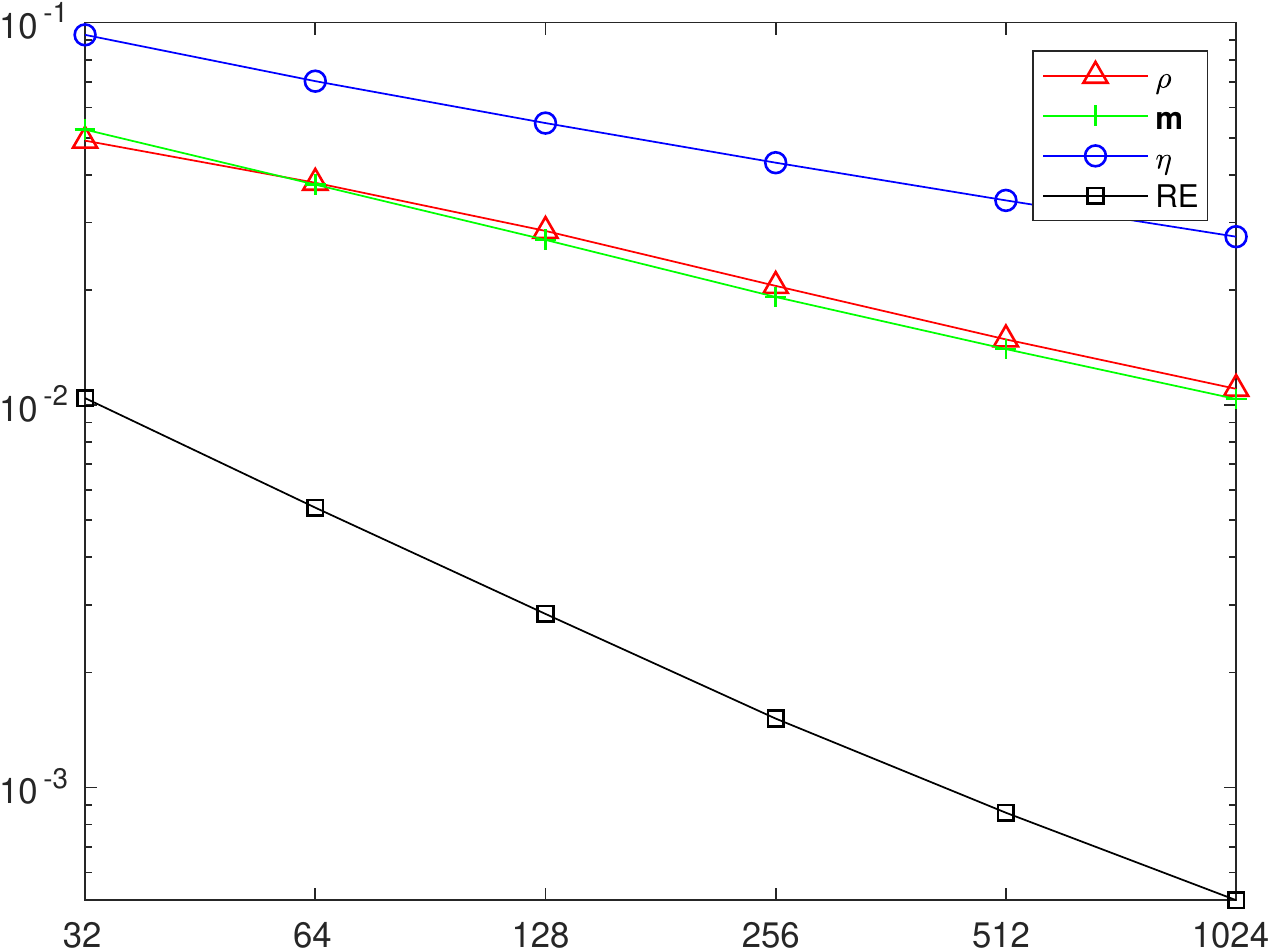}
		\caption{VFV - error}
	\end{subfigure}	
	\caption{\small{Example \ref{example:1D-Sod}: density $\vr$ and errors obtained on different meshes.}}\label{figure:1D-Sod}
\end{figure}

\begin{table}[htbp]
	\centering
	\caption{Example \ref{example:1D-Sod}: errors and convergence rates of $\vr, \vm,\eta,\E$ of the Godunov and VFV methods. } \label{table:1D-Sod}
	\begin{tabular}{|c|cc|cc|cc|cc|}
		\hline
		\multirow{2}{*}{$n$} & \multicolumn{2}{c|}{ density } & \multicolumn{2}{c|}{ momentum  } & \multicolumn{2}{c|}{ entropy } & \multicolumn{2}{c|}{ relative energy}  \\
		\cline{2-9}
		& error & order &  error & order & error & order &  error & order  \\
		\hline
		\hline
		\multicolumn{9}{|c|}{ \bf{ Godunov}} \\
		\hline
		\hline
32 & 0.0378 & - & 0.0376 & - & 0.0615 & - & 0.005135 & -   \\
64 & 0.0273 & 0.4693 & 0.0269 & 0.4819 & 0.0484 & 0.3481 & 0.002642 & 0.9587 \\
128 & 0.0203 & 0.4260 & 0.0206 & 0.3855 & 0.0400 & 0.2735 & 0.001561 & 0.7594 \\
256 & 0.0151 & 0.4268 & 0.0154 & 0.4217 & 0.0328 & 0.2865 & 0.000913 & 0.7741 \\
512 & 0.0114 & 0.4025 & 0.0117 & 0.4003 & 0.0268 & 0.2895 & 0.000554 & 0.7202 \\
1024 & 0.0088 & 0.3773 & 0.0087 & 0.4153 & 0.0221 & 0.2831 & 0.000342 & 0.6978 \\
		\hline
		\hline
		\multicolumn{9}{|c|}{ \bf{ VFV}} \\
		\hline
		\hline
32 & 0.0491 & - & 0.0525 & - & 0.0929 & - & 0.010427 & - \\
64 & 0.0381 & 0.3658 & 0.0377 & 0.4779 & 0.0703 & 0.4023 & 0.005392 & 0.9514 \\
128 & 0.0285 & 0.4183 & 0.0270 & 0.4802 & 0.0546 & 0.3641 & 0.002844 & 0.9231\\
256 & 0.0205 & 0.4774 & 0.0192 & 0.4958 & 0.0430 & 0.3441 & 0.001514 & 0.9090\\
512 & 0.0148 & 0.4650 & 0.0140 & 0.4524 & 0.0343 & 0.3274 & 0.000859 & 0.8176 \\
1024 & 0.0110 & 0.4293 & 0.0104 & 0.4337 & 0.0276 & 0.3153 & 0.000508 & 0.7595 \\
		\hline
	\end{tabular}
\end{table}

\subsection{Two dimensional experiments}\label{sec_exp:2D}
In this section we present four two-dimensional Riemann problems. The computational domain is taken as $[0,1]^2$. 
Here the exact solution $\tvU$ used in the relative energy is taken as the reference solution computed on the uniform mesh of $4096^2$ cells.

\begin{Example}\label{example:2D-RP-3}\rm 
The first 2D Riemann problem describes the interaction of four rarefaction waves. 
The initial data are given by
\begin{equation*}
(\vr ,  u , v, p)(x,0)
\; = \; \begin{cases}
(1 ,\, 0 ,\, 0 ,\, 1 ) , & x > 0.5,\,y>0.5, \\
(0.5197 ,\,  -0.7259,\, 0 , \, 0.4) , & x<0.5,\,y>0.5, \\
(1 ,\,  -0.7259,\, -0. 7259, \, 1) , & x<0.5,\,y<0.5, \\
(0.5197 ,\,  0,\, -0.7259 , \, 0.4) , & x>0.5,\,y<0.5. 
\end{cases}
\end{equation*}
In this example the final time is set to $T = 0.2$. 
%The output solutions is $T = 0.2$. 
%It describes the interaction of four rarefaction waves.
%As time increase those four initial discontinuities first evolve as four rarefaction waves and then interact each other and form two (almost parallel) curved shock waves perpendicular to the line $x = y$ as time increases. 
Figure \ref{figure:2D-RP-3}(a) and (c) show the density $\vr$ obtained by the Godunov and VFV method on a mesh with $1024^2$ cells. Moreover, Figure \ref{figure:2D-RP-3}(b) and (d) show the $L^2$-errors of $\vr, \vm, \eta$ and $L^1$-norm of $\E$ on different meshes, see also Table \ref{table:2D-RP-3}. 

The numerical results show that the convergence rates of $\vr, \vm, \eta$ (resp. $\E$) are slightly better than $1/2$ (resp. 1). 
This may indicate that our rigorous error estimates are suboptimal in the case of finitely many rarefaction waves.
%This may be because for this example the rarefaction waves dominate.
\end{Example}

\begin{figure}[htbp]
	\setlength{\abovecaptionskip}{0.cm}
	\setlength{\belowcaptionskip}{-0.cm}
	\centering
	\begin{subfigure}{0.4\textwidth}
		\includegraphics[width=\textwidth]{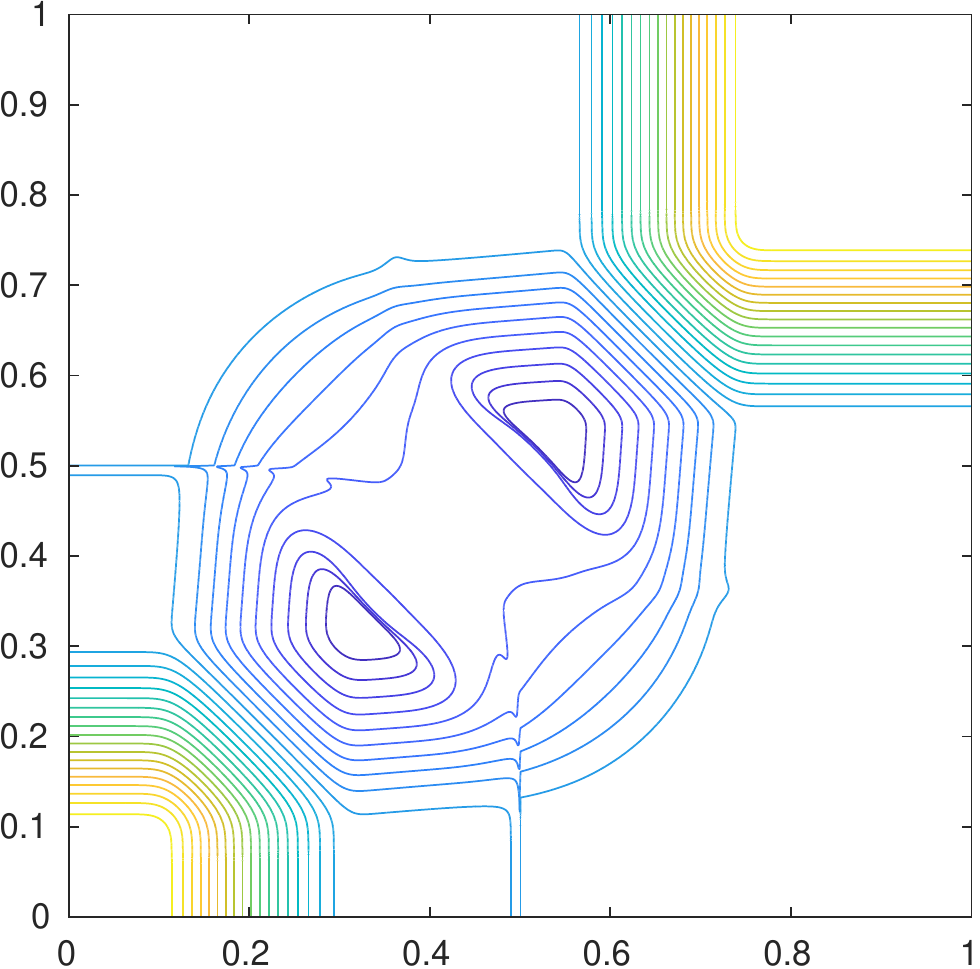}
		\caption{ \bf Godunov - $\vr$}
	\end{subfigure}	
	\begin{subfigure}{0.5\textwidth}
		\includegraphics[width=\textwidth]{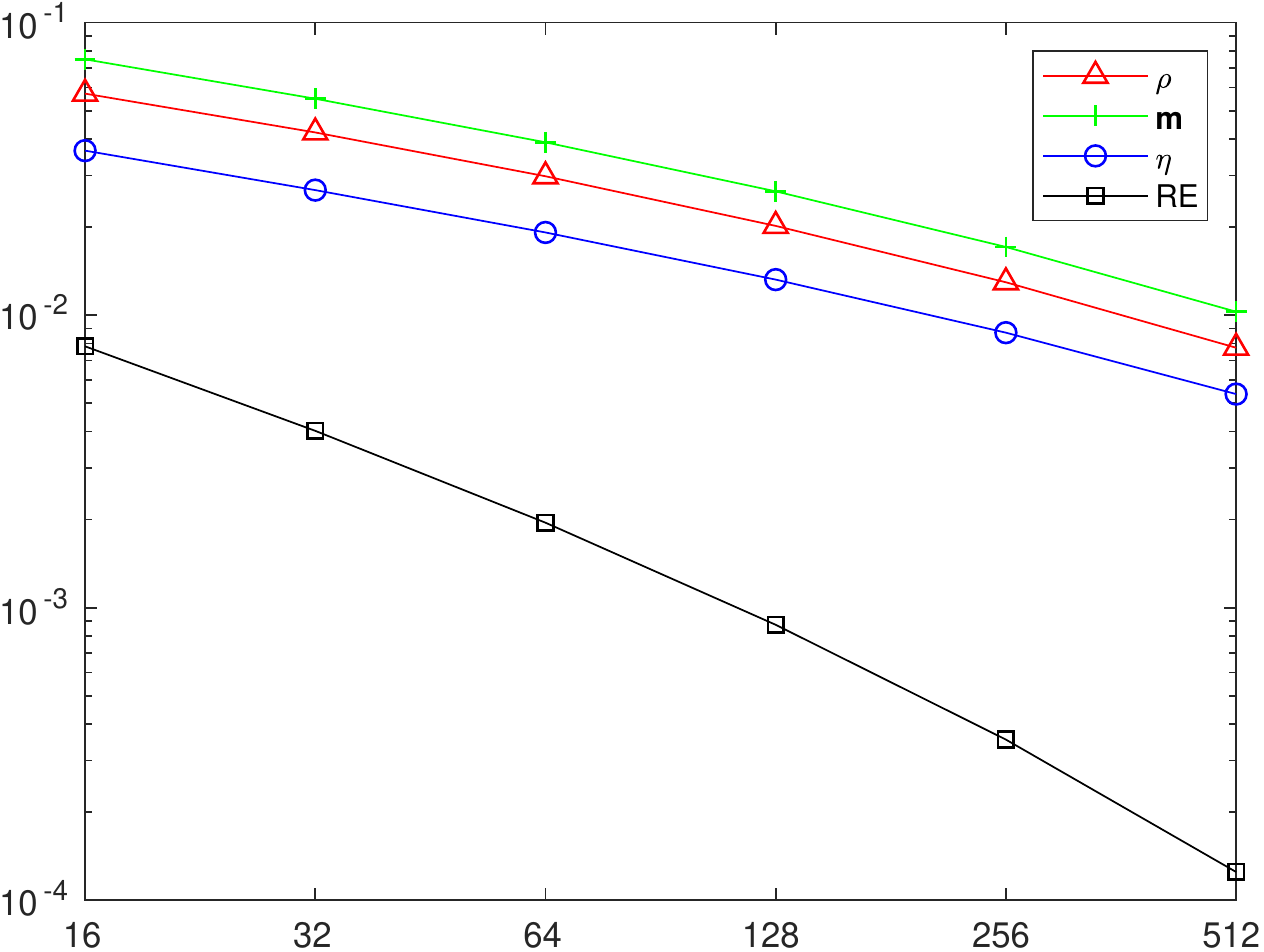}
		\caption{\bf Godunov - error}
	\end{subfigure}	\\
	\begin{subfigure}{0.4\textwidth}
		\includegraphics[width=\textwidth]{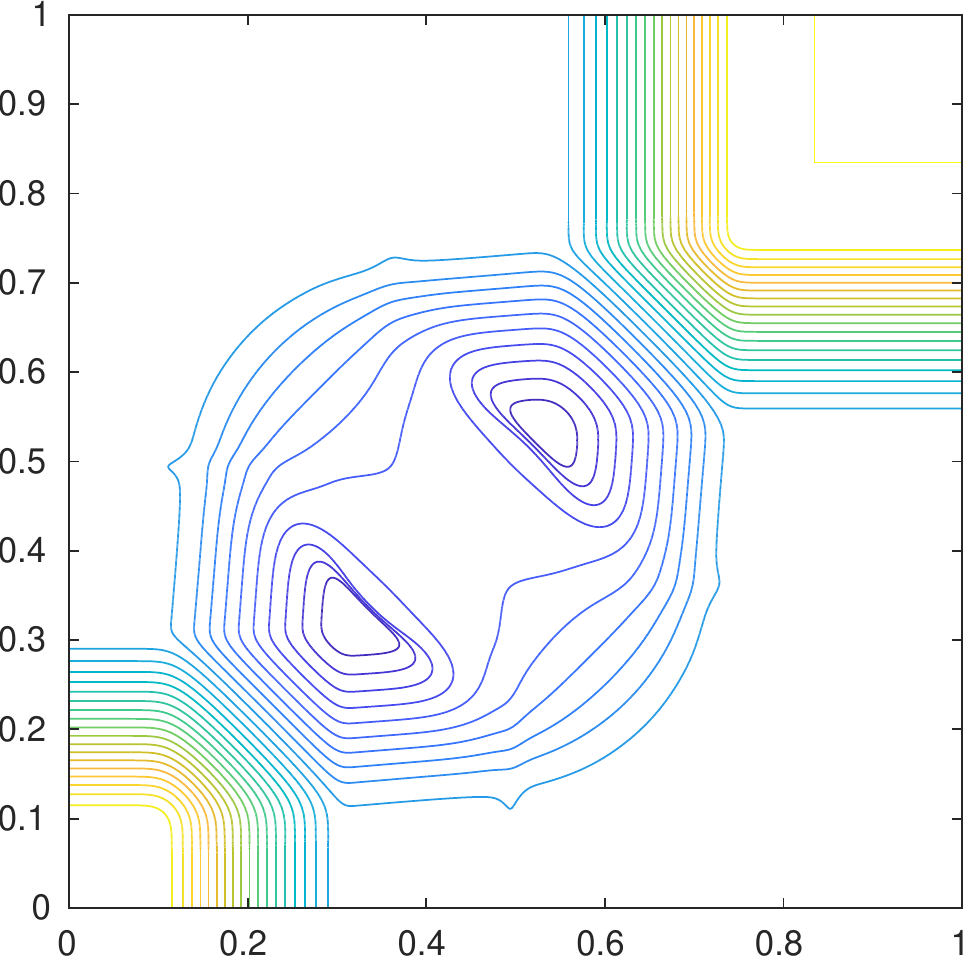}
		\caption{ \bf VFV - $\vr$}
	\end{subfigure}	
	\begin{subfigure}{0.5\textwidth}
		\includegraphics[width=\textwidth]{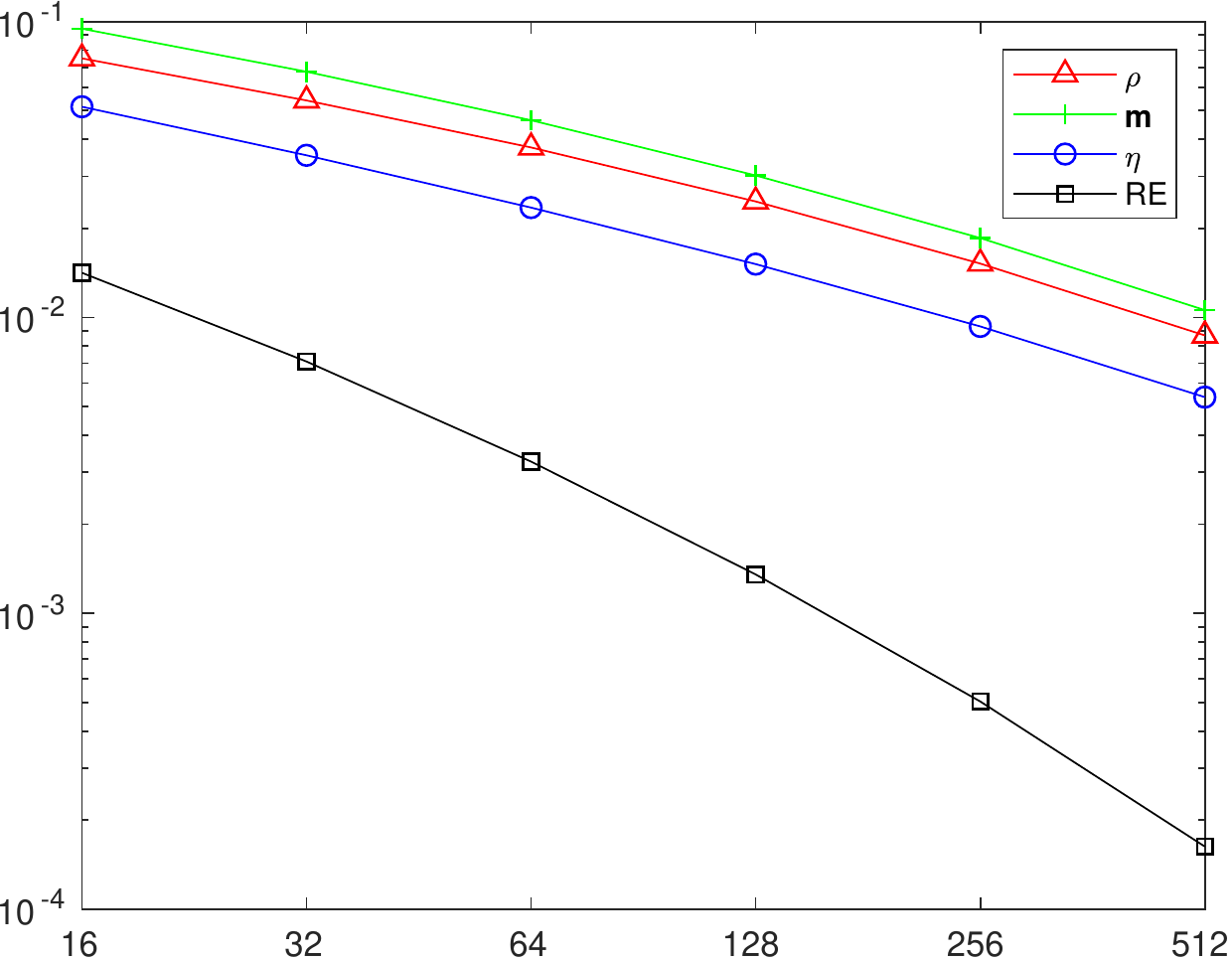}
		\caption{\bf VFV - error}
	\end{subfigure}	
	\caption{\small{Example \ref{example:2D-RP-3}: density on a mesh with $1024^2$ cells and errors at $T=0.2$.}}\label{figure:2D-RP-3}
\end{figure}

\begin{table}[htbp]
	\centering
	\caption{Example \ref{example:2D-RP-3}: errors and convergence rates of $\vr, \vm,\eta,\E$ of the Godunov and VFV methods. } \label{table:2D-RP-3}
	\begin{tabular}{|c|cc|cc|cc|cc|}
		\hline
		\multirow{2}{*}{$n$} & \multicolumn{2}{c|}{ density } & \multicolumn{2}{c|}{ momentum } & \multicolumn{2}{c|}{ entropy } & \multicolumn{2}{c|}{ relative energy}  \\
		\cline{2-9}
		& error & order &  error & order & error & order &  error & order  \\
		\hline
		\hline
		\multicolumn{9}{|c|}{ \bf{ Godunov}} \\
		\hline
		\hline
16 & 0.0572 & - & 0.0749 & - & 0.0365 & - & 0.007821 & - \\
32 & 0.0421 & 0.4408 & 0.0549 & 0.4475 & 0.0267 & 0.4482 & 0.004021 & 0.9597 \\
64 & 0.0298 & 0.4975 & 0.0390 & 0.4950 & 0.0192 & 0.4808 & 0.001952 & 1.0430 \\
128 & 0.0202 & 0.5636 & 0.0265 & 0.5567 & 0.0132 & 0.5354 & 0.000874 & 1.1594 \\
256 & 0.0129 & 0.6402 & 0.0171 & 0.6316 & 0.0087 & 0.6026 & 0.000354 & 1.3038 \\
512 & 0.0077 & 0.7434 & 0.0103 & 0.7353 & 0.0054 & 0.6973 & 0.000125 & 1.5033 \\
%1024 & 0.0041 & 0.9163 & 0.0055 & 0.9100 & 0.0030 & 0.8641 & 0.000035 & 1.8429 \\
%2048 & 0.0016 & 1.3257 & 0.0022 & 1.3225 & 0.0012 & 1.2715 & 0.000006 & 2.6567 \\
		\hline
		\hline
		\multicolumn{9}{|c|}{ \bf{ VFV}} \\
		\hline
		\hline
16 & 0.0751 & - & 0.0946 & - & 0.0515 & - & 0.014156 & -  \\
32 & 0.0541 & 0.4729 & 0.0677 & 0.4823 & 0.0353 & 0.5451 & 0.007097 & 0.9962 \\
64 & 0.0375 & 0.5276 & 0.0464 & 0.5454 & 0.0235 & 0.5868 & 0.003257 & 1.1237 \\
128 & 0.0247 & 0.6061 & 0.0302 & 0.6195 & 0.0151 & 0.6347 & 0.001354 & 1.2666 \\
256 & 0.0152 & 0.6976 & 0.0186 & 0.7026 & 0.0093 & 0.6997 & 0.000504 & 1.4263 \\
512 & 0.0087 & 0.8063 & 0.0106 & 0.8093 & 0.0054 & 0.7938 & 0.000163 & 1.6287 \\
%1024 & 0.0044 & 0.9776 & 0.0054 & 0.9815 & 0.0028 & 0.9575 & 0.000042 & 1.9596 \\
%2048 & 0.0017 & 1.3825 & 0.0020 & 1.3875 & 0.0011 & 1.3614 & 0.000006 & 2.7622 \\
		\hline
	\end{tabular}
\end{table}

\begin{Example}\label{example:2D-RP-1}\rm	
The initial data of the second 2D Riemann problem are given by 
	\begin{equation*}%\label{eq:2DRiemann01}
	(\vr ,  u , v, p)(x,0)
	=  \begin{cases}
	(0.5 ,\, 0.5 ,\, -0.5 ,\, 5 ) , & x > 0.5,\,y>0.5, \\
	(1 ,\,  0.5,\, 0.5 , \, 5) , & x<0.5,\,y>0.5, \\
	(2 ,\,  -0.5,\, 0.5 , \, 5) , & x<0.5,\,y<0.5, \\
	(1.5 ,\,  -0.5,\, -0.5 , \, 5) , & x>0.5,\,y<0.5.
	\end{cases}
	\end{equation*}
The exact solution consists of four interacting contact discontinuities yielding vortex sheets with negative signs. 
We simulate till $T=0.2$. 
%	We compute the solutions up to $T = 0.2$. 
%	It describes the interaction of four contact discontinuities (vortex sheets) with the same sign (the negative sign). 	As time increases the four initial vortex sheets interact each other to form a spiral with the low density around the center of the domain as time increases, which is the typical cavitation phenomenon in gas dynamics. 	
Figure \ref{figure:2D-RP-1}(a) and (c)  show the density obtained by the Godunov method and the VFV method on a mesh with $1024^2$ cells. The $L^2$-errors of $(\vr, \vm, \eta)$ as well as the $L^1$-norm of $\E$ are shown in Figure~\ref{figure:2D-RP-1}(b) and (d), see also Table \ref{table:2D-RP-1}. 
	
Numerical results indicate that  $(\vr, \vm, \eta)$  converges with the convergence rate about $1/2$ and the convergence rate for $\E$ is approximately $1$. 
It seems that our theoretical results for the convergence rates obtained for the strong exact solutions practically holds also for some discontinuous (weak) solutions. 
\end{Example}

\begin{figure}[htbp]
	\setlength{\abovecaptionskip}{0.cm}
	\setlength{\belowcaptionskip}{-0.cm}
	\centering
	\begin{subfigure}{0.4\textwidth}
		\includegraphics[width=\textwidth]{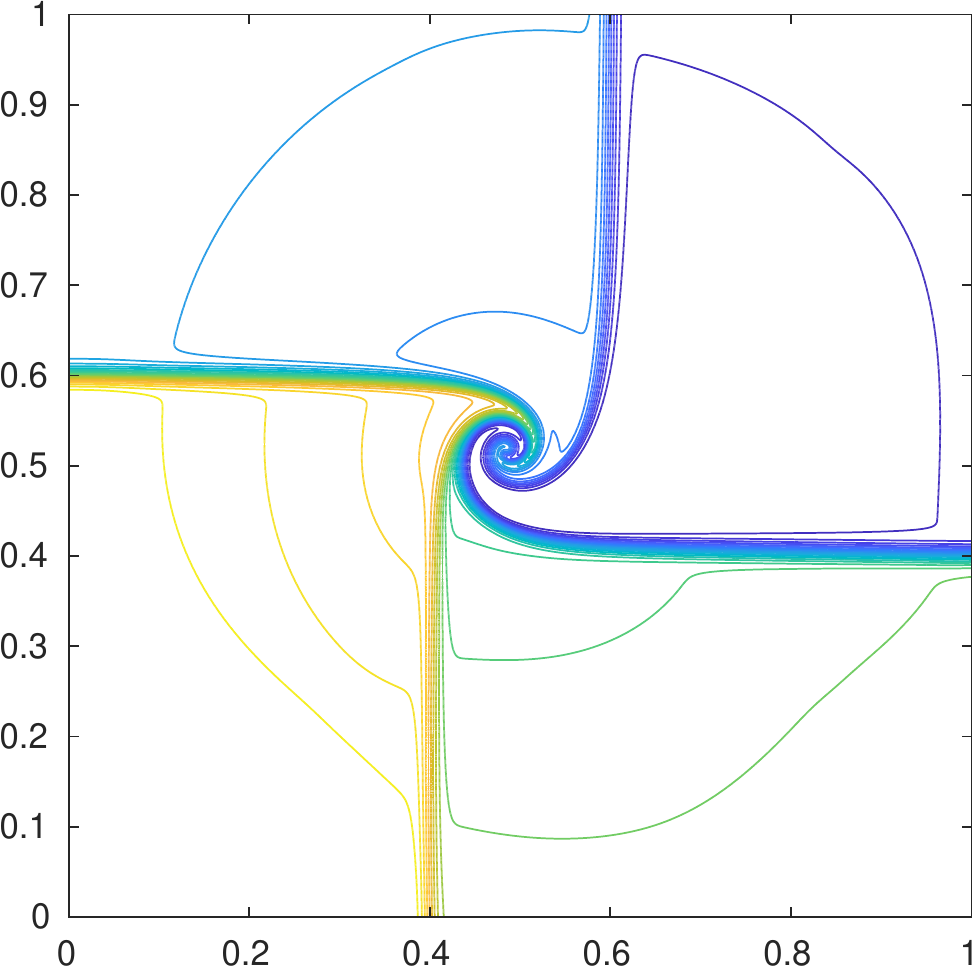}
		\caption{ \bf Godunov - $\vr$}
	\end{subfigure}	
	\begin{subfigure}{0.5\textwidth}
		\includegraphics[width=\textwidth]{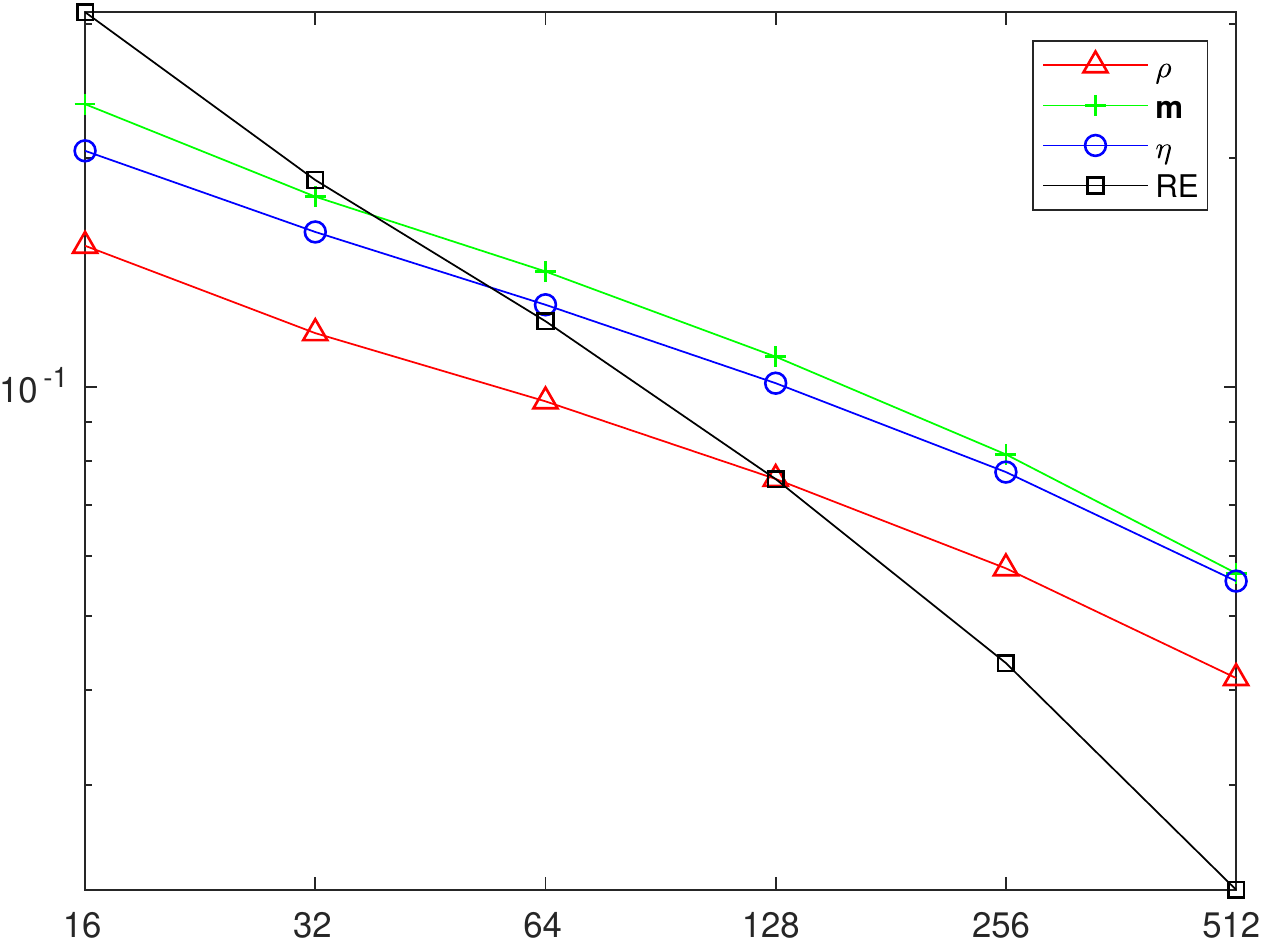}
		\caption{\bf Godunov - error}
	\end{subfigure}	\\
	\begin{subfigure}{0.4\textwidth}
		\includegraphics[width=\textwidth]{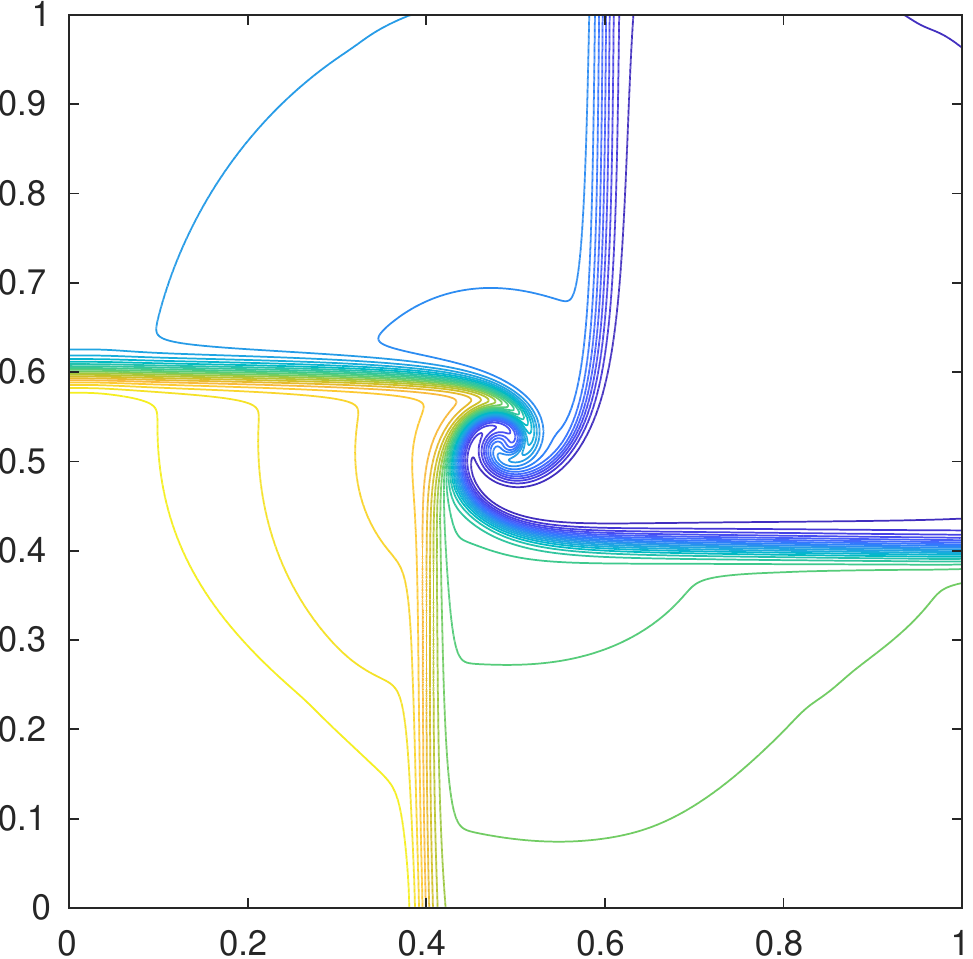}
		\caption{ \bf VFV - $\vr$}
	\end{subfigure}	
	\begin{subfigure}{0.5\textwidth}
		\includegraphics[width=\textwidth]{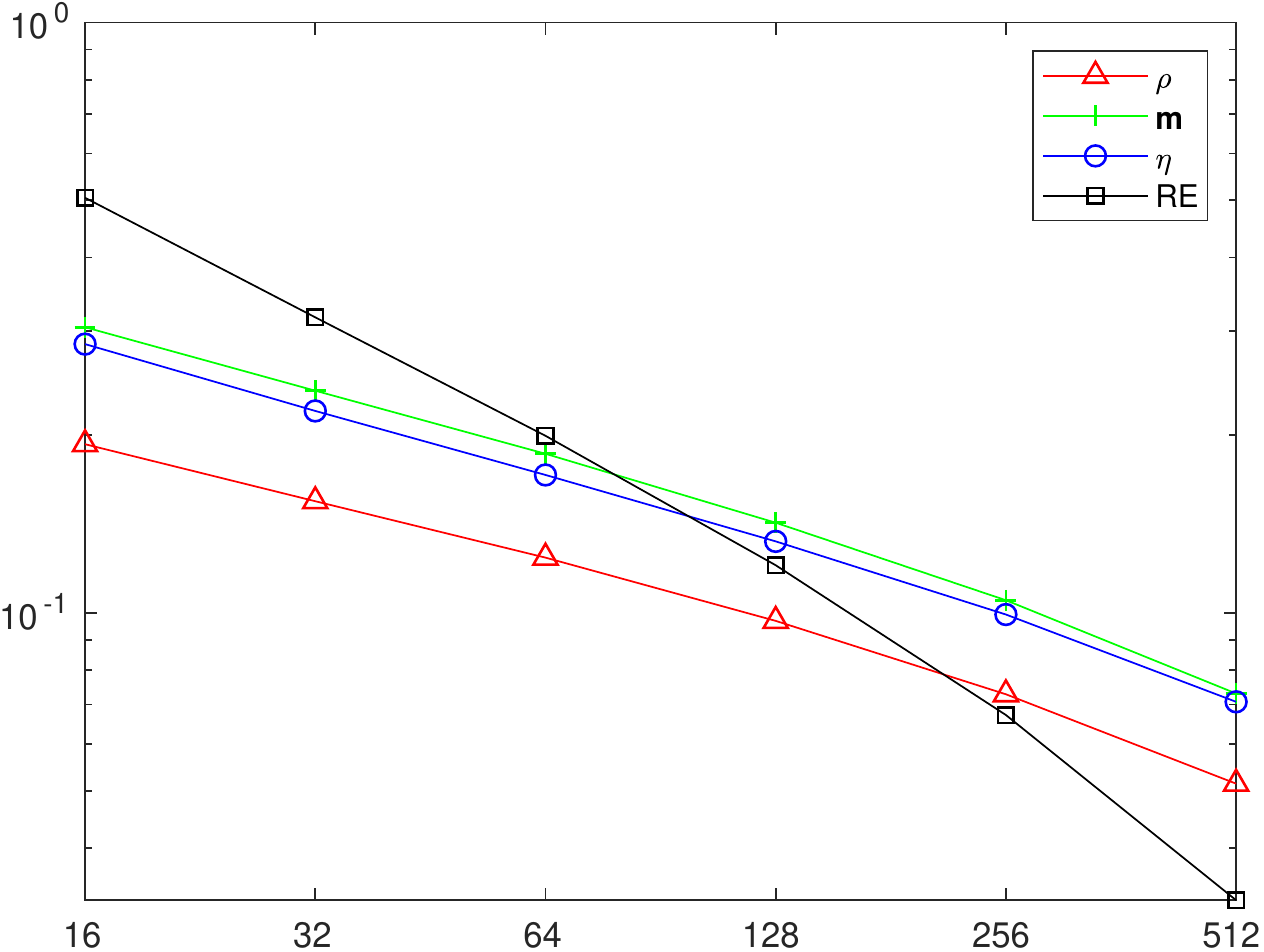}
		\caption{\bf VFV - error}
	\end{subfigure}	
	\caption{\small{Example \ref{example:2D-RP-1}: density on a mesh with $1024^2$ cells and errors obtained on different meshes.}}\label{figure:2D-RP-1}
\end{figure}

\begin{table}[htbp]
	\centering
	\caption{Example \ref{example:2D-RP-1}: errors and convergence rates of $\vr, \vm,\eta,\E$ of the Godunov and VFV methods.  } \label{table:2D-RP-1}
	\begin{tabular}{|c|cc|cc|cc|cc|}
		\hline
		\multirow{2}{*}{$n$} & \multicolumn{2}{c|}{ density } & \multicolumn{2}{c|}{ momentum } & \multicolumn{2}{c|}{ entropy } & \multicolumn{2}{c|}{ relative energy}  \\
		\cline{2-9}
		& error & order &  error & order & error & order &  error & order  \\
		\hline
		\hline
		\multicolumn{9}{|c|}{ \bf{ Godunov}} \\
		\hline
		\hline
16 & 0.1534 & - & 0.2355 & - & 0.2045 & - & 0.311123 & -   \\
32 & 0.1177 & 0.3816 & 0.1780 & 0.4040 & 0.1599 & 0.3543 & 0.187175 & 0.7331 \\
64 & 0.0958 & 0.2979 & 0.1419 & 0.3267 & 0.1283 & 0.3180 & 0.122058 & 0.6168 \\
128 & 0.0757 & 0.3390 & 0.1096 & 0.3724 & 0.1012 & 0.3425 & 0.075685 & 0.6895 \\
256 & 0.0578 & 0.3903 & 0.0816 & 0.4266 & 0.0773 & 0.3881 & 0.043366 & 0.8034 \\
512 & 0.0414 & 0.4792 & 0.0569 & 0.5190 & 0.0556 & 0.4761 & 0.021832 & 0.9901 \\
%1024 & 0.0266 & 0.6373 & 0.0354 & 0.6869 & 0.0356 & 0.6426 & 0.008913 & 1.2925 \\
%2048 & 0.0128 & 1.0532 & 0.0165 & 1.1005 & 0.0172 & 1.0530 & 0.001995 & 2.1598 \\
		\hline
		\hline
		\multicolumn{9}{|c|}{ \bf{ VFV}} \\
		\hline
		\hline
16 & 0.1932 & - & 0.3048 & - & 0.2854 & - & 0.505011 & -  \\
32 & 0.1547 & 0.3206 & 0.2380 & 0.3572 & 0.2199 & 0.3760 & 0.316830 & 0.6726 \\
64 & 0.1241 & 0.3173 & 0.1861 & 0.3548 & 0.1714 & 0.3601 & 0.199627 & 0.6664 \\
128 & 0.0970 & 0.3558 & 0.1422 & 0.3878 & 0.1323 & 0.3737 & 0.120510 & 0.7281 \\
256 & 0.0729 & 0.4129 & 0.1051 & 0.4360 & 0.0994 & 0.4115 & 0.067199 & 0.8426 \\
512 & 0.0514 & 0.5029 & 0.0731 & 0.5248 & 0.0708 & 0.4910 & 0.032644 & 1.0416 \\
%1024 & 0.0323 & 0.6727 & 0.0450 & 0.6980 & 0.0448 & 0.6596 & 0.012400 & 1.3965 \\
%2048 & 0.0151 & 1.0992 & 0.0207 & 1.1190 & 0.0211 & 1.0872 & 0.002603 & 2.2519 \\	
		\hline
	\end{tabular}
\end{table}

%----------------------------------
\begin{Example}\label{example:2D-RP-2}\rm
The initial data of third 2D Riemann problem are given by
\begin{equation*}
(\vr ,  u , v, p)(x,0)
\; = \; \begin{cases}
(1.5 ,\, 0 ,\, 0 ,\, 1.5 ) , & x > 0.5,\,y>0.5, \\
(0.5323 ,\,  1.206,\, 0 , \, 0.3) , & x<0.5,\,y>0.5,\\
(0.138 ,\,  1.206,\, 1.206, \, 	0.029) , & x<0.5,\,y<0.5,\\
(0.5323 ,\,  0,\, 1.206 , \, 0.3) , & x>0.5,\,y<0.5,
\end{cases}
\end{equation*}
which describes the interaction of four shock waves. 
In this example the final time is set to $T = 0.35$. 
%The output solutions is $T = 0.35$. 
%It describes the interaction of four shock waves. 
%As time increases it forms a complicated wave configuration at the bottom left region. 
Figure \ref{figure:2D-RP-2} shows the density on a mesh with $1024^2$ cells and errors of $(\vr, \vm, \eta)$ and $\E$ obtained on different meshes. Table \ref{table:2D-RP-2} lists the errors and convergence rate.  

From these numerical results we see that  $(\vr, \vm, \eta)$  converges with a ratio between $1/4$ and $1/2$ and $\E$ converges to a ratio between $1/2$ and $1$. 
\end{Example}

\begin{figure}[htbp]
	\setlength{\abovecaptionskip}{0.cm}
	\setlength{\belowcaptionskip}{-0.cm}
	\centering
	\begin{subfigure}{0.4\textwidth}
		\includegraphics[width=\textwidth]{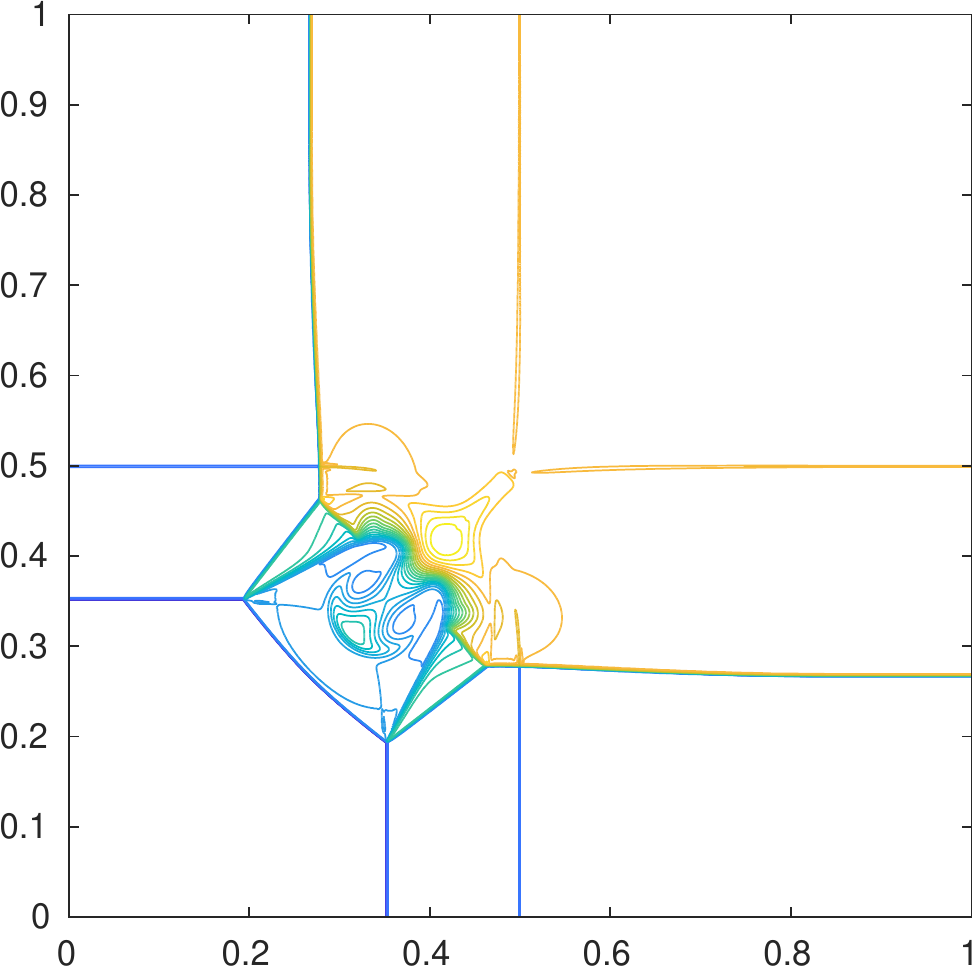}
		\caption{ \bf Godunov - $\vr$}
	\end{subfigure}	
	\begin{subfigure}{0.5\textwidth}
		\includegraphics[width=\textwidth]{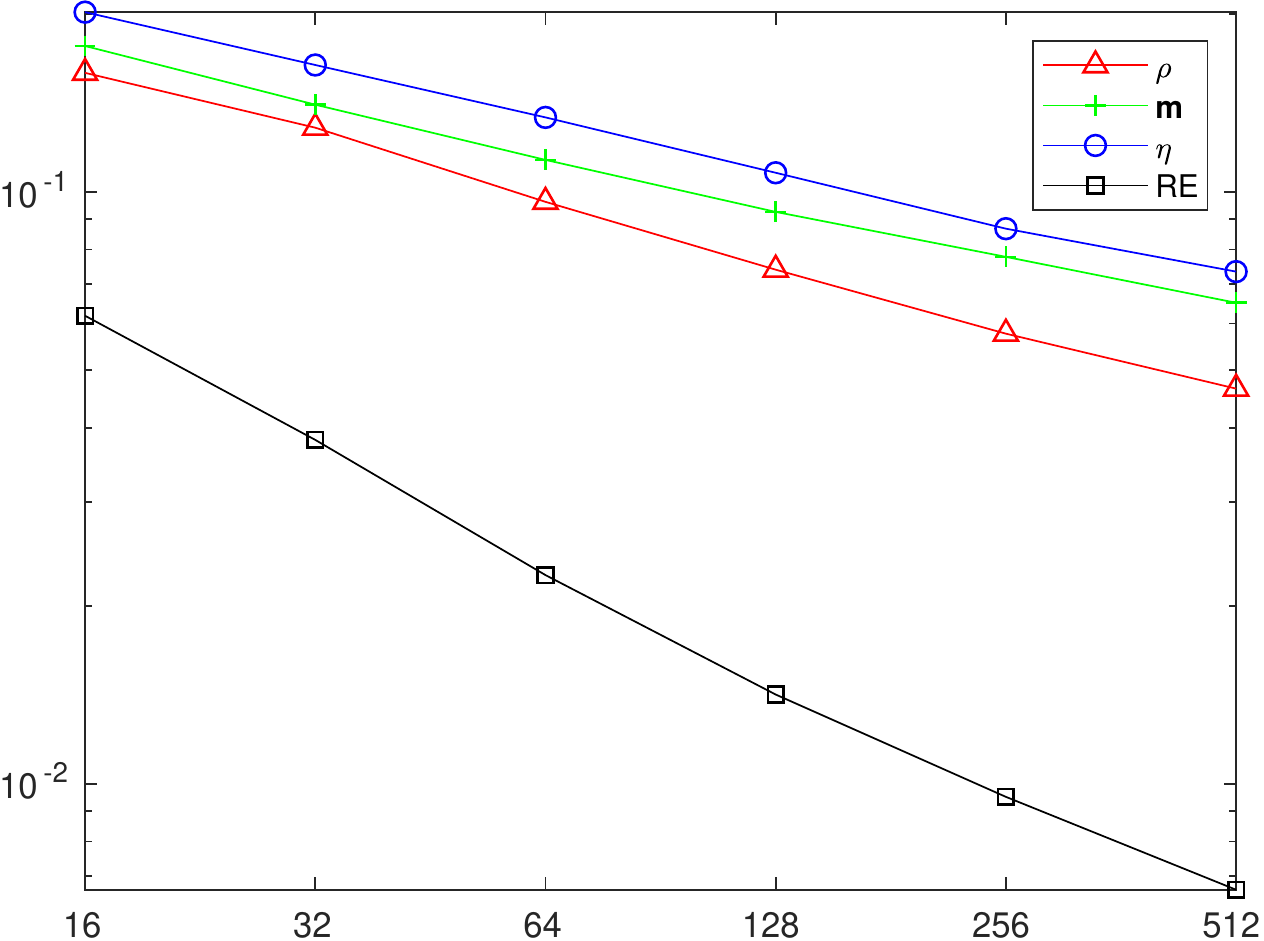}
		\caption{\bf Godunov - error}
	\end{subfigure}	\\
	\begin{subfigure}{0.4\textwidth}
		\includegraphics[width=\textwidth]{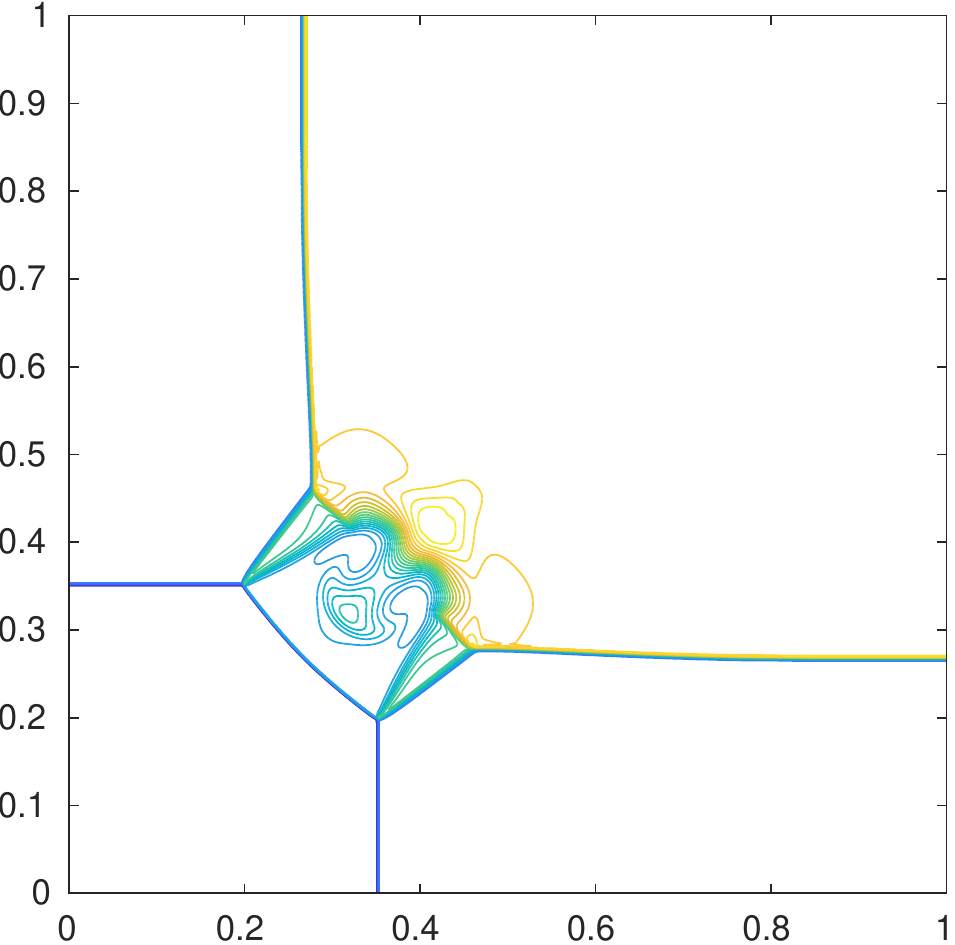}
		\caption{ \bf VFV - $\vr$}
	\end{subfigure}	
	\begin{subfigure}{0.5\textwidth}
		\includegraphics[width=\textwidth]{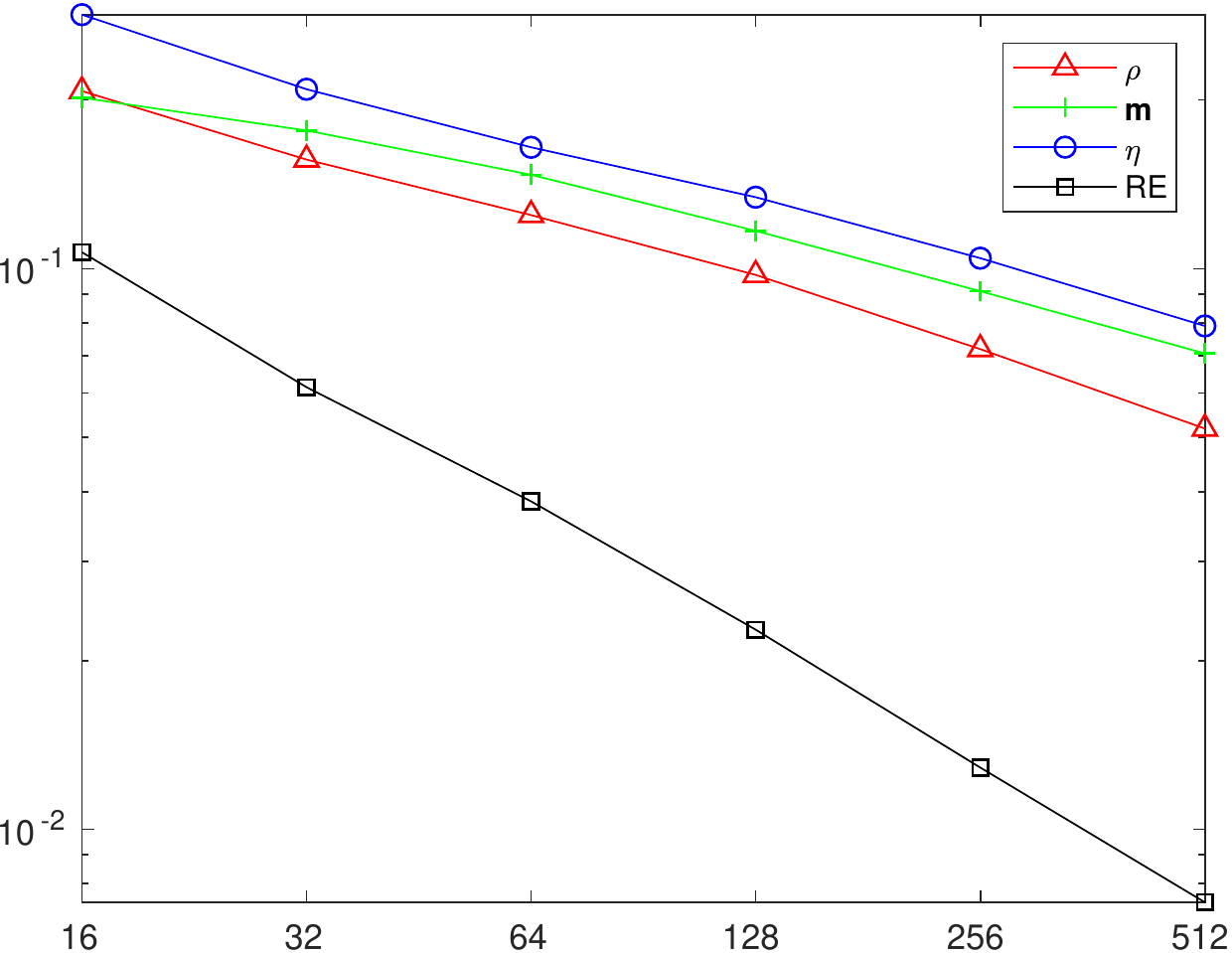}
		\caption{\bf VFV - error}
	\end{subfigure}	
	\caption{\small{Example \ref{example:2D-RP-2}: density on a mesh with $1024^2$ cells and errors obtained on different meshes.}}\label{figure:2D-RP-2}
\end{figure}

\begin{table}[htbp]
	\centering
	\caption{Example \ref{example:2D-RP-2}: errors and convergence rates of $\vr, \vm,\eta,\E$ of the Godunov and VFV methods.  } \label{table:2D-RP-2}
	\begin{tabular}{|c|cc|cc|cc|cc|}
		\hline
		\multirow{2}{*}{$n$} & \multicolumn{2}{c|}{ density } & \multicolumn{2}{c|}{ momentum } & \multicolumn{2}{c|}{ entropy } & \multicolumn{2}{c|}{ relative energy}  \\
		\cline{2-9}
		& error & order &  error & order & error & order &  error & order  \\
		\hline
		\hline
		\multicolumn{9}{|c|}{ \bf{ Godunov}} \\
		\hline
		\hline
16 & 0.1589 & - & 0.1764 & - & 0.2013 & - & 0.061809 & -   \\
32 & 0.1284 & 0.3077 & 0.1404 & 0.3292 & 0.1639 & 0.2964 & 0.038141 & 0.6965 \\
64 & 0.0963 & 0.4160 & 0.1133 & 0.3098 & 0.1337 & 0.2940 & 0.022547 & 0.7584 \\
128 & 0.0739 & 0.3806 & 0.0926 & 0.2917 & 0.1078 & 0.3106 & 0.014170 & 0.6701 \\
256 & 0.0576 & 0.3590 & 0.0777 & 0.2523 & 0.0867 & 0.3137 & 0.009518 & 0.5740 \\
512 & 0.0466 & 0.3084 & 0.0650 & 0.2577 & 0.0734 & 0.2409 & 0.006632 & 0.5212\\
%1024 & 0.0377 & 0.3035 & 0.0500 & 0.3773 & 0.0619 & 0.2455 & 0.004207 & 0.6569 \\
%2048 & 0.0248 & 0.6064 & 0.0298 & 0.7460 & 0.0444 & 0.4777 & 0.001688 & 1.3174 \\
		\hline
		\hline
		\multicolumn{9}{|c|}{ \bf{ VFV}} \\
		\hline
		\hline
16 & 0.2075 & - & 0.2018 & - & 0.2840 & - & 0.107017 & -   \\
32 & 0.1566 & 0.4063 & 0.1765 & 0.1938 & 0.2090 & 0.4420 & 0.061465 & 0.8000 \\
64 & 0.1246 & 0.3290 & 0.1471 & 0.2626 & 0.1647 & 0.3441 & 0.038455 & 0.6766 \\
128 & 0.0975 & 0.3546 & 0.1168 & 0.3325 & 0.1342 & 0.2957 & 0.022700 & 0.7605 \\
256 & 0.0719 & 0.4397 & 0.0912 & 0.3576 & 0.1044 & 0.3614 & 0.012882 & 0.8173 \\
512 & 0.0519 & 0.4707 & 0.0706 & 0.3689 & 0.0790 & 0.4018 & 0.007407 & 0.7985 \\
%1024 & 0.0392 & 0.4020 & 0.0530 & 0.4144 & 0.0608 & 0.3780 & 0.004353 & 0.7669 \\
%2048 & 0.0255 & 0.6233 & 0.0320 & 0.7281 & 0.0414 & 0.5535 & 0.001779 & 1.2906 \\	
		\hline
	\end{tabular}
\end{table}

\begin{Example}\label{example:2D-RP-4}\rm
The initial data of the fourth 2D Riemann problem are given by 
\begin{equation*}
(\vr ,  u , v, p)(x,0)
\; = \; \begin{cases}
(0.5313 ,\, 0 ,\, 0 ,\, 0.4 ) , & x > 0.5,\,y>0.5, \\
(1 ,\,  0.7276,\, 0 , \, 1) , & x<0.5,\,y>0.5, \\
(0.8 ,\,  0,\, 0, \, 1) , & x<0.5,\,y<0.5, \\
(1 ,\,  0,\, 0.7276 , \, 1) , & x>0.5,\,y<0.5.
\end{cases}
\end{equation*}
This experiment describes the interaction of four discontinuities (the left and bottom discontinuities are two contact discontinuities and the top and right are two shock waves). 
The final time is set to $T=0.25$.
%The output solutions is $T = 0.25$. 
%It describes the interaction of four discontinuities, where the left and bottom discontinuities are two contact discontinuities and the top and right are two shock waves with the speed of $0.9345632754$.
%As time increase four initial discontinuities interact each other and form a ``mushroom cloud'' around the point $(0.5, 0.5)$ as $t$ increases.
Figure \ref{figure:2D-RP-4} shows the density obtained by the Godunov and VFV methods on a mesh with $1024^2$ cells, respectively. 
The $L^2$-errors of $\vr, \vm, \eta$, and the $L^1$-norm of  $\E$ obtained on different meshes are presented in  Figure \ref{figure:2D-RP-4} and Table \ref{table:2D-RP-4}.  

These numerical results indicate the convergence rate around $1/2$ for the $L^2$-errors in $(\vr, \vm, \eta)$ and rates around $1$ for the $L^1$-norm in the relative energy $\E$. 
Similarly as in the previous experiments, it seems that the VFV method converges faster than the Godunov method.
\end{Example}

\begin{figure}[htbp]
	\setlength{\abovecaptionskip}{0.cm}
	\setlength{\belowcaptionskip}{-0.cm}
	\centering
	\begin{subfigure}{0.4\textwidth}
		\includegraphics[width=\textwidth]{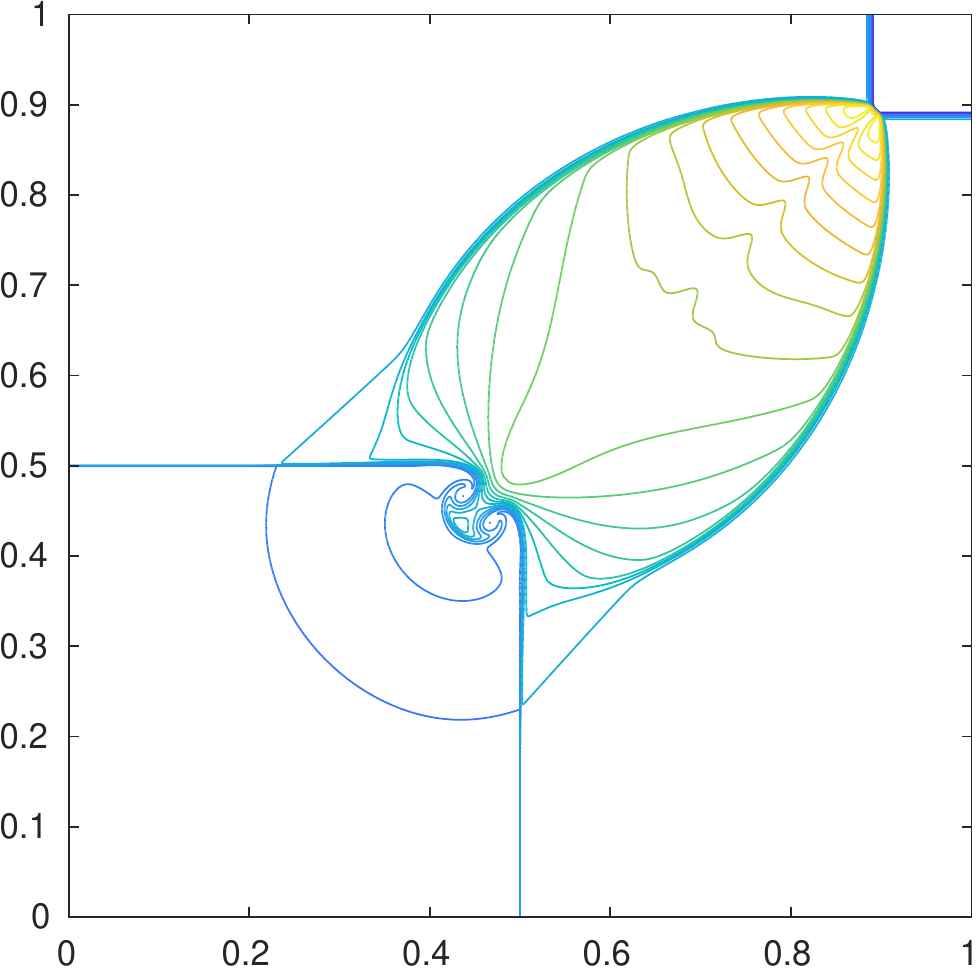}
		\caption{ \bf Godunov - $\vr$}
	\end{subfigure}	
	\begin{subfigure}{0.5\textwidth}
		\includegraphics[width=\textwidth]{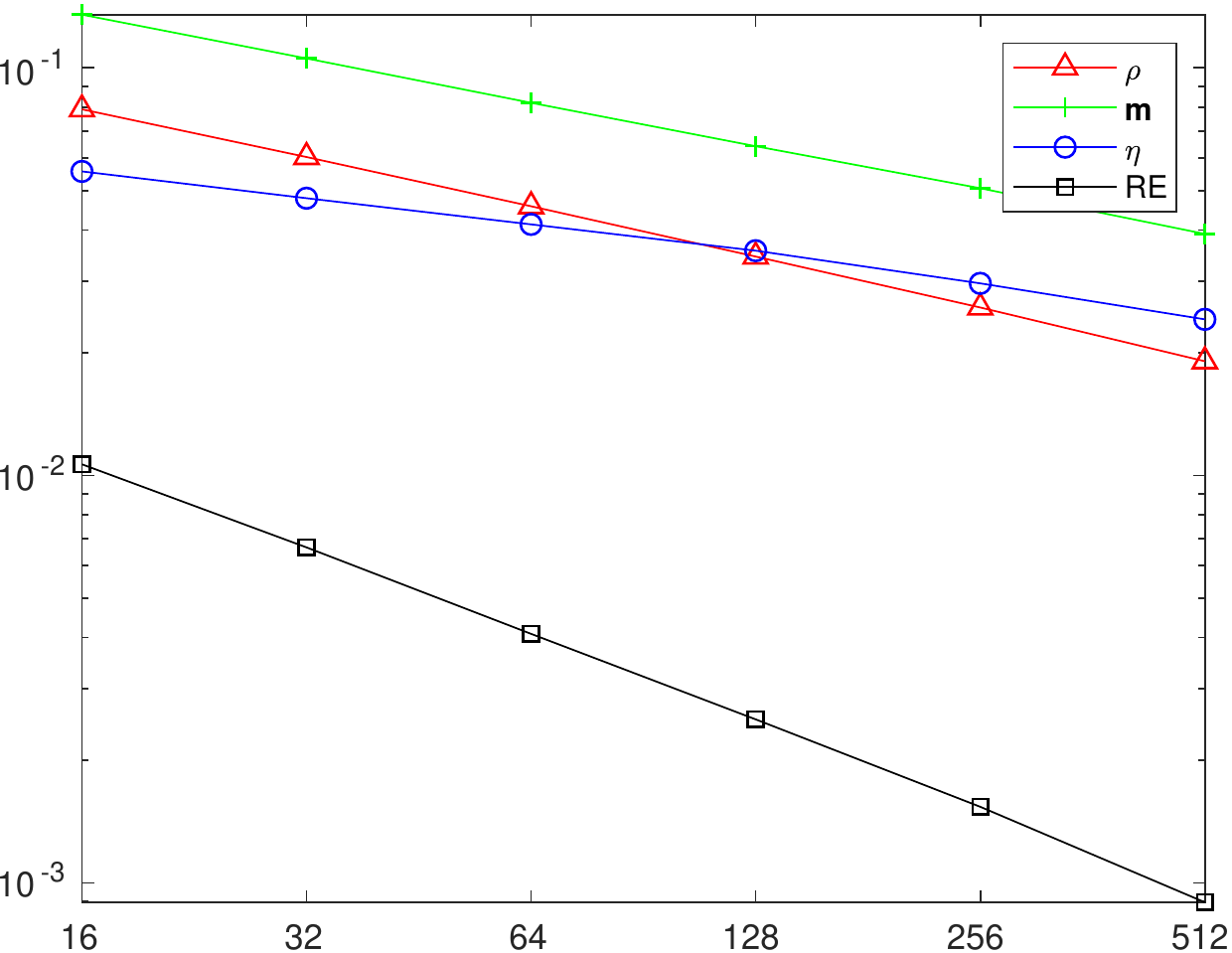}
		\caption{\bf Godunov - error}
	\end{subfigure}	\\
	\begin{subfigure}{0.4\textwidth}
		\includegraphics[width=\textwidth]{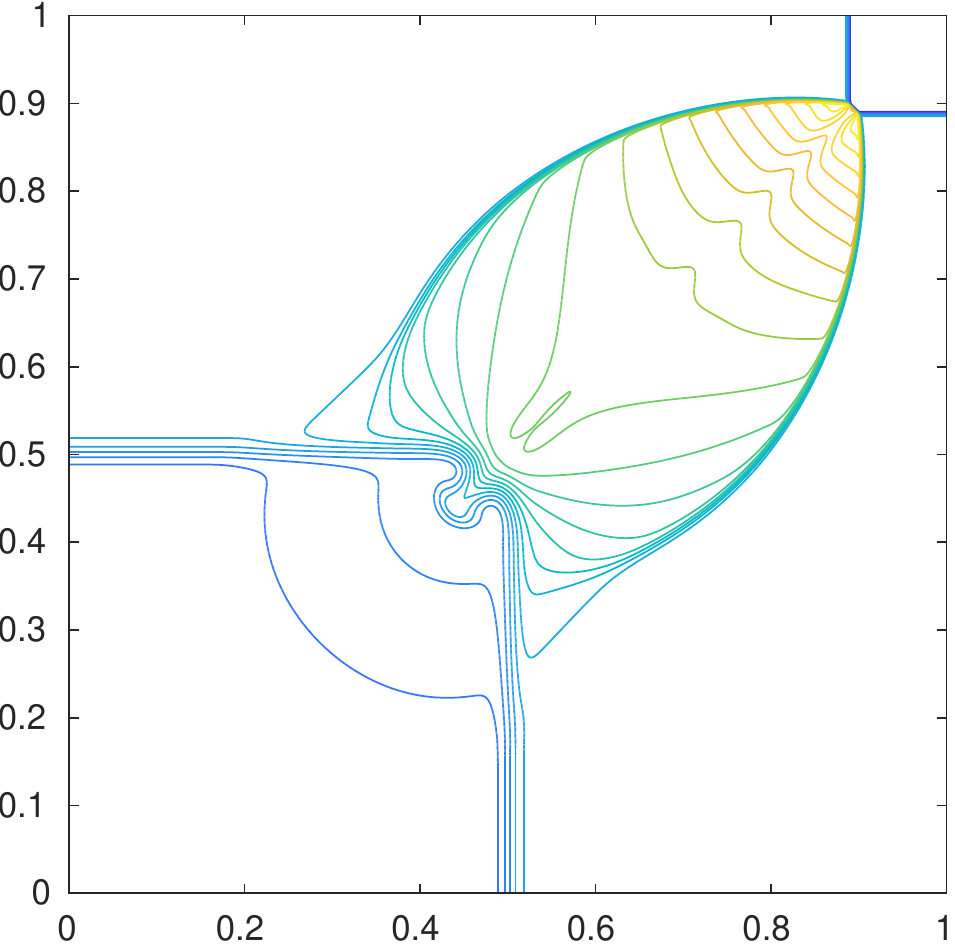}
		\caption{ \bf VFV - $\vr$}
	\end{subfigure}	
	\begin{subfigure}{0.5\textwidth}
		\includegraphics[width=\textwidth]{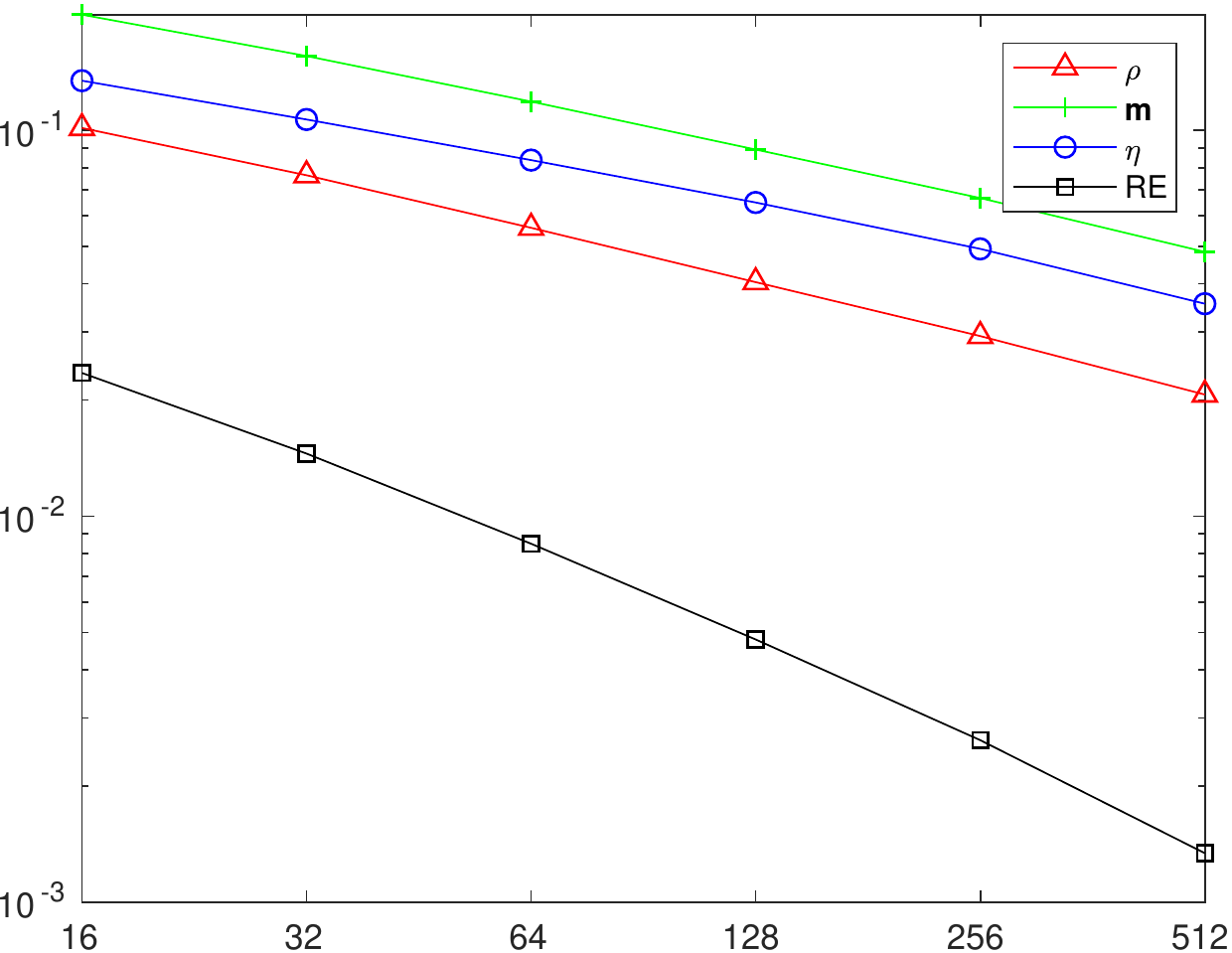}
		\caption{\bf VFV - error}
	\end{subfigure}	
	\caption{\small{Example \ref{example:2D-RP-4}: density on a mesh with $1024^2$ cells and errors obtained on different meshes.}}\label{figure:2D-RP-4}
\end{figure}

\begin{table}[htbp]
	\centering
	\caption{Example \ref{example:2D-RP-4}: errors and convergence rates of $\vr, \vm,\eta,\E$ of the Godunov and VFV methods. } \label{table:2D-RP-4}
	\begin{tabular}{|c|cc|cc|cc|cc|}
		\hline
		\multirow{2}{*}{$n$} & \multicolumn{2}{c|}{ density } & \multicolumn{2}{c|}{ momentum } & \multicolumn{2}{c|}{ entropy } & \multicolumn{2}{c|}{ relative energy}  \\
		\cline{2-9}
		& error & order &  error & order & error & order &  error & order  \\
		\hline
		\hline
		\multicolumn{9}{|c|}{ \bf{ Godunov}} \\
		\hline
		\hline
16 & 0.0791 & - & 0.1351 & - & 0.0557 & - & 0.010648 & -  \\
32 & 0.0604 & 0.3891 & 0.1055 & 0.3567 & 0.0479 & 0.2184 & 0.006658 & 0.6775 \\
64 & 0.0458 & 0.4012 & 0.0821 & 0.3619 & 0.0413 & 0.2134 & 0.004084 & 0.7050 \\
128 & 0.0344 & 0.4103 & 0.0643 & 0.3538 & 0.0356 & 0.2145 & 0.002519 & 0.6971\\
256 & 0.0258 & 0.4152 & 0.0507 & 0.3426 & 0.0296 & 0.2664 & 0.001537 & 0.7128 \\
512 & 0.0191 & 0.4382 & 0.0391 & 0.3724 & 0.0242 & 0.2932 & 0.000896 & 0.7786 \\
%1024 & 0.0130 & 0.5510 & 0.0271 & 0.5324 & 0.0192 & 0.3343 & 0.000440 & 1.0273\\
%2048 & 0.0070 & 0.8874 & 0.0141 & 0.9418 & 0.0126 & 0.6026 & 0.000129 & 1.7653 \\
		\hline
		\hline
		\multicolumn{9}{|c|}{ \bf{ VFV}} \\
		\hline
		\hline
16 & 0.1013 & - & 0.1992 & - & 0.1343 & - & 0.023507 & -   \\
32 & 0.0764 & 0.4069 & 0.1556 & 0.3559 & 0.1066 & 0.3340 & 0.014532 & 0.6938 \\
64 & 0.0559 & 0.4522 & 0.1186 & 0.3921 & 0.0837 & 0.3493 & 0.008488 & 0.7757 \\
128 & 0.0404 & 0.4676 & 0.0891 & 0.4126 & 0.0650 & 0.3647 & 0.004795 & 0.8240 \\
256 & 0.0293 & 0.4640 & 0.0667 & 0.4174 & 0.0493 & 0.3992 & 0.002630 & 0.8664 \\
512 & 0.0207 & 0.5035 & 0.0484 & 0.4632 & 0.0355 & 0.4719 & 0.001340 & 0.9725 \\
%1024 & 0.0135 & 0.6164 & 0.0321 & 0.5922 & 0.0234 & 0.5997 & 0.000578 & 1.2142\\
%2048 & 0.0072 & 0.9116 & 0.0166 & 0.9506 & 0.0126 & 0.8981 & 0.000154 & 1.9084 \\	
		\hline
	\end{tabular}
\end{table}

\newpage
\section{Conclusion}

In this paper we have analyzed a priori errors between numerical solutions obtained by the Godunov method and the strong exact solution for the multidimensional Euler system via the relative energy.  
Assuming that there exist a uniform lower bound on the density and an upper bound on the energy, we showed that the $L^1$-norm of the relative energy  is equivalent to the $L^2$-norm of errors of the numerical solutions, see \eqref{REL2-1}.  
Recalling the consistency formulation proved in \cite{LMY} and applying  Gronwall's lemma,    
we have derived the estimates for the relative energy in Theorem~\ref{Th2}.  
Specifically, the relative energy converges {\it at least} at the rate of $1/2$ in the $L^1$-norm. 
At the same time, the density, momentum and entropy converge {\it at least} at the rate of $1/4$ in the $L^2$-norm. 
Being inspired by the fact that the Godunov method for scalar conservation laws has bounded total variations we have formulated additional hypothesis \eqref{TVB1}. 
If we assume that \eqref{TVB1} holds, the convergence rate of density, momentum and entropy (resp. relative energy) can be improved to at least $1/2$ (resp. $1$), see Theorem \ref{Th3}. 
%We recall that {\cblue for scalar multidimensional conservation laws} it indeed holds that the Godunov method has a bounded total variation and the numerical solutions converge strongly to a strong solution, {\cblue reference}.
Finally, we pointed out that our theoretical analysis rigorously holds only for strong solutions, e.g. for a solution that contains finitely many rarefaction waves.

We have experimentally computed convergence rates for several one- and two-dimensional Riemann problems. 
From Example \ref{example:1D-singlewave} and  Example \ref{example:1D-RP-2} containing only rarefaction waves, we observed that the convergence rate of density, momentum and entropy (resp. relative energy) is slightly higher than $1/2$ (resp.~$1$), which is consistent with the theoretical results presented in Theorem~\ref{Th3}. 
Our numerical experiments for the Riemann problems with discontinuous solutions show that the convergence rate of the Godunov method are about $1/4$ for the contact wave and about $1/2$ for the shock wave. 
In future it will be interesting to analyze theoretically the convergence rates towards a weak exact solution containing shock and contact wave.

\section*{Funding}
\noindent M.L. has been funded by the Deutsche Forschungsgemeinschaft (DFG, German Research Foundation) - Project number 233630050 - TRR 146 as well as by  TRR 165 Waves to Weather. She is grateful to the Gutenberg Research College for supporting her research. 

	The research of B.S. leading to these results has received funding from the Czech Sciences Foundation (GA\v CR), Grant Agreement 21-02411S. The Institute of Mathematics of the Academy of Sciences of the Czech Republic is supported by RVO:67985840.

	The research of Y. Y. was funded by Sino-German (CSC-DAAD) Postdoc Scholarship Program in 2020 - Project number 57531629.

\section*{Availability of data and materials}
The datasets supporting the conclusions of this article are included within the article.


\begin{thebibliography}{30}
\bibitem{Brezina-Feireisl:2018a}
J.~B\v{r}ezina and E.~Feireisl.
\newblock Measure-valued solutions to the complete {Euler} system.
\newblock {\em J. Math. Soc. Japan}, 70(4):1227 - 1245, 2018.


\bibitem{Cockburn-Coquel-LeFloch:1994}
B.~Cockburn, F.~Coquel, and P.~G. LeFloch.
\newblock An error estimate for finite volume methods for multidimensional
  conservation laws.
\newblock {\em Math. Comp.}, 63(207):77-103, 1994.

\bibitem{Dafer}
C.~M. Dafermos.
\newblock The second law of thermodynamics and stability.
\newblock {\em Arch. Ration. Mech. Anal.}, 70(2):167-179, 1979.

\bibitem{FHMN}
E.~Feireisl, R.~Ho{\v{s}}ek, D.~Maltese, and A.~Novotn{\`y}.
\newblock Unconditional convergence and error estimates for bounded numerical
  solutions of the barotropic {Navier-Stokes} system.
\newblock {\em Numer. Methods Partial Differential Equations},
  33(4):1208-1223, 2017. 
  
  
\bibitem{FLM}
E.~Feireisl, M.~Luk{\'a}{\v c}ov{\'a}-Medvi{\softd}ov{\'a}, and H.~Mizerov{\'a}.
\newblock {A finite volume scheme for the {Euler} system inspired by the two
  velocities approach}.
\newblock {\em Numer. Math.}, 144(1):89-132, 2020.


\bibitem{Feireisl-Lukacova-Mizerova:2020a}
	E.~Feireisl, M.~Luk\'{a}\v{c}ov\'{a}-Medvi{\softd}ov\'{a}, and H.~Mizerov{\'a}.
	\newblock {Convergence of finite volume schemes for the {Euler} Equations via
		dissipative measure-valued solutions}.
	\newblock {\em Found. Comput. Math.}, 20(4):923-966, 2020.

\bibitem{FLMS}
E.~Feireisl, M.~Luk{\'a}{\v c}ov{\'a}-Medvi{\softd}ov{\'a}, {H}. Mizerov{\'a}  and B.~She.
\newblock {\em Numerical analysis of compressible fluid flows}.
\newblock Volume 20 of MS\&A series,  Springer-Verlag, 2021. 


\bibitem{Feireisl-Lukacova-Necasova-Novotny-She:2018}
E.~Feireisl, M.~Luk\'{a}\v{c}ov\'{a}-Medvi{\softd}ov\'{a}, \v{S}.
  Ne\v{c}asov\'{a}, A.~Novotn\'{y}, and B.~She.
\newblock Asymptotic preserving error estimates for numerical solutions of
  compressible Navier-Stokes equations in the low Mach number regime.
\newblock {\em Multiscale Modeling \& Simulation}, 16(1):150-183, 2018.

\bibitem{FeNo}
E. Feireisl and A. Novotn\'y.
\newblock {\em Singular limits in thermodynamics of viscous fluids, Second
  edition}.
\newblock {\em Birkh\"auser/Springer, Cham}, 2017.
 

\bibitem{Feistauer-Felcman-Straskraba:2003}
M.~Feistauer, J.~Felcman, and I.~Stra{\v{s}}kraba.
\newblock {\em Mathematical and computational methods for compressible flow}.
\newblock Oxford University Press, 2003.

\bibitem{Godunov}
S.~K. Godunov.
\newblock A difference method for numerical calculation of discontinuous
  solutions of the equations of hydrodynamics.
\newblock {\em Mat. Sb. (N.S.)}, 47(89):271-306, 1959.

\bibitem{JoRo}
V.~Jovanovi\'{c} and Ch. Rohde.
\newblock Error estimates for finite volume approximations of classical
  solutions for nonlinear systems of hyperbolic balance laws.
\newblock {\em SIAM J. Numer. Anal.}, 43(6):2423-2449, 2006.



\bibitem{Kuznetsov:1976}
N.~N. Kuznetsov.
\newblock Accuracy of some approximate methods for computing the weak solutions
  of a first-order quasi-linear equation.
\newblock {\em USSR Comput. Math. Math. Phys.}, 16:105-119, 1976.


\bibitem{Leveque:1992}
R.~J. LeVeque.
\newblock {\em Numerical methods for conservation laws}, volume 132.
\newblock Springer, 1992.

\bibitem{Li-Zhang-Yang:1998}
J.~Li, T.~Zhang, and S.~Yang.
\newblock {\em The two-dimensional Riemann problem in gas dynamics}, volume~98.
\newblock CRC Press, 1998.

\bibitem{LMY}
M.~Luk{\'a}{\v c}ov{\'a}-Medvi{\softd}ov{\'a} and Y.~Yuan.
\newblock Convergence of first-order finite volume method based on exact
  {Riemann} solver for the complete compressible {Euler} equations.
\newblock arXiv:2105.02165, 2021.

\bibitem{HS_MAC}
H.~Mizerov{\'a} and B.~She.
\newblock Convergence and error estimates for a finite difference scheme for
  the multi-dimensional compressible {Navier-Stokes} system.
\newblock {\em J. Sci. Comput.}, 84(1):25, 2020.



\bibitem{Shu_Osher}
C.-W. Shu and S.~Osher.
\newblock Efficient implementation of essentially non-oscillatory
  shock-capturing schemes.
\newblock {\em J. Comput. Phys.}, 77(2):439-471, 1988.



\bibitem{Tadmor-Tang:1999}
E.~Tadmor and T.~Tang.
\newblock Pointwise error estimates for scalar conservation laws with piecewise
  smooth solutions.
\newblock {\em SIAM J. Numer. Anal.}, 36(6):1739-1758, 1999.

\bibitem{Tang-Teng:1995}
T.~Tang and Z.-H. Teng.
\newblock The sharpness of {Kuznetsov's} $\mathcal{O}(\sqrt{\Delta x}) \,
  l^1$-error estimate for monotone difference schemes.
\newblock {\em Math. Comp.}, 64(210):581-589, 1995.

\bibitem{Tang-Teng:1997}
T.~Tang and Z.-H. Teng.
\newblock Viscosity methods for piecewise smooth solutions to scalar
  conservation laws.
\newblock {\em Math. Comp.}, 66(218):495-526, 1997.

\bibitem{Teng-Zhang:1997}
Z.-H. Teng and P.~Zhang.
\newblock Optimal $l^1$-rate of convergence for the viscosity method and
  monotone scheme to piecewise constant solutions with shocks.
\newblock {\em SIAM J. Numer. Anal.}, 34(3):959-978, 1997. 



\bibitem{Toro}
E.~F.~Toro.
\newblock Riemann solvers and numerical methods for fluid dynamics. A practical introduction. 
\newblock Third edition. Springer-Verlag, Berlin, 2009. xxiv+724 pp.


\bibitem{Vila:1994}
J.P. Vila.
\newblock Convergence and error estimates in finite volume schemes for general
  multidimensional scalar conservation laws. {I.} explicite monotone schemes.
\newblock {\em ESAIM: Math. Model. Numer. Anal.}, 28(3):267-295, 1994.

\end{thebibliography}
\end{document}